\pdfoutput=1
\RequirePackage{ifpdf}
\ifpdf 
\documentclass[pdftex]{sigma}
\else
\documentclass{sigma}
\fi

\usepackage{enumitem}

\numberwithin{equation}{section}

\newtheorem{Theorem}{Theorem}[section]
\newtheorem*{Theorem*}{Basic problem RH$\boldsymbol{{}_0}$}
\newtheorem{Corollary}[Theorem]{Corollary}
\newtheorem{Lemma}[Theorem]{Lemma}

 { \theoremstyle{definition}

\newtheorem{Remark}[Theorem]{Remark} }

\newcommand{\tr}{\mathop{\mathrm{tr}}\nolimits}
\newcommand{\sgn}{\mathop{\mathrm{sgn}}}
\newcommand{\ii}{\mathrm{i}}
\newcommand{\dd}{\textup{d}}
\newcommand{\la}{\lambda}
\renewcommand{\Im}{\operatorname{Im}}
\renewcommand{\Re}{\operatorname{Re}}
\newcommand{\R}{{\mathbb R}}
\newcommand{\mC}{{\mathbb C}}

\usepackage{pgf,tikz}
\usetikzlibrary{arrows.meta}

\begin{document}
\allowdisplaybreaks

\newcommand{\arXivNumber}{2212.04524}

\renewcommand{\PaperNumber}{096}

\FirstPageHeading

\ShortArticleName{Initial-Boundary Value Problem for the Maxwell--Bloch Equations}

\ArticleName{Initial-Boundary Value Problem\\ for the Maxwell--Bloch Equations with an Arbitrary\\ Inhomogeneous Broadening and Periodic Boundary\\ Function}

\Author{Maria FILIPKOVSKA~$^{\rm ab}$}

\AuthorNameForHeading{M.~Filipkovska}

\Address{$^{\rm a)}$~Friedrich-Alexander-Universit\"at Erlangen-N\"urnberg,\\
\hphantom{$^{\rm a)}$}~Cauerstra{\ss}e 11, 91058 Erlangen, Germany}
\EmailD{\href{mailto:maria.filipkovska@fau.de}{maria.filipkovska@fau.de}}

\Address{$^{\rm b)}$~B.~Verkin Institute for Low Temperature Physics and Engineering of NAS of Ukraine,\\
\hphantom{$^{\rm b)}$}~47 Nauky Ave., 61103 Kharkiv, Ukraine}
\EmailD{\href{mailto:filipkovskaya@ilt.kharkov.ua}{filipkovskaya@ilt.kharkov.ua}}

\ArticleDates{Received April 13, 2023, in final form November 17, 2023; Published online December 13, 2023}

\Abstract{The initial-boundary value problem (IBVP) for the Maxwell--Bloch equations with an arbitrary inhomogeneous broadening and periodic boundary condition is studied. This IBVP describes the propagation of an electromagnetic wave generated by periodic pumping in a resonant medium with distributed two-level atoms. We extended the inverse scattering transform method in the form of the matrix Riemann--Hilbert problem for solving the considered IBVP. Using the system of Ablowitz--Kaup--Newell--Segur equations equivalent to the system of the Maxwell--Bloch (MB) equations, we construct the associated matrix Riemann--Hilbert (RH) problem. Theorems on the existence, uniqueness and smoothness properties of a solution of the constructed RH problem are proved, and it is shown that a~solution of the considered IBVP is uniquely defined by the solution of the associated RH problem. It is proved that the RH problem provides the causality principle. The representation of a solution of the MB equations in terms of a solution of the associated RH problem are given. The significance of this method also lies in the fact that, having studied the asymptotic behavior of the constructed RH problem and equivalent ones, we can obtain formulas for the asymptotics of a solution of the corresponding IBVP for the MB equations.}


\Classification{35F31; 35Q15; 37K15; 34L25; 35Q60}

 \section{Introduction}\label{Intro}

Consider the Maxwell--Bloch (MB) equations written in the form
\begin{equation}\label{MB1}
\mathcal E_t+\mathcal E_x =\Omega\int_{-\infty}^\infty n(s)\rho(t,x,s)\, \dd s, \qquad
\rho_t+2\ii\la\rho=\mathcal N \mathcal E,\qquad
{\mathcal N}_t =-\frac{1}{2}\big(\overline{\mathcal E}\rho+ \mathcal E \overline{\rho}\big),
\end{equation}
where $x\in (0,L)$, $L\le\infty$, $t\in\R_+=(0,+\infty)$, $\la\in\R$, ${\mathcal E}={\mathcal E}(t,x)$ and $\rho=\rho(t,x,\la)$ are complex valued functions, ${\mathcal N}={\mathcal N}(t,x,\la)$ is a real function, subscripts mean the partial derivatives in $t$ and $x$, and $\overline{\phantom{Z}}$ denotes complex conjugation. Hermitian conjugation will be denoted by~${}^*$. It follows from the second and third equations of the system~\eqref{MB1} that $\frac{\partial}{\partial t}\big({\mathcal N}^2(t,x,\la)+|\rho(t,x,\la)|^2\big)=0$. This relation gives the well-known normalization condition
 \begin{equation}\label{rhoN}
{\mathcal N}^2(t,x,\la)+|\rho(t,x,\la)|^2\equiv 1.
 \end{equation}
The system of Maxwell--Bloch (MB) equations is a system of integrable nonlinear PDEs which can be solved using the inverse scattering transform (IST) method. The MB equations were originally proposed by Lamb~\cite{Lamb67,Lamb71}. Ablowitz, Kaup and Newell~\cite{AKN} have studied the coherent pulse propagation through a resonant two level optical medium and shown that the Maxwell--Bloch equations describing this phenomenon can be solved by the inverse scattering method. Zakharov introduced in the paper~\cite{Zakh} the concepts of spontaneous and causal solutions of the MB equations that initiates the systematic application of the IST method to laser problems. This work was further evolved by Gabitov, Zakharov and Mikhailov~\cite{GZM83,GZM84,GZM85} which have used the Marchenko integral equations and obtained a new version of IST method. They presented a description of general solutions of the MB equations and gave their classification. Generally, there are many pioneering works relating to the Maxwell--Bloch equations and only a small part of them are cited here.
In this paper, we extend the IST method in the form of the matrix Riemann--Hilbert (RH) problem for solving the initial-boundary value problem (IBVP) for the MB equations with a periodic boundary condition. In~\cite{Fokas97,Fokas02}, Fokas proposed the unified transform method for solving IBVPs for linear and for integrable nonlinear PDEs, the main idea of which is the simultaneous spectral analysis of both Lax operators (whose compatibility condition is provided by the satisfaction of corresponding nonlinear PDEs) to construct a unifying transform for solving certain IBVPs. We use the similar idea, namely, to construct the associated matrix RH problem, we use compatible solutions of the systems of Ablowitz--Kaup--Newell--Segur (AKNS) equations (in general, we consider the three systems), which are equivalent to the MB equations in the sense that their compatibility conditions are provided by the satisfaction of the corresponding MB equations.
Notice that the compatibility condition of the AKNS system is the equivalent of the Lax representation, i.e., the compatibility condition of the Lax pair. Then we obtain the representation of a solution of the MB equations in terms of a solution of the associated RH problem and prove the theorems on the existence, uniqueness and smoothness properties of a solution of the constructed RH problem.

The MB equations are used to describe different physical phenomena, including self-induced transparency~\cite{AKN,Lamb71}, superfluorescence~\cite{GZM83,GZM85} and others related to the problems on the propagation of an electromagnetic wave in a medium with distributed two-level atoms~\cite{AS,Lamb71,Manakov}. Short reviews of the physical meaning of the MB equations and the application of the inverse scattering transform method can be found in~\cite{AKN,AS,BGKL19,GZM85,Kiselev,LiMiller}, and in~\cite{HuangL/ChenY, WWG} for the reduced MB equations. We consider the problem on the propagation of an electromagnetic wave in a~resonant medium with distributed two-level atoms. In~\eqref{MB1}, ${\mathcal E}(t,x)$ is the complex envelope of an electric field, ${\mathcal N}={\mathcal N}(t,x,\la)$ and $\rho=\rho(t,x,\la)$ are the entries of the density matrix of a~quantum two-level atom subsystem: $\left(\begin{smallmatrix} {\mathcal N}(t,x,\la) & \rho(t,x,\la) \\
\overline{\rho(t,x,\la)} & -{\mathcal N}(t,x,\la) \end{smallmatrix}\right)$, the parameter $\la$~is the deviation of the transition frequency of the given two-level atom from the mean frequency $\Omega$, and $L$~is the length of the attenuator.
The weight function $n(\la)$ ($n(\la)\ge 0$, $\la\in\R$) normalized by the condition $\int_{-\infty}^\infty n(\la)\, \dd\la=1$ describes the inhomogeneous broadening of the medium (the shape of the spectral line).
Consider general initial and boundary conditions for the mixed problem for the MB equations:
 \begin{align}
&{\mathcal E}(0,x)= {\mathcal E}_0(x),\quad \rho(0,x,\la)= \rho_0(x,\la),\quad {\mathcal N}(0,x,\la)= {\mathcal N}_0(x,\la),\quad 0<x<L\le\infty,\! \label{ic} \\
&{\mathcal E}(t,0)= {\mathcal E}_{\rm in}(t),\qquad 0<t<\infty, \label{bc}
 \end{align}
where ${\mathcal E}_0(x)$, ${\mathcal N}_0(x,\la)$, $\rho_0(x,\la)$ and ${\mathcal E}_{\rm in}(t)$ are infinitely differentiable for $x\in (0,L)$, $\la\in\R$ and~$t\in \R_+$.
Taking into account~\eqref{rhoN}, a function ${\mathcal N}_0(x,\la)$ is defined by $\rho_0(x,\la)$,
\[
{\mathcal N}_0(x,\la)= -\sqrt{1-|\rho_0(x,\la)|^2},\qquad 0<x<L\le\infty.
\]
Here we choose the positive branch of the square root. This means that we consider the problem for an attenuator (self-induced transparency).

In what follows, we put $\Omega=1$ in~\eqref{MB1} and consider the IBVP for the Maxwell--Bloch equations%
\begin{equation}\label{MB1a}
{\mathcal E}_t+{\mathcal E}_x =\int_{-\infty}^\infty n(s)\rho(t,x,s)\, \dd s, \qquad
\rho_t+2\ii\la\rho={\mathcal N}{\mathcal E},\qquad
{\mathcal N}_t =-\frac{1}{2}\big(\overline{{\mathcal E}} \rho+{\mathcal E} \overline{\rho}\big),
\end{equation}
where $x\in (0,L)$, $L\le\infty$, $t\in\R_+=(0,\infty)$ and $\la\in\R$ is a parameter, with the trivial initial conditions
\begin{gather}
{\mathcal E}(0,x)={\mathcal E}_0(x)=0,\qquad \rho(0,x,\la)=\rho_0(x,\la)=0, \nonumber\\
{\mathcal N}(0,x,\la)={\mathcal N}_0(x,\la)=-1,\qquad x\in (0,L),\label{MBini}
\end{gather}
and the periodic boundary function (input signal)
\begin{equation}\label{MBboundAe}
\mathcal E(t,0)=\mathcal E_{\rm in}(t)= A_0{\rm e}^{\ii\omega_0 t},\qquad t\in\R_+,
\end{equation}
where $A_0, \omega_0>0$. Given the form of boundary and initial data, it is natural to assume that the function $\mathcal E(t,x)$ is bounded in the vicinity of $(0,0)$ (as well as its partial derivatives with respect to $t$ and $x$).

The following notation will be used in the paper: $[A,B]$ is the commutator of the matrices~$A$ and~$B$, i.e., $[A,B]=AB-BA$; $I$ denotes the identity operator (matrix); the symbol ``$\rm{p.v.}$'' denotes the principal value of the Cauchy integral. Also, we will use the Pauli matrices
\[
\sigma_2=\begin{pmatrix} 0 & -\ii \\ \ii&\hphantom{-}0 \end{pmatrix},\qquad
\sigma_3=\begin{pmatrix} 1&\hphantom{-}0\\0&-1 \end{pmatrix}.
\]

It is known that the MB equations~\eqref{MB1a} are equivalent to the overdetermined system of linear differential equations (see, e.g.,~\cite{AS, GZM85})
 \begin{gather}
W_t=U(t,x,\la)W, \label{teq}\\
W_x=V(t,x,\la)W, \label{xeq} \end{gather}
where
 \begin{align}
& U(t,x,\la)=-\ii\la\sigma_3-H(t,x),\qquad
 H(t,x)=\frac{1}{2} \begin{pmatrix} 0 & \mathcal E(t,x) \label{U-MB} \\
 -\overline{\mathcal E(t,x)} & 0 \end{pmatrix}, \\
& V(t,x,\la)=\ii\la\sigma_3+H(t,x)-\ii G(t,x,\la), \label{V-MB} \\
& G(t,x,\la)= \frac{1}{4} {\rm p.v.}\int_{-\infty}^\infty \frac{F(t,x,s)n(s)}{s-\la}\, \dd s,\qquad
 F(t,x,s)= \begin{pmatrix}
 {\mathcal N}(t,x,s) &\rho(t,x,s)\\
 \overline{\rho(t,x,s)} &-{\mathcal N}(t,x,s) \end{pmatrix}. \nonumber 
\end{align}
The system~\eqref{teq},~\eqref{xeq} (where $U$, $V$ of the form~\eqref{U-MB} and~\eqref{V-MB}) is commonly referred to as the \emph{system of the Ablowitz--Kaup--Newell--Segur (AKNS) equations for the MB equations}.

Since the equation~\eqref{xeq} cannot be used to construct an appropriate associated RH problem, we introduce the $x^+$-equation and $x^-$-equation (the upper and lower bank equations) in addition to $x$-equation~\eqref{xeq}:
\begin{equation}\label{pmxeq}
W_x=V_\pm(t,x,\la) W,
\end{equation}
where
 \begin{equation}\label{V_pm-MB}
V_\pm(t,x,\la)=V(t,x,\la)\pm\frac{\pi}{4}F(t,x,\la)n(\la)= \ii\la\sigma_3+H(t,x)-\ii G_\pm(t,x,\la),
 \end{equation}
and the functions
\begin{align}
 G_\pm(t,x,\la)& =\frac{1}{4} \int_{-\infty}^\infty \frac{F(t,x,s) n(s)}{s-\la\mp\ii0}\, \dd s \nonumber\\
 & =\frac{1}{4}{\rm p.v.}\int_{-\infty}^\infty \frac{F(t,x,s)n(s)}{s-\la}\, \dd s\pm\frac{\pi\ii}{4}F(t,x,\la)n(\la)\label{G_pm}
\end{align}
are the limits of the function
 \begin{equation}\label{continuation_G}
G(t,x,z)=\frac{1}{4}\int_{-\infty}^\infty \frac{F(t,x,s)n(s)\, \dd s}{s-z},\qquad z=\la+\ii\nu,
 \end{equation}
as $z\to\la\pm\ii0$ (i.e., $G_\pm(t,x,\la)= \lim_{\varepsilon\to0,\,\varepsilon>0}G(t,x,\la\pm\ii\varepsilon)$) and have the form~\eqref{G_pm} due to the Sokhotski--Plemelj formulas.

In what follows, systems of matrix differential equations of the form~\eqref{teq},~\eqref{xeq} and of the form~\eqref{teq},~\eqref{pmxeq} will be called the \emph{AKNS systems}.

A system of the form~\eqref{teq} and~\eqref{xeq} is \emph{compatible} if ${W_{tx}\equiv W_{xt}}$ for any its solution $W(t,x,\la)$, i.e., the following condition holds:
 \begin{equation}\label{ZCC}
U_x-V_t+[U,V]=0.
 \end{equation}
Thus, the \emph{compatibility condition of the AKNS system}~\eqref{teq},~\eqref{xeq} has the form~\eqref{ZCC}. The \emph{compatibility conditions of the AKNS systems}~\eqref{teq},~\eqref{pmxeq} take the form
 \begin{gather}
U_x(t,x,\la)-(V_{\pm})_t(t,x,\la)+[U(t,x,\la),V_\pm(t,x,\la)]= U_x(t,x,\la)-V_t(t,x,\la) \nonumber \\
\qquad{} +[U(t,x,\la),V(t,x,\la)]\mp\frac{\pi n(\la)}{4} (F_t(t,x,\la)-[U(t,x,\la),F(t,x,\la)])=0,\label{ZCC_pm}
 \end{gather}
where $(V_{\pm})_t:=\partial V_{\pm}/\partial t$.

The $t$-equation~\eqref{teq} and the $x$-equation~\eqref{xeq} are compatible if and only if ${\mathcal E}(t,x)$, $\rho(t,x,\la)$ and ${\mathcal N}(t,x,\la)$ satisfy the MB equations~\eqref{MB1a}~\cite{GZM85}, which can be written in the matrix form (see, e.g.,~\cite{Kotlyarov13}):
 \begin{align} 
& H_t(t,x)+H_x(t,x)- \frac{1}{4}\int_{-\infty}^\infty [\sigma_3,F(t,x,s)]n(s)\, \dd s =0, \label{HFeq_1} \\
& F_t(t,x,\la)+[\ii\la\sigma_3+H(t,x),F(t,x,\la)]=0. \label{HFeq_2}
 \end{align}
In particular, the equation~\eqref{HFeq_2} is a matrix form of the second and third equations from~\eqref{MB1a}.
It is easy to prove that for any fixed $t$, $x$ the matrix function $F(t,x,\la)=\pm W(t,x,\la)\sigma_3W^{-1}(t,x,\la)$ (or $F(t,x,\la)=\pm W(t,x,\la)\sigma_3W^*(t,x,\la)$ where ${}^*$ denotes a Hermitian conjugation), where $W(t,x,\la)$ is a solution of the $t$-equation~\eqref{teq} satisfying the initial condition $W(0,x,\la)\equiv I$, is a unique solution of the matrix linear differential equation~\eqref{HFeq_2} satisfying the initial condition $F(0,x,\la)\equiv \pm\sigma_3$.
The sign plus corresponds to a model of amplifier, and the sign minus to a model of attenuator which is the subject of this paper. In the absence of the electric field (${\mathcal E}(t,x)\equiv 0$) one has $F(t,x,\la)\equiv \pm\sigma_3$ that means an unstable/stable medium (amplifier/attenuator) under consideration.

As mentioned above, if ${\mathcal E}(t,x)$, $\rho(t,x,\la)$ and ${\mathcal N}(t,x,\la)$ satisfy the MB equations~\eqref{MB1a}, then \eqref{HFeq_2} (i.e., ${F_t-[U,F]=0}$) and \eqref{ZCC} hold and, therefore, the compatibility conditions~\eqref{ZCC_pm} are satisfied. Conversely, if the conditions~\eqref{ZCC_pm} are satisfied, then the MB equations in the matrix form~\eqref{HFeq_1} and~\eqref{HFeq_2}, which are equivalent to~\eqref{MB1a}, hold (see the proof of Theorem~\ref{solvMBbyM}). Thus, the $t$-equation~\eqref{teq} and the $x^\pm$-equation~\eqref{pmxeq} are compatible if and only if ${\mathcal E}(t,x)$, $\rho(t,x,\la)$ and ${\mathcal N}(t,x,\la)$ satisfy the MB equations~\eqref{MB1a}.

In~\cite{Filipkovska/Kotlyarov}, the IBVP~\eqref{MB1a}--\eqref{MBboundAe} has been studied in the case when $n(\la)$ is the Dirac delta-function, i.e., without inhomogeneous broadening (an unbroadened medium), and asymptotic formulas for its solution in different sectors of the light cone has been obtained, as well as the fulfillment of the causality principle has been shown.
The MB equations without inhomogeneous broadening has been also considered in~\cite{LiMiller}, but different initial-boundary conditions were used there and the equations were written in a comoving frame of reference. In the present paper as well as in~\cite{Filipkovska/Kotlyarov} the MB equations are considered in a laboratory frame. The work~\cite{LiMiller} provides a proper RH problem that generates the unique causal solution (i.e., the solution vanishes outside of the light cone) of the IBVPs for the MB equations.
An application of the IST method to the IBVP~\eqref{MB1},~\eqref{ic} and \eqref{bc} with a smooth and fast decreasing input signal ${\mathcal E}_{\rm in}(t)$ was given in~\cite{Kotlyarov13}, using the simultaneous spectral analysis of equations~\eqref{teq} and~\eqref{pmxeq} and the RH problem. The present paper shows that the IST method in the form of matrix Riemann--Hilbert problem can be used for the case of periodic boundary condition and give an integral representation for ${\mathcal E}(t,x)$, ${\mathcal N}(t,x,\la)$ and $\rho(t,x,\la)$ through the solution of a singular integral equation. Also, it shows that the RH problem allows to study the long-time asymptotic behavior of a solution of the mixed problem for the MB equations, choosing as an example the problem when ${\mathcal E}(t,0)= A_0 {\rm e}^{\ii\omega_0 t}$, which is interesting itself. Later, having studied the asymptotic behavior of the constructed RH problem and equivalent ones, we will be able to obtain formulas for the asymptotics of a solution of the considered IBVP.

\emph{The paper has the following structure}. In Section~\ref{Intro}, the problem statement and its physical interpretation are given, as well as the AKNS systems for the MB equations and their compatibility conditions are considered. In Section~\ref{BasicSolAKNS}, basic (compatible) solutions of the AKNS systems equivalent to the system of the MB equations are constructed and their properties are stated. These solutions will be used to construct the matrix Riemann--Hilbert problem associated with the IBVP~\eqref{MB1a}--\eqref{MBboundAe} which is given in Section~\ref{RHP0}. Then, in Section~\ref{Solv_IBVP_RHP}, we obtain the theorems on the existence, uniqueness and smoothness of a solution of the associated RH problem and prove that this RH problem generates a solution of the IBVP~\eqref{MB1a}--\eqref{MBboundAe}. We give the representation of the solution of the considered IBVP in terms of the solution of the associated RH problem, and also the representation of the electric field envelope in terms of a solution of an integral equation and an equivalent associated RH problem. Also, we prove that the constructed RH problem provides the causality principle. Thus, the problem on the propagation of an electromagnetic wave generated by periodic input signal in a stable medium with distributed two-level atoms (attenuator) is solved. The presented results will be used later to obtain the asymptotics of a~solution of the considered IBVP. In Appendix~\ref{appendixA}, the analysis of the phase function is carried out, the results of which are used in proving the theorems.

 \section{Basic solutions of the systems of the AKNS equations}\label{BasicSolAKNS}

Let us suppose that a solution of the MB equations~\eqref{MB1a} exists, then the AKNS systems~\eqref{teq}, \eqref{pmxeq} (as well as the AKNS system~\eqref{teq},~\eqref{xeq}) are compatible and we can define their ``basic'' solutions with the properties enabling to obtain the matrix RH problem which generates a solution of the IBVP~\eqref{MB1a}--\eqref{MBboundAe}. The upper and lower bank equations~\eqref{pmxeq} allow us to obtain solutions (which we call basic and denote $Y_\pm$, $Z_\pm$; see Section~\ref{PropertyBasicSol}) that have analytic continuations to the upper and lower complex half-planes $\mC_\pm=\{z\in \mC\mid \pm\Im z>0\}$ ($Y_\pm$~have analytic continuations to $\mC_\pm$ except for certain cuts).
First, we will find a ``background'' solution of the AKNS system~\eqref{teq},~\eqref{xeq}, which will be used below.

\subsection{The background solution}

We seek the background solution of the AKNS system~\eqref{teq}, \eqref{xeq} in the form (cf.~\cite{Filipkovska/Kotlyarov})
 \begin{equation}\label{PsolAKNS}
\Phi_0(t,x,\la)= {\rm e}^{\ii(\alpha_0 t+\beta_0 x)\sigma_3} M_0(\la) {\rm e}^{-\ii(\alpha(\la) t+\beta(\la) x)\sigma_3},\qquad \la\in\R,
 \end{equation}
where $\Phi_0(t,x,\la)$ is a matrix function with the unit determinant, and constants $\alpha_0$, $\beta_0$, a $2\times 2$ matrix function $M_0(\la)$ and scalar functions $\alpha(\la)$, $\beta(\la)$ are to be determined.
Consider the logarithmic derivative of $\Phi_0$ in $t$:
\[
(\Phi_0)_t(t,x,\la) \Phi_0^{-1}(t,x,\la)=\ii\alpha_0\sigma_3- \ii\alpha(\la)
 {\rm e}^{\ii(\alpha_0 t+\beta_0 x)\sigma_3}
M_0(\la)\sigma_3 M_0^{-1}(\la) {\rm e}^{-\ii(\alpha_0 t+\beta_0 x)\sigma_3}.
\]
Let
\begin{equation}\label{M_0}
M_0(\la):=\begin{pmatrix} a(\la) & b(\la) \\ b(\la) & a(\la) \end{pmatrix}
\end{equation}
be a matrix with the unit determinant, that is, $a^2(\la)-b^2(\la)\equiv 1$ (then $\det \Phi_0(t,x,\la)\equiv 1$). Then $M_0(\la)\sigma_3 M_0^{-1}(\la)= \big(a^2+b^2\big)\sigma_3-2\ii ab\sigma_2$ and, therefore,
\begin{gather*}
 (\Phi_0)_t(t,x,\la)\Phi_0^{-1}(t,x,\la)\\
 \qquad {} = \ii\big(\alpha_0-\alpha(\la)\big[a^2(\la)+b^2(\la)\big]\big)\sigma_3 -2\alpha(\la) a(\la)b(\la) {\rm e}^{\ii(\alpha_0 t+\beta_0 x)\sigma_3} \sigma_2 {\rm e}^{-\ii(\alpha_0 t+\beta_0 x)\sigma_3}.
\end{gather*}
On the other hand, the relation
\[
(\Phi_0)_t\Phi_0^{-1}=-\ii\la\sigma_3-H_0(t,x),\qquad H_0(t,x)=\frac{1}{2}
\begin{pmatrix}
0& {\mathcal E}_{\rm bg}(t,x)\\-\overline{{\mathcal E}_{\rm bg}(t,x)}& 0
\end{pmatrix},
\]
must be satisfied (because then $\Phi_0$ satisfies the equation~\eqref{teq}, where ${\mathcal E}(t,x)={\mathcal E}_{\rm bg}(t,x)$). Comparing the obtained relations for the logarithmic derivative $(\Phi_0)_t\Phi_0^{-1}$, we obtain that
\[
\alpha(\la)=
\frac{\la+\alpha_0}{a^2(\la)+b^2(\la)}, \qquad
{\mathcal E}_{\rm bg}(t,x)= -4\ii\alpha(\la)a(\la)b(\la)
 {\rm e}^{2\ii(\alpha_0 t+\beta_0 x)}.
\]
We need to find $a(\la)$, $b(\la)$ and $\alpha_0$ such that $\alpha(\la)a(\la)b(\la)$ is constant, since ${\mathcal E}_{\rm bg}(t,x)$ is independent of $\la$. We introduce the function
\begin{equation}\label{varkappa}
\varkappa(z)= \left(\frac{z-\overline{E}}{z-E}\right)^{1/4},
\end{equation}
where $z\in \mC$ ($z=\la+\ii\nu$) and $E$ is a point in the complex plane which will be specified below (see~\eqref{E}), take the branch cut for $\varkappa(z)$ along the closed interval $\big[E,\overline{E}\big]=\big\{z\in\mC\mid z=E+\alpha\big(\overline{E}-E\big), \alpha\in [0,1]\big\}$
and fix its branch by the asymptotics
\[
\varkappa (z)=1+\big(E-\overline{E}\big)/(4z)+O\big(z^{-2}\big),\qquad z\to\infty.
\]
Further, we introduce the function
 \begin{equation}\label{w-funct}
w(z)=\sqrt{(z-E)\big(z-\overline{E}\big)},
 \end{equation}
where $z\in \mC$ and $E$ is a point mentioned above, and fix the branch of $w(z)$ by taking the branch cut along $\big[E,\overline{E}\big]$ and by asymptotics
\[
 w(z)=z-\Re E+O\big(z^{-1}\big),\qquad z\to\infty.
\]
Thus, $\varkappa(z)$ and $w(z)$ are analytic for $z\in \mC\setminus \big[E,\overline{E}\big]$.
Now, we set
\begin{equation}\label{ab}
a(\la)=\frac{\varkappa (\la)+\varkappa^{-1}(\la)}{2},\qquad b(\la)=\frac{\varkappa (\la)-\varkappa^{-1}(\la)}{2},
\end{equation}
then $(a(\la)+b(\la))^2=\varkappa^2(\la)$ and
\begin{equation}\label{a_b}
a^2(\la)+b^2(\la)=\frac{\la-\Re E}{w(\la)},\qquad a(\la)b(\la)=\frac{\ii\Im E}{2 w(\la)},\qquad \la\in\R.
\end{equation}
Hence, ${\mathcal E}_{\rm bg}(t,x)= 2\Im E\frac{\la+\alpha_0}{\la-\Re E} {\rm e}^{2\ii(\alpha_0 t+\beta_0 x)}$ and it is independent of $\la$ if and only if $\alpha_0=-\Re E$. Therefore,
\[
\alpha(\la)=w(\la),\qquad {\mathcal E}_{\rm bg}(t,x)=2\Im E {\rm e}^{2\ii(\alpha_0 t+\beta_0 x)},\qquad \alpha_0=-\Re E.
\]
Comparing ${\mathcal E}_{\rm bg}(t,0)$ with ${\mathcal E}_{\rm in}(t)=A_0 {\rm e}^{\ii\omega_0 t}$, we obtain that $2\Im E=A_0$ and $2\Re E=-2\alpha_0=-\omega_0$, and hence
\begin{equation}\label{E}
E=-\frac{\omega_0}{2}+\ii \frac{A_0}{2}.
\end{equation}
Thus, the matrix function $\Phi_0(t,x,\la)$ satisfies the equation $(\Phi_0)_t= -(\ii\la\sigma_3+H_0(t,x))\Phi_0$, and the matrix $H_0(t):=H_0(t,0)$ is defined by the input signal ${\mathcal E}_{\rm in}(t)$~\eqref{MBboundAe}:
\begin{equation}\label{H_0}
H_0(t):=H_0(t,0)=\frac{1}{2}
\begin{pmatrix}
0&{\mathcal E}_{\rm bg}(t,0)\\-\overline{{\mathcal E}_{\rm bg}(t,0)}
\end{pmatrix}=\frac{1}{2}
\begin{pmatrix}0&{\mathcal E}_{\rm in}(t)\\-\overline{{\mathcal E}_{\rm in}(t)}\end{pmatrix}.
\end{equation}

Further, using the second MB equation $\rho_{{\rm bg}\, t}+2\ii\la\rho_{\rm bg}=\mathcal{N}_{\rm bg}{\mathcal E}_{\rm bg}$, we find that if $\rho_{\rm bg}=\ii c_0\mathcal{E}_{\rm bg}$, then $\mathcal{N}_{\rm bg}=-2c_0(\la-\Re E)$. Taking into account the normalization condition
\[
\mathcal{N}_{\rm bg}^2+|\rho_{\rm bg}|^2=
4c_0^2\big[(\la-\Re E)^2+\Im^2 E\big]=1,
\]
we obtain $c_0=\frac{1}{2w(\la)}$ and $\rho_{\rm bg}=\frac{\ii \mathcal{E}_{\rm bg}}{2w(\la)}$. The first MB equation gives
\[
\int_{-\infty}^\infty n(s)\rho_{\rm bg}(t,x,s)\, \dd s ={\mathcal E}_{{\rm bg}\, t}+{\mathcal E}_{{\rm bg}\, x} =2\ii(\alpha_0+\beta_0)\mathcal{E}_{\rm bg}= 2\ii(\beta_0- \Re E)\mathcal{E}_{\rm bg},
\]
and hence
\begin{equation}\label{beta0}
\beta_0=\Re E +\frac{1}{4}\int_{-\infty}^\infty\frac{n(s)}{w(s)}\, \dd s.
\end{equation}

Thus, we obtain the periodic solution in the form of a plane wave for the MB equations~\eqref{MB1a}:
 \begin{gather}
{\mathcal E}_{\rm bg}(t,x)=
2\Im E {\rm e}^{2\ii(\alpha_0 t+\beta_0 x)},\qquad
\rho_{\rm bg}(t,x,\la)= \frac{\ii\Im E}{w(\la)} {\rm e}^{2\ii(\alpha_0 t+\beta_0 x)},\nonumber\\
\mathcal{N}_{\rm bg}(t,x,\la)=-\frac{\la-\Re E}{w(\la)}, \label{BackgrSol}
 \end{gather}
where $\Im E=A_0/2$, $\alpha_0=-\Re E=\omega_0/2$, $\beta_0$ has the form~\eqref{beta0}, and
\[
{\mathcal E}_{\rm bg}(t,0)= A_0 {\rm e}^{\ii\omega_0 t}={\mathcal E}_{\rm in}(t)
\]
is the input signal~\eqref{MBboundAe}. The third MB equation is evidently fulfilled.
The functions $\rho_{\rm bg}(t,x,z)$ and $\mathcal{N}_{\rm bg}(t,x,z)$ are analytic for $z\in \mC\setminus\big[E,\overline{E}\big]$, i.e., they are analytic in the same domain as~$w(z)$. Note that $\rho_{\rm bg}(t,x,\la)$ and the partial derivative $(\mathcal{N}_{\rm bg})_\la(t,x,\la)$ are discontinuous in $\la$ at the point~$\Re E$.
To obtain the infinitely differentiable solution with respect to $\la$ on the whole line $\R$, we have to redefine the function $w(z)$: we draw a branch cut connecting the points $E$ and~$\overline{E}$ via infinity without intersection of the real line and fix the branch of $w(z)$ by the condition $w(0)=|E|$.

The solution~\eqref{BackgrSol} is equivalent to the background solution obtained in \cite[Section II.B]{BGKL19}, although the work \cite{BGKL19} considers a different IBVP for the MB equations, written in a comoving frame of reference, with similar nonzero boundary conditions as $t\to +\infty$, where $t$ is a retarded time and does not coincide with the variable $t$ that is used in the present paper (which considers the MB equations in a laboratory frame). In comparison with \cite{BGKL19} that considers a more general input signal and studies soliton solutions, the present work carries out a more rigorous analysis: carefully formulates the IST for the IBVP, give the representations of a solution of the IBVP in terms of the associated RH problems and provides conditions for the existence and uniqueness of solutions of the IBVP and the associated RH problems, as well as proves certain important properties of solutions.

Now, let us consider the logarithmic derivative of $\Phi_0(t,x,\la)$~\eqref{PsolAKNS} with respect to $x$:
\begin{align*}
(\Phi_0)_x\Phi_0^{-1}&=
\ii\beta_0\sigma_3- \ii\beta(\la)
 {\rm e}^{\ii(\alpha_0 t+\beta_0 x)\sigma_3} M_0(\la) \sigma_3 M_0^{-1}(\la) {\rm e}^{-\ii(\alpha_0 t+\beta_0 x)\sigma_3} \\
&=\ii\big(\beta_0- \beta(\la)\big[a^2(\la)+b^2(\la)\big]\big) \sigma_3 - 2\beta(\la)a(\la)b(\la)
 {\rm e}^{\ii(\alpha_0 t+\beta_0 x)\sigma_3}
\sigma_2 {\rm e}^{-\ii(\alpha_0 t+\beta_0 x)\sigma_3}.
\end{align*}
Denote
\begin{gather*}
G_0(t,x,\la)= \frac{1}{4} {\rm p.v.}\int_{-\infty}^\infty \frac{F_0(t,x,s)n(s)}{s-\la}\, \dd s, \\
F_0(t,x,s)= \begin{pmatrix} {\mathcal N}_{\rm bg}(t,x,s) &\rho_{\rm bg}(t,x,s)\\ \overline{\rho_{\rm bg}(t,x,s)} &-{\mathcal N}_{\rm bg}(t,x,s) \end{pmatrix}.
\end{gather*}
Since we need to obtain the $x$-equation of the AKNS system, i.e., the equation~\eqref{xeq}, then
\begin{align*}
(\Phi_0)_x\Phi_0^{-1} ={}& \ii\la\sigma_3+H_0-\ii G_0= \ii\left(\la-\frac{1}{4} {\rm p.v.}\int_{-\infty}^\infty \frac{\mathcal{N}_{\rm bg}(t,x,s) n(s)}{s-\la}\, \dd s \right)\sigma_3 \\
 &{}+\frac{1}{2} \begin{pmatrix} 0 & {\mathcal E}_{\rm bg}(t,x) \\ -\overline{{\mathcal E}_{\rm bg}(t,x)} & 0\end{pmatrix}- \frac{\ii}{4} {\rm p.v.}\int_{-\infty}^\infty
\begin{pmatrix} 0 & \rho_{\rm bg}(t,x,s) \\ \overline{\rho_{\rm bg}(t,x,s)} & 0 \end{pmatrix} \frac{n(s)}{s-\la}\, \dd s \\
 ={}&\ii\left(\la +\frac{1}{4} {\rm p.v.}\int_{-\infty}^\infty \frac{(s-\Re E)n(s)}{w(s)(s-\la)}\, \dd s\right) \sigma_3 \\
 &{}+\frac{1}{2} \begin{pmatrix} 0 & {\mathcal E}_{\rm bg}(t,x) \\ -\overline{{\mathcal E}_{\rm bg}(t,x)} & 0 \end{pmatrix} \left[1+\frac{1}{4} {\rm p.v.}\int_{-\infty}^\infty \frac{n(s)}{w(s)(s-\la)}\, \dd s\right].
\end{align*}
Using the relations $s-\Re E=\la-\Re E +s-\la$,~\eqref{a_b}, and~\eqref{beta0} and comparing the relations for the logarithmic derivative $(\Phi_0)_x\Phi_0^{-1}$, we obtain
\[
\beta(\la)=-w(\la)\left(1+
\frac{1}{4} {\rm p.v.}\int_{-\infty}^\infty\frac{n(s)}{w(s)(s-\la)}\, \dd s \right),\qquad \la\in\R.
\]

 \subsection{Preliminaries}

In what follows, \emph{we suppose that $n(\pm\Lambda)=0$, where $0<\Lambda\le+\infty$, and if $\Lambda<\infty$, then $n(\la)=0$ on $\R\setminus[-\Lambda,\Lambda]$} (\emph{i.e., $n(\la)$ has a compact support $[-\Lambda,\Lambda]$}, $\Lambda>0$).
\emph{In the case when $\Lambda=+\infty$, we denote $[-\Lambda,\Lambda]:=\R$} for convenience.
Thus, we consider the two cases: $\Lambda=+\infty$ and $0<\Lambda<+\infty$.

Below, we use the ordinary H\"{o}lder condition and the H\"{o}lder condition in a neighborhood of infinity \cite{Muskhelishvili}. As usual, a scalar function $f(\zeta)$ satisfies the H\"{o}lder condition on a set $Z\subset\mC$ or~$\subset\R$ if there exist positive constants $C$, $\mu$ such that $|f(\zeta_1)-f(\zeta_2)|\le C|\zeta_1-\zeta_2|^\mu$ for any $\zeta_1, \zeta_2\in Z$. The scalar function $f(\zeta)$ ($\zeta\in\mC$ or $\in\R$) satisfies the \emph{H\"{o}lder condition in a neighborhood of infinity} if there exist positive constants $C$, $\mu$ such that $|f(\zeta_1)-f(\zeta_2)|\le C\big|\zeta_1^{-1}-\zeta_2^{-1}\big|^\mu$ for sufficiently large $|\zeta_1|$, $|\zeta_2|$ (for any $|\zeta_1|$, $|\zeta_2|$ from the neighborhood of infinity).
Notice that if $\lim_{|\zeta|\to\infty}f(\zeta)=c<\infty$ and $f(\zeta)=c+O(|\zeta|^{-\mu})$, where $\mu={\rm const}>0$, for sufficiently large $|\zeta|$ ($|\zeta|\to\infty$), then $f(\zeta)$ is said to satisfy the H\"{o}lder condition at the point $\zeta=\infty$ \cite{Muskhelishvili}.
It is known that if the function $f(\zeta)$ is defined on an arc $L$ and satisfies the H\"{o}lder condition with the index $\mu>1$ on it, then $f(\zeta)={\rm const}$ on $L$ (see, e.g., \cite{Muskhelishvili}). Thus, in what follows, we assume that $0<\mu\le 1$ in the H\"{o}lder conditions.

In what follows, \emph{we assume that $n(\la)$} (recall that, according to the problem statement, $n(\la)\ge 0$ and $\int_{-\infty}^\infty n(\la)\, \dd\la=1$, and thus $n(\la)$ is absolutely integrable on any interval in $\R$) \emph{satisfies the following conditions}:
 \begin{enumerate}[label={\upshape\arabic*.},ref={\upshape\arabic*}]\itemsep=0pt
\item\label{Lambda_infty} In the case when $\Lambda=+\infty$, the function $n(\la)$ satisfies the H\"{o}lder condition on $\R$ and in a~neighborhood of infinity, and ${n(\pm\Lambda)=0}$ ($n(\la)=O(|\la|^{-\mu})$, $|\la|\to+\infty$, $0<\mu\le 1$).
\item\label{Lambda_finite} In the case when $\Lambda<\infty$ ($\Lambda>0$), the function $n(\la)$ has a compact support $[-\Lambda,\Lambda]$ and one of the following conditions holds:
 \begin{enumerate}[label={\upshape 2.\arabic*.},ref={\upshape 2.\arabic*}]\itemsep=0pt 
 \item\label{Lambda_finite1} $n(\la)$ satisfies the H\"{o}lder condition on $[-\Lambda,\Lambda]$ and $n(\pm\Lambda)=0$;
 \item\label{Lambda_finite2} $n(\la)={\rm const}$ on $(-\Lambda,\Lambda)$ and $n(\pm\Lambda)=0$; in this case
 \begin{equation}\label{Ex.n-const}
 n(\la)= \begin{cases} \displaystyle \frac{1}{2\Lambda}, & |\la|<\Lambda, \\ 0, & |\la|\ge\Lambda. \end{cases}
 \end{equation}
 \end{enumerate}
 \end{enumerate}

In general, $n(\la)$ can be a singular generalized function, for example, the delta function. The case when $n(\la)=\delta(\la)$, where $\delta(\lambda)$ is the Dirac delta function, was considered in \cite{Filipkovska/Kotlyarov}.

Introduce the function
\begin{equation}\label{eta}
\eta(z)= z+\frac{1}{4}\int_{-\Lambda}^\Lambda \frac{n(s)}{s-z}\, \dd s,\qquad \Lambda\le \infty,
\end{equation}
where $\Lambda<\infty$ if it is supposed that $n(\la)$ has the compact support $[-\Lambda,\Lambda]$, and $\Lambda=+\infty$ otherwise.
The function $\eta(z)$ is analytic for $z\in \mC\setminus[-\Lambda,\Lambda]$, and, hence, the function $\Im\eta(z)$ is harmonic in $\mC\setminus[-\Lambda,\Lambda]$. Since
\begin{equation*}
 \Im\eta(z)=\Im z\left(1+\frac{1}{4}\int_{-\Lambda}^\Lambda \frac{n(s)}{(s-\Re z)^2+\Im^2 z}\, \dd s\right),
\end{equation*}
then
\[
\sgn (\Im\eta(z))=\sgn(\Im z).
\]
The function $\eta(z)$ has the boundary (limiting) values $\eta_\pm(\la)=\eta(\la\pm\ii0)$ ($z=\la+\ii\nu$), continuous on $[-\Lambda,\Lambda]$:
\begin{align}
 \eta_\pm(\la)&{}=\la+\frac{1}{4}\int_{-\Lambda}^\Lambda \frac{n(s)}{s-\la\mp\ii0}\, \dd s \nonumber\\
 &{}= \la+ \frac{1}{4} {\rm p.v.} \int_{-\Lambda}^\Lambda \frac{n(s)}{s-\la}\, \dd s \pm\frac{\pi\ii}{4}n(\la),\qquad \la\in[-\Lambda,\Lambda],\ \Lambda\le\infty,\label{eta_pm}
\end{align}
and the jump $\eta(\la+\ii0)-\eta(\la-\ii0)=\frac{\pi\ii}{2}n(\la)$ across $[-\Lambda,\Lambda]$.

In the formula~\eqref{eta_pm}, for case~\ref{Lambda_infty} and for case~\ref{Lambda_finite} with condition~\ref{Lambda_finite1} (see above), we have
\begin{gather*}
{\rm p.v.} \int_{-\Lambda}^\Lambda \frac{n(s)}{s-\la}\, \dd s= \begin{cases}
 \displaystyle\int_{-\Lambda}^\Lambda \frac{n(s)-n(\la)}{s-\la}\, \dd s+ n(\la) \ln\frac{\Lambda-\la}{\Lambda+\la}, & \la\in (-\Lambda,\Lambda),\ \Lambda<\infty, \vspace{1mm}\\
\displaystyle \int_{-\Lambda}^\Lambda \frac{n(s)}{s\mp\Lambda}\, \dd s, &
 \la=\pm\Lambda,\ \Lambda<\infty, \vspace{1mm}\\
\displaystyle \int_{-\Lambda}^\Lambda \frac{n(s)-n(\la)}{s-\la}\, \dd s, & \la\in\R,\ \Lambda=+\infty,
 \end{cases}
\end{gather*}
for case~\ref{Lambda_finite2}, i.e., $n(\la)$ of the form~\eqref{Ex.n-const}, we have
\[
{\rm p.v.} \int_{-\Lambda}^\Lambda \frac{n(s)}{s-\la}\, \dd s =\begin{cases}
\displaystyle \frac{1}{2\Lambda}\ln\frac{\Lambda-\la}{\Lambda+\la}, & \la\in (-\Lambda,\Lambda), \\
 0, & \la=\pm\Lambda.
\end{cases}
\]

Notice that, generally, instead of condition~\ref{Lambda_finite1} one can require that the function $n(\la)$ \emph{belong to the class $H^*$ on} $[-\Lambda,\Lambda]$ \cite{Muskhelishvili}, which means the following: $n(\la)$ satisfies the H\"{o}lder condition on any closed interval in $(-\Lambda,\Lambda)$ and near the endpoints $\mp\Lambda$ can be represented as $n(\la)=\frac{n_i(\la)}{|\la\pm \Lambda|^{\alpha_i}}$, $i=1,2$, respectively, where $0\le\alpha_i<1$ and $n_1(\la)$, $n_2(\la)$ are defined and satisfy the H\"{o}lder condition on sufficiently small parts of $[-\Lambda,\Lambda]$ that are adjacent to the endpoints $-\Lambda$,~$\Lambda$, respectively (the contour $[-\Lambda,\Lambda]$ is oriented from $-\Lambda$ to~$\Lambda$).
Then the function $\eta(z)$~\eqref{eta} is analytic on $\mC\setminus[-\Lambda,\Lambda]$, continuous up to $(-\Lambda,\Lambda)$ and has weak singularities at the points~$\pm\Lambda$ \big(the singularities of the type $O\big((z\mp\Lambda)^{-\nu}\big)$, $0<\nu<1$\big). Near the endpoints $\pm\Lambda$, the Cauchy integral in~\eqref{eta} admits the estimate $\left|\int_{-\Lambda}^\Lambda \frac{n(s){\rm d}s}{s-z}\right|<\frac{{\rm const}}{|z\mp\Lambda|^{\nu_i}}$, $\alpha_i<\nu_i<1$, $i=1,2$, where~$\alpha_1$,~$\alpha_2$ were defined above (cf. \cite{Muskhelishvili}).
A function $\eta(z)$ will be used when constructing the RH problem as well as the basic solutions of the AKNS systems (see below). In order for the RH problem to have no singularities, including weak singularities, at the points $\pm\Lambda$, we will not use this requirement and keep condition~\ref{Lambda_finite1} for $n(\la)$, although in the case of a more general RH problem it can be used.

Since we consider the problem on the propagation of an electromagnetic wave in the stable medium (i.e., the problem for an attenuator), we can assume that ${\mathcal N}(t,x,\la)+1\to 0$ fast enough as $x\to\infty$ so that $[{\mathcal N}(t,x,\la)+1]\in L^1(\R_+)$ with respect to $x$ for all $t\in\R_+$, $\la\in\R$ (or, given the relationship~\eqref{rhoN} between $\rho(t,x,\la)$ and ${\mathcal N}(t,x,\la)$, we can assume that $\rho(t,x,\la)\to 0$ fast enough as $x\to\infty$ for all $t\in\R_+$, $\la\in\R$).

The next known lemma will be used below.
 \begin{Lemma}[\cite{BK00,Kotlyarov13}]\label{lem 2.1}
Let the system~\eqref{teq},~\eqref{xeq} be compatible $($i.e., the condition~\eqref{ZCC} holds$)$ for all $t,x,\la \in\R$. Assume that $W(t,x,\la)$ is a matrix satisfying the $t$-equation~\eqref{teq} for all $x$, and $W(t_0,x,\la)$ satisfies the $x$-equation~\eqref{xeq} for some $t=t_0$, $t_0\le\infty$. Then $W(t,x,\la)$ satisfies~\eqref{xeq} for all $t$. The same result is true, if one replaces $t$ by $x$ and vice versa.
 \end{Lemma}

 \subsection{The basic solutions of the AKNS systems and their properties}\label{PropertyBasicSol}

We will seek the two pairs $Y_+(t,x,\la)$, $Z_+(t,x,\la)$ and $Y_-(t,x,\la)$, $Z_-(t,x,\la)$ of the ``basic'' solutions of the AKNS systems~\eqref{teq}, \eqref{pmxeq}, which are $2\times 2$ matrix functions and satisfy the properties specified in Theorem~\ref{Theorem_SolAKNS} below.

The first ``basic'' solutions of the AKNS systems~\eqref{teq},~\eqref{pmxeq} we seek in the form
\begin{equation}\label{Y}
 Y_\pm(t,x,\la)= W_\pm(t,x,\la)\Phi(t,\la),\qquad \la\in\R,
\end{equation}
where the matrix functions $W_\pm(t,x,\la)$ satisfy the $x^\pm$-equations~\eqref{pmxeq} (respectively) for each $t\in\R_+$ and $W_\pm(t,0,\la)=I$, and the matrix function $\Phi(t,\la)$ satisfies the $t$-equation~\eqref{teq} for $x=0$ with the matrix $H(t,0)=H_0(t)$~\eqref{H_0} defined by the input signal~\eqref{MBboundAe}, i.e.,
\begin{equation}\label{Phi_AKNS}
\Phi(t,\la)=\Phi_0(t,0,\la)= {\rm e}^{-\ii\Re E t\sigma_3} M_0(\la) {\rm e}^{-\ii w(\la) t\sigma_3},
\end{equation} 
where $M_0(\la)$ and $w(\la)$ are defined by~\eqref{M_0},~\eqref{ab},~\eqref{varkappa} and~\eqref{w-funct}, and $E=-\omega_0/2+\ii A_0/2$ is determined by the boundary data~\eqref{MBboundAe}.
The second ``basic'' solutions we seek in the form
\begin{equation}\label{Z}
Z_\pm(t,x,\la)=\Psi(t,x,\la)w_\pm(x,\la),\qquad \la\in\R,
\end{equation}
where $\Psi(t,x,\la)$ satisfies the $t$-equation~\eqref{teq} for each $x\in (0,L)$, ${L\le\infty}$, and~$\Psi(0,x,\la)=I$, and the matrix functions $w_\pm(x,\la)$ satisfy the $x^\pm$-equations~\eqref{pmxeq} (respectively) for $t=0$ and ${w_\pm(L,\la){\rm e}^{-\ii\eta_\pm(\la)L\sigma_3}=I}$, $L\le+\infty$, where{\samepage
\[
w_\pm(L,\la){\rm e}^{-\ii\eta_\pm(\la)L\sigma_3}= \lim_{x\to+\infty}w_\pm(x,\la){\rm e}^{-\ii\eta_\pm(\la)x\sigma_3}=I
\] when $L=+\infty$.}

Under the conditions imposed on the function $n(\la)$ above, $Y_+(t,x,\la)=Y_-(t,x,\la)$ and $Z_+(t,x,\la)=Z_-(t,x,\la)$ (similarly for $W_\pm(t,x,\la)$ and $w_\pm(x,\la)$) for $\la\in\R\setminus[-\Lambda,\Lambda]$, $\Lambda<\infty$.

Recall that $\big[E,\overline{E}\big]=\big\{E+\alpha\big(\overline{E}-E\big) \mid \alpha\in [0,1]\big\}$ is a closed interval on the complex plane $\mC$. The interval $\big[E,\overline{E}\big]$ is oriented downward. Below we consider the parts of $\big[E,\overline{E}\big]$, for example, the semi-open intervals $[E,\Re E)$ and $\big(\Re E,\overline{E}\big]$.

Throughout the paper, $A[i]$ denotes the $i$-th column of a matrix $A$.

In the theorem below, for convenience, we assume that $L=+\infty$, i.e., $x\in\R_+$. For the case when $x\in (0,L)$, $L<+\infty$, the theorem remains valid and the proof is similar.

 \begin{Theorem}\label{Theorem_SolAKNS}
Let the functions ${\mathcal E}(t,x)$, $\rho(t,x,\la)$ and $\mathcal N(t,x,\la)$ be a solution of the IBVP \eqref{MB1a}--\eqref{MBboundAe} and have the following properties: ${{\mathcal E}(t,x)\in C^1(\R_+\times\R_+)}$, the partial derivatives of $\rho(t,x,\la)$ and ${\mathcal N}(t,x,\la)$ with respect to $t$ and $\la$ are continuous for ${(t,x,\la)\in\R_+\times\R_+\times\R}$, and \linebreak
${{\mathcal N}(t,x,\la)+1\to 0}$ fast enough as $x\to\infty$ so that $[{\mathcal N}(t,x,\la)+1]\in L^1(\R_+)$ with respect to~$x$ for all $t\in\R_+$, $\la\in\R$.
Suppose $n(\la)$ satisfies conditions~\ref{Lambda_infty},~\ref{Lambda_finite} mentioned above.

Then there exist solutions $Y_\pm(t,x,\la)$ and $Z_\pm(t,x,\la)$ of the systems of the $t$- and $x^\pm$-equa\-tions~\eqref{teq}, \eqref{pmxeq}, which have the form~\eqref{Y} and~\eqref{Z}, where ${Y_+(t,x,\la)=Y_-(t,x,\la)}$ and \linebreak
${Z_+(t,x,\la)=Z_-(t,x,\la)}$ for $\la\in\R\setminus[-\Lambda,\Lambda]$, $0<\Lambda<\infty$ $(t, x\in[0,\infty))$, and these solutions have the following properties:
 \begin{enumerate}[label={\upshape(\alph*)},ref={\upshape(\alph*)}]\itemsep=0pt
\item\label{UnitDeterminant} $\det Y_\pm(t,x,\la)\equiv 1$ and $\det Z_\pm(t,x,\la)\equiv 1$, $\la\in\R$;

\item\label{Differ} $Y_\pm(t,x,\la)$ and $Z_\pm(t,x,\la)$ are continuously differentiable in $t$ and $x$;

\item\label{AsymptY} the functions $Y_+(t,x,\la)$ and $Y_-(t,x,\la)$ have the analytic continuations $Y^+(t,x,z)$ and $Y^-(t,x,z)$ for $z\in\mC_+^{\rm cl}\setminus([-\Lambda,\Lambda]\cup[E,\Re E])$ and $z\in\mC_-^{\rm cl}\setminus\big([-\Lambda,\Lambda]\cup\big[{\Re} E,\overline{E}\big]\big)$, respectively, where $\mC_\pm^{\rm cl}=\mC_\pm\cup\R$, $\Lambda\le+\infty$, $(Y^+(t,x,z)=Y^-(t,x,z)$ for $\la\in\R\setminus[-\Lambda,\Lambda]$, $\Lambda<\infty)$, which are continuous up to $([-\Lambda,\Lambda]\setminus\{\Re E\})\cup[E,\Re E)$ and $\big([-\Lambda,\Lambda]\setminus\{\Re E\}\big)\cup\big(\Re E,\overline{E}\big]$ respectively, have the singularities of the type $(z-E)^{-1/4}$ and \smash{$\big(z-\overline{E}\big)^{-1/4}$} at the points $E$ and~$\overline{E}$ respectively, and are bounded in the neighborhood of the point $\Re E$, and also, for fixed $t$ and $x$,
 \begin{alignat*}{3}
 &Y^-[1](t,x,z){\rm e}^{\ii(zt-\eta(z)x)}= \begin{pmatrix}1\\0\end{pmatrix}+ O\big(z^{-1}\big),\qquad &&z\to\infty,\ z\in\mC_-^{\rm cl}, &\\
 &Y^+[2](t,x,z){\rm e}^{-\ii(zt-\eta(z)x)}= \begin{pmatrix}0\\1\end{pmatrix}+ O\big(z^{-1}\big),\qquad && z\to\infty,\ z\in\mC_+^{\rm cl}.& 
 \end{alignat*}

\item\label{AsymptZ} the functions $Z_+(t,x,\la)$ and $Z_-(t,x,\la)$ have the analytic continuations $Z^+(t,x,z)$ and $Z^-(t,x,z)$ for $z\in\mC_+^{\rm cl}\setminus[-\Lambda,\Lambda]$ and $z\in\mC_-^{\rm cl}\setminus[-\Lambda,\Lambda]$ respectively, where $\Lambda\le+\infty$, ($Z^+(t,x,z)=Z^-(t,x,z)$ for $\la\in\R\setminus[-\Lambda,\Lambda]$, $\Lambda<\infty$) which are continuous up to $[-\Lambda,\Lambda]$, and
 \begin{alignat*}{3}
 &Z^+[1](t,x,z){\rm e}^{\ii(zt-\eta(z)x)}= \begin{pmatrix}1\\0\end{pmatrix}+ O\big(z^{-1}\big),\qquad &&z\to\infty,\ z\in\mC_+^{\rm cl},& \\
 &Z^-[2](t,x,z){\rm e}^{-\ii(zt-\eta(z)x)}= \begin{pmatrix}0\\1\end{pmatrix}+ O\big(z^{-1}\big),\qquad &&z\to\infty,\ z\in\mC_-^{\rm cl}. &
 \end{alignat*}
 \end{enumerate}
 \end{Theorem}

 \begin{proof}
Since ${\mathcal E}(t,x)$, $\rho(t,x,\la)$ and $\mathcal N(t,x,\la)$ satisfy the MB equations~\eqref{MB1a}, then the AKNS systems \big(the $t$- and $x^\pm$-equations\big)~\eqref{teq},~\eqref{pmxeq} are compatible, i.e., the compatibility conditions~\eqref{ZCC_pm} hold.

By Lemma~\ref{lem 2.1}, if the matrices $W_\pm(t,x,\la)$, $\Phi(t,\la)$ and $\Psi(t,x,\la)$, $w_\pm(x,\la)$ from~\eqref{Y} and~\eqref{Z} exist, then the matrix functions $W_\pm(t,x,\la)\Phi(t,\la)$ \eqref{Y} and
$\Psi(t,x,\la)w_\pm(x,\la)$ \eqref{Z} are solutions of the compatible AKNS systems~\eqref{teq},~\eqref{pmxeq}.
Indeed, since $W_\pm(t,x,\la)$ are solutions of the $x^\pm$-equations~\eqref{pmxeq} (respectively) for all $t$, then the product $W_\pm(t,x,\la)\Phi(t,\la)$ are also solutions of the $x^\pm$-equations for all $t$. Further, since $\Phi(t,\la)$ is a~solution of the $t$-equation~\eqref{teq} for $x=0$, then $W_\pm(t,0,\la)\Phi(t,\la)= \Phi(t,\la)$ are also solutions of $t$-equation for $x=0$ and therefore, by Lemma~\ref{lem 2.1}, $Y_\pm(t,x,\la)=W_\pm(t,x,\la)\Phi(t,\la)$ are solutions of the $t$-equation for all $x$. Thus, $Y_\pm(t,x,\la)$~\eqref{Y} are solutions of the equations~\eqref{teq} and~\eqref{pmxeq} for all~$t$,~$x$. The proof for $Z_\pm(t,x,\la)$~\eqref{Z} is similar.

The matrix function $\Phi(t,\la)$~\eqref{Phi_AKNS} was obtained above. In addition, it follows from the form of $\Phi(t,\la)$ that it is infinitely differentiable in $t$.

The existence of the matrix functions $W_\pm(t,x,\la)$, $\Psi(t,x,\la)$, and $w_\pm(x,\la)$ are proved below.
In view of the above, the existence of the solutions~\eqref{Y} and~\eqref{Z} follows from the existence of these matrix functions.

Along with proving the existence of the solutions, we will prove their properties~\ref{Differ},~\ref{AsymptY} and~\ref{AsymptZ}. Property~\ref{UnitDeterminant} will be proved at the end.

Note that the function~\eqref{continuation_G}, that is,
\[
G(t,x,z)=\frac{1}{4}\int_{-\Lambda}^\Lambda \frac{F(t,x,s)n(s)\, \dd s}{s-z},\qquad \Lambda\le\infty,
\]
is analytic for $z\in\mC\setminus[-\Lambda,\Lambda]$, and for $\la\in [-\Lambda,\Lambda]$ it has the continuous boundary values
\[
G_\pm(t,x,\la)= \frac{1}{4}\int_{-\Lambda}^\Lambda \frac{F(t,x,s) n(s)}{s-\la\mp\ii 0}\, \dd s
\]
on the left (``${}_+$'') and right (``${}_-$'') of the contour $[-\Lambda,\Lambda]$ oriented from the left to the right.

\emph{First consider the case when $\Lambda<\infty$ $(\Lambda>0)$ and $\la\in \R\setminus[-\Lambda,\Lambda]$}. For $\la\in \R\setminus[-\Lambda,\Lambda]$ the $x^\pm$-equations~\eqref{pmxeq} are the same and take the form~\eqref{xeq}, i.e.,
\begin{equation}\label{DE_Wx}
W_x=(\ii\la\sigma_3 +H(t,x)-\ii G(t,x,\la))W.
\end{equation}
We rewrite the equation~\eqref{DE_Wx} in the form
\[
 W_x=\left(\ii\left[\eta(\la)-\frac{1}{4} \int_{-\Lambda}^\Lambda\frac{[{\mathcal N}(t,x,s)+1]n(s)}{s-\la}\, \dd s \right]\sigma_3+\widehat H(t,x,\la)\right)W,\qquad \la\in \mathbb{R}\setminus[-\Lambda, \Lambda],
\]
where $\eta(\la)$ has the form~\eqref{eta} ($z=\la\in\R\setminus[-\Lambda,\Lambda]$, $\Lambda<\infty$), and
\[
\widehat H(t,x,\la)=
 \begin{pmatrix} 0 & \displaystyle\frac{1}{2}{\mathcal E}(t,x) - \frac{\ii}{4}\int_{-\Lambda}^\Lambda \frac{\rho(t,x,s)n(s)}{s-\la}\, \dd s \\
\displaystyle -\frac{1}{2}\overline{\mathcal E(t,x)} - \frac{\ii}{4}\int_{-\Lambda}^\Lambda \frac{\overline{\rho(t,x,s)}n(s)}{s-\la}\, \dd s & 0 \end{pmatrix}.
\]
Further let us put $\widehat W={\rm e}^{-\ii\alpha\sigma_3}W$, where
 \begin{gather}
\alpha=\alpha(t,x,\la)= -\frac{1}{4}\int_0^x\int_{-\Lambda}^\Lambda \frac{[{\mathcal N}(t,y,s)+1] n(s)}{s-\la}\, \dd s \dd y \nonumber\\
\hphantom{\alpha=\alpha(t,x,\la)}{}
= -\frac{1}{4}\int_{-\Lambda}^\Lambda \frac{\widehat{{\mathcal N}}(t,x,s)n(s)}{s-\la}\, \dd s,\qquad \la\in \R\setminus[-\Lambda,\Lambda], \nonumber\\
\widehat{\mathcal N}(t,x,s)=\int_0^x [{\mathcal N}(t,y,s)+1]\, \dd y.\label{widehat_N}
 \end{gather}
For $\widehat W$ the equation~\eqref{DE_Wx} takes the form
\[
\widehat W_x=\big[ \ii\eta(\la)\sigma_3+ \widetilde H(t,x,\la) \big] \widehat W,\qquad \widetilde H(t,x,\la)={\rm e}^{-\ii\alpha\sigma_3}\widehat H(t,x,\la) {\rm e}^{\ii\alpha\sigma_3}.
\]
This equation is equivalent to the matrix integral equation:
 \begin{equation}\label{widetilde_W}
\widetilde W(t,x,\la)=I+\int_0^x
{\rm e}^{\ii\eta(\la)(x-y)\sigma_3}\widetilde H(t,y,\la) \widetilde W(t,y,\la) {\rm e}^{-\ii\eta(\la)(x-y)\sigma_3} \dd y,
 \end{equation}
where
\[
\widetilde W(t,x,\la)=\widehat W(t,x,\la) {\rm e}^{-\ii\eta(\la)x\sigma_3}= {\rm e}^{-\ii\alpha(t,x,\la)\sigma_3}W(t,x,\la) {\rm e}^{-\ii\eta(\la)x\sigma_3}.
\]
From~\eqref{widetilde_W}, we obtain the vector integral equations \big(where $\widetilde W[i](t,y,\la)$ is the $i$-th column of the matrix $\widetilde W(t,y,\la)$\big)
 \begin{gather*}
\widetilde W[1](t,x,\la)=\begin{pmatrix}
1\\0\end{pmatrix} + \int_0^x
\begin{pmatrix} 0&\widetilde H_{12}(t,y,\la) \\
\widetilde H_{21}(t,y,\la){\rm e}^{-2\ii\eta(\la)(x-y)}&0 \end{pmatrix}
\widetilde W[1](t,y,\la)\, \dd y, \\ 
\widetilde W[2](t,x,\la)=\begin{pmatrix}
0\\1\end{pmatrix} + \int_0^x
\begin{pmatrix} 0&\widetilde H_{12}(t,y,\la) {\rm e}^{2\ii\eta(\la)(x-y)} \\
\widetilde H_{21}(t,y,\la)&0 \end{pmatrix}
\widetilde W[2](t,y,\la)\, \dd y. 
 \end{gather*}

Let us prove that there exists a solution $\widetilde W(t,x,\la)$ of the integral equation~\eqref{widetilde_W} and it is continuously differentiable in $t$,~$x$ and has an analytic continuation to $\mC\setminus[-\Lambda,\Lambda]$. To do this, consider the Neumann series $\sum_{n=0}^\infty \widetilde W_n(t,x,\la)$, where
\[
 \widetilde W_0(t,x,\la)\equiv I,\qquad \widetilde W_n(t,x,\la)=\int_0^x
 {\rm e}^{\ii\eta(\la)(x-y)\sigma_3}\widetilde H(t,y,\la) \widetilde W_{n-1}(t,y,\la) {\rm e}^{-\ii\eta(\la)(x-y)\sigma_3} \dd y,
\]
$n=1,2,\dots $, and $t\in\R_+$, $x\in\R_+$, $\la\in\R\setminus[-\Lambda,\Lambda]$. Denote
\[
\Upsilon(t,x,\la)= \begin{pmatrix}
 0 & \displaystyle - \frac{\ii}{4}\int_{-\Lambda}^\Lambda \frac{\rho(t,x,s)n(s)}{s-\la}\, \dd s \\ \displaystyle -\frac{\ii}{4}\int_{-\Lambda}^\Lambda \frac{\overline{\rho(t,x,s)}n(s)}{s-\la}\, \dd s & 0 \end{pmatrix},
\]
then ${\widehat H(t,x,\la)=H(t,x)+\Upsilon(t,x,\la)}$.

Recall that $|{\mathcal N}(t,x,\la)|\le 1$ and $|\rho(t,x,\la)|\le 1$ for all $(t,x,\la)\in \R_+\times\R_+\times\R$ due to~\eqref{rhoN} and that
${\mathcal N}(t,x,\la)+1\to0$ fast enough as $x\to\infty$ so that $[{\mathcal N}(t,x,\la)+1]\in L^1(\R_+)$ with respect to~$x$ for all $t\in\R_+$, $\la\in\R$. Therefore,
\begin{equation}\label{Bound-widehat_N}
0\le\widehat{\mathcal N}(t,x,\la)\le \widehat{K}={\rm const}
\end{equation}
for all $t,x\in\R_+$, $\la\in \R$, where $\widehat{\mathcal N}$ is defined in~\eqref{widehat_N}.
It follows from the properties of the functions $n$, $\mathcal N$, $\widehat{\mathcal N}$ and $\rho$ that $\alpha$, $\Upsilon$, $\alpha_t$ and $\Upsilon_t$ are continuous for $(t,x,\la)\in \R_+\times\R_+\times \R\setminus[-\Lambda,\Lambda]$, and also that $\alpha$, $\Upsilon$ are bounded for all $t, x\in\R_+$ and $\la\in(-\infty,-\Lambda_0]\cup[\Lambda_0,+\infty)$, where ${\Lambda_0>\Lambda}$ is an arbitrary number. Obviously, $\big\|{\rm e}^{\pm\ii\eta(\la)(x-y)\sigma_3}\big\|\le C_1={\rm const}$ and $\big\|{\rm e}^{\pm\ii\alpha(t,x,\la)\sigma_3}\big\|\le C_2={\rm const}$ for $t, x\in \R_+$ and $\la\in \R\setminus[-\Lambda,\Lambda]$. Here we use the operator (matrix) norm subordinate to the vector norm $\|x\|_\infty=\sup_{1\le k\le 2}|x_k|$, $x\in \mC^2$. In general, by $\|\cdot\|$ we denote any suitable norm (such that $\|I\|=1$), unless it is specified exactly which norm is considered.

Since ${\mathcal E}(t,x)\in C^1(\R_+\times\R_+)$ is bounded near $(0,0)$, i.e., bounded on any compact set $[0,L_0]\times [0,T_0]$ in $[0,\infty)\times [0,\infty)$, then there exists a constant $C_3=C_3(\Lambda_0)$ such that $\big\|\widehat H(t,x,\la)\big\|\le C_3$ for all $t\in [0,T_0]$, $x\in[0,L_0]$ and $\la\in(-\infty,-\Lambda_0]\cup[\Lambda_0,+\infty)$, where $T_0>0$, $L_0>0$ and $\Lambda_0>\Lambda$ are arbitrary numbers. In addition, we obtain that ${\mathcal E}(t,x)\in L^1[a,b]$, $0\le a\le b<\infty$, with respect to $x$ for each $t\in\R_+$ and with respect to $t$ for each $x\in\R_+$

It follows from the above that $\big\|\widetilde W_0(t,x,\la)\big\|\equiv \|I\|=1$ and
\[
\big\|\widetilde W_n(t,x,\la)\big\|\le C \int_0^x \big\|\widetilde W_{n-1}(t,y,\la)\big\|\,\dd y,
\]
where $C=C_1^2 C_2^2 C_3$, $t\in [0,T_0]$, ${T_0>0}$, $x\in [0,L_0]$, $L_0>0$, and $\la\in(-\infty,-\Lambda_0]\cup[\Lambda_0,+\infty)$, $\Lambda_0>\Lambda$, and using the method of mathematical induction (MMI) it is easy to prove that $\big\|\widetilde W_n(t,x,\la)\big\|\le \frac{C^n x^n}{n!}$.
Consequently, $\big\|\widetilde W(t,x,\la)\big\|\le \sum_{n=0}^\infty \big\|\widetilde W_n(t,x,\la)\big\|\le \sum_{n=0}^\infty \frac{C^n x^n}{n!}={\rm e}^{Cx}$ and hence the series $\sum_{n=0}^\infty \widetilde W_n(t,x,\la)$ converges absolutely and uniformly for all $t\in [0,T_0]$, $x\in [0,L_0]$ and $\la\in(-\infty,-\Lambda_0]\cup[\Lambda_0,+\infty)$, and in view of the arbitrariness of $L_0$, $T_0$ and $\Lambda_0$ it converges locally uniformly for all $t, x\in\R_+$ and $\la\in\R\setminus[-\Lambda,\Lambda]$.
Consequently, the sum $\widetilde W(t,x,\la)=\sum_{n=0}^\infty \widetilde W_n(t,x,\la)$ is a solution of the equation~\eqref{widetilde_W}, at that, the function $\widetilde W(t,x,\la)$ is continuous on $\R_+\times\R_+\times \R\setminus[-\Lambda,\Lambda]$ (since $W_n(t,x,\la)$ is continuous on this set) and continuously differentiable in $x$ \big(since $\widetilde W(t,x,\la)$ is a solution of the integral equation~\eqref{widetilde_W}\big).
Moreover, since $\eta(\la)$, $\alpha(t,x,\la)$ and $\Upsilon(t,x,\la)$ \big(and therefore $\widehat H(t,x,\la)$\big) have analytic continuations to ${\mC\setminus[-\Lambda,\Lambda]}$, then $\widetilde W(t,x,\la)$ has an analytic continuation $\widetilde W(t,x,z)$ for $z\in\mC\setminus[-\Lambda,\Lambda]$.
From the properties of the functions $\mathcal E$, $\mathcal N$, $\rho$ as a solution of the considered IBVP~\eqref{MB1a}--\eqref{MBboundAe} it follows that $\mathcal E_t(t,x)$, $\mathcal N_t(t,x,\la)$ and $\rho_t(t,x,\la)$ are bounded for $(t,x)$ close to $(0,0)$ and $\la$ from any compact set in ${\R\setminus[-\Lambda,\Lambda]}$. Then
\[
\alpha_t(t,x,\la)=-\frac{1}{4}\int_{-\Lambda}^\Lambda \frac{\widehat{\mathcal N}_t(t,x,s)n(s)}{s-\la}\, \dd s, \qquad \widehat H_t(t,x,\la)=H_t(t,x)+\Upsilon_t(t,x,\la),
\]
where
\[
\Upsilon_t= \left(\begin{matrix} 0 & \displaystyle-\frac{\ii}{4}\int_{-\Lambda}^\Lambda \frac{\rho_t(t,x,s)n(s)}{s-\la}\, \dd s \\ \displaystyle-\frac{\ii}{4}\int_{-\Lambda}^\Lambda \frac{\overline{\rho_t(t,x,s)}n(s)}{s-\la}\, \dd s & 0 \end{matrix}\right),
\]
and
\[
\widetilde H_t(t,x,\la)={\rm e}^{-\ii\alpha(t,x,\la)\sigma_3}\big(\widehat H_t(t,x,\la)+\ii\alpha_t(t,x,\la)\big[\widehat H(t,x,\la),\sigma_3\big]\big){\rm e}^{\ii\alpha(t,x,\la)\sigma_3}
\]
 are bounded on any compact set in $\R_+\times [0,\infty)\times \R\setminus[-\Lambda,\Lambda]$. Denote $\big(\widetilde W_n\big)_t= \frac{\partial \widetilde W_n}{\partial t}$. Obviously, $\big(\widetilde W_n\big)_t(t,x,\la)$ ($n=0,1,\dots $) are continuous on $\R_+\times\R_+\times\R\setminus[-\Lambda,\Lambda]$. In the same way as above, we prove that the series $\sum_{n=0}^\infty \big(\widetilde W_n\big)_t$ converges absolutely and locally uniformly on $\R_+\times\R_+\times \R\setminus[-\Lambda,\Lambda]$. Consequently, $\widetilde W(t,x,\la)$ is continuously differentiable in $t$.

 \medskip
\emph{Now consider the case when $\Lambda\le+\infty$ and $\la\in[-\Lambda,\Lambda]$} (recall that we set $[-\Lambda,\Lambda]=\R$ for $\Lambda=\infty$).

The $x^\pm$-equations~\eqref{pmxeq} (for $W_\pm(t,x,\la)$ from~\eqref{Y}), as above, can be written in the form
\[
(W_\pm)_x=\left(\ii\left[\eta_\pm(\la)-\frac{1}{4} \int_{-\Lambda}^\Lambda\frac{[{\mathcal N}(t,x,s)+1]n(s)}{s-\la\mp\ii0}\, \dd s \right]\sigma_3+\widehat H_\pm(t,x,\la)\right)W_\pm,
\]
where $\la\in[-\Lambda,\Lambda]$, $\Lambda\le\infty$, $\eta_\pm(\la)$ are the boundary values defined in~\eqref{eta_pm}, and
\begin{equation*}
 \widehat H_\pm(t,x,\la)= \begin{pmatrix}
 0 & \displaystyle\frac{1}{2}{\mathcal E}(t,x) - \frac{\ii}{4}\int_{-\Lambda}^\Lambda \frac{\rho(t,x,s)n(s)}{s-\la\mp\ii0}\, \dd s \\
\displaystyle -\frac{1}{2}\overline{{\mathcal E}(t,x)} - \frac{\ii}{4}\int_{-\Lambda}^\Lambda \frac{\overline{\rho(t,x,s)}n(s)}{s-\la\mp\ii0}\, \dd s & 0 \end{pmatrix}.
\end{equation*}
Let us put $\widehat W_\pm={\rm e}^{-\ii\alpha_\pm\sigma_3}W_\pm$, where
\begin{equation*}
 \alpha_\pm=\alpha_\pm(t,x,\la)= -\frac{1}{4}\int_0^x\int_{-\Lambda}^\Lambda \frac{[{\mathcal N}(t,y,s)+1] n(s)}{s-\la\mp\ii0}\, \dd s\dd y = -\frac{1}{4}\int_{-\Lambda}^\Lambda \frac{\widehat{{\mathcal N}}(t,x,s)n(s)}{s-\la\mp\ii0}\, \dd s,
\end{equation*}
and $\widehat{\mathcal N}(t,x,s)$ has the form~\eqref{widehat_N}. Since
\[
(\alpha_\pm)_x=-\frac{1}{4}\int_{-\Lambda}^\Lambda \frac{[{\mathcal N}(t,x,s)+1] n(s)}{s-\la\mp\ii0}\, \dd s,
\]
 then
\begin{gather*}
(\widehat W_\pm)_x=\big[ \ii\eta_\pm(\la)\sigma_3 + \widetilde H_\pm(t,x,\la) \big] \widehat W_\pm,\\ \widetilde H_\pm(t,x,\la)= {\rm e}^{-\ii\alpha_\pm(t,x,\la)\sigma_3}\widehat H_\pm(t,x,\la){\rm e}^{\ii\alpha_\pm(t,x,\la)\sigma_3}.
\end{gather*}
These equations are equivalent to the matrix integral equations
\begin{equation}\label{widetilde_W-pm}
\widetilde W_\pm(t,x,\la)=I+\int_0^x
{\rm e}^{\ii\eta_\pm(\la)(x-y)\sigma_3} \widetilde H_\pm(t,y,\la) \widetilde W_\pm(t,y,\la) {\rm e}^{-\ii\eta_\pm(\la)(x-y)\sigma_3} \dd y,
\end{equation}
where
\begin{gather*}
\widetilde W_\pm(t,x,\la)=\widehat W_\pm(t,x,\la) {\rm e}^{-\ii\eta_\pm(\la)x\sigma_3}= {\rm e}^{-\ii\alpha_\pm(t,x,\la)\sigma_3}W_\pm(t,x,\la) {\rm e}^{-\ii\eta_\pm(\la)x\sigma_3},\\
\widetilde W_\pm(t,0,\la)=\widehat W_\pm(t,0,\la)=W_\pm(t,0,\la)=I.
\end{gather*}
From~\eqref{widetilde_W-pm}, we obtain the vector integral equations
\begin{gather*}
\widetilde W_\pm[1](t,x,\la)= \begin{pmatrix}1\\0\end{pmatrix} + \int_0^x
 \begin{pmatrix} 0& \widehat H_{\pm 12} {\rm e}^{-2\ii\alpha_\pm} \\
\widehat H_{\pm 21} {\rm e}^{2\ii\alpha_\pm} {\rm e}^{-2\ii\eta_\pm(\la)(x-y)}&0
 \end{pmatrix} \widetilde W_\pm[1](t,y,\la)\, \dd y, 
\\
\widetilde W_\pm[2](t,x,\la)=\begin{pmatrix}0\\1\end{pmatrix} +\int_0^x
\begin{pmatrix} 0& \widehat H_{\pm 12} {\rm e}^{-2\ii\alpha_\pm} {\rm e}^{2\ii\eta_\pm(\la)(x-y)} \\
\widehat H_{\pm 21} {\rm e}^{2\ii\alpha_\pm}&0 \end{pmatrix} \widetilde W_\pm[2](t,y,\la)\, \dd y. 
\end{gather*}

Let us show that solutions of the integral equations~\eqref{widetilde_W-pm} exist and they satisfy the required properties. Solutions of the equations~\eqref{widetilde_W-pm} can be represented as the Neumann series $\widetilde W_\pm(t,x,\la)=\sum_{n=0}^\infty\widetilde W_{\pm, n} (t,x,\la)$, where
\begin{gather*}
 \widetilde W_{\pm, 0}(t,x,\la)\equiv I,\\
 \widetilde W_{\pm, n}(t,x,\la)= \int_0^x
 {\rm e}^{\ii\eta_\pm(\la)(x-y)\sigma_3}\widetilde H_\pm(t,y,\la) \widetilde W_{\pm, n-1}(t,y,\la) {\rm e}^{-\ii\eta_\pm(\la)(x-y)\sigma_3} \dd y.
\end{gather*}
Introduce the matrix{\samepage
\[
\Upsilon_\pm(t,x,\la)= \begin{pmatrix}
 0 & \displaystyle -\frac{\ii}{4}\int_{-\Lambda}^\Lambda \frac{\rho(t,x,s)n(s)}{s-\la\mp\ii0}\, \dd s \\
 \displaystyle-\frac{\ii}{4}\int_{-\Lambda}^\Lambda \frac{\overline{\rho(t,x,s)}n(s)}{s-\la\mp\ii0}\, \dd s & 0 \end{pmatrix},
\]
then $\widehat H_\pm(t,x,\la)=H(t,x)+\Upsilon_\pm(t,x,\la)$.}

Obviously, $n(\la)$ is bounded on $\R$ and hence $\Im\eta_\pm(\la)=\pm\frac{\pi}{4}n(\la)$ are bounded on $\R$. Consequently, for $y\in [0,x]$, $x\in \R_+$ we obtain
\begin{gather*}
\big|{\rm e}^{\ii\eta_+(\la)(x-y)}\big|=\big|{\rm e}^{-\ii\eta_-(\la)(x-y)}\big|= {\rm e}^{-\pi/4 n(\la)(x-y)}\le 1,
\\
\big|{\rm e}^{-\ii\eta_+(\la)(x-y)}\big|=\big|{\rm e}^{\ii\eta_-(\la)(x-y)}\big|= {\rm e}^{\pi/4 n(\la)(x-y)}\le {\rm e}^{\pi/4 n(\la)x}\le {\rm e}^{Nx}, \qquad N={\rm const}>0.
\end{gather*}

Since $|\rho(t,x,\la)|\le 1$ for all $(t,x,\la)$, then, due to the properties of $n(\la)$ indicated in conditions~\ref{Lambda_infty} and~\ref{Lambda_finite}, where~\ref{Lambda_finite1} holds, the function $\rho(t,x,\la)n(\la)$ satisfies the H\"{o}lder condition uniformly with respect to $t$,~$x$ for $\la\in[-\Lambda,\Lambda]$ ($\Lambda\le\infty$) and, if $\Lambda=\infty$ (i.e., $[-\Lambda,\Lambda]=\R$), in a neighborhood of infinity.
If condition~\ref{Lambda_finite2} holds, i.e., $n(\la)$ has the form~\eqref{Ex.n-const}, then $\rho(t,x,\la)n(\la)=\frac{1}{2\Lambda}\rho(t,x,\la)$ for $\la\in(-\Lambda,\Lambda)$. Note that $\rho(t,x,\la)$ satisfies the Lipschitz condition in $\la$ on $[-\Lambda,\Lambda]$. Moreover, $\rho(t,x,\pm\Lambda)n(\pm\Lambda)\equiv 0$ ($\Lambda\le\infty$) and $\rho(t,x,\la)n(\la)$ belongs to $L^1[-\Lambda,\Lambda]$ with respect to $\la$ under both conditions~\ref{Lambda_infty} and~\ref{Lambda_finite}.
Consequently, there exists the Cauchy integral $\int_{-\Lambda}^\Lambda \frac{\rho(t,x,s)n(s)}{s-z}\, \dd s$ which is analytic on $\mC\setminus[-\Lambda,\Lambda]$ and continuous up to $[-\Lambda,\Lambda]$ (i.e., continuously
 extendible to $[-\Lambda,\Lambda]$ from the left and
from the right: $\int_{-\Lambda}^\Lambda \frac{\rho(t,x,s)n(s)}{s-\la\mp\ii0}\, \dd s <\infty$, $\la=\Re z\in[-\Lambda,\Lambda]$) as a function of the variable $z$ with the parameters $t$,~$x$, and $\int_{-\Lambda}^\Lambda \frac{\rho(t,x,s)n(s)}{s-z}\, \dd s=O\big(z^{-1}\big)$, $z\to\infty$.
The same is true for the integral with $\overline{\rho(t,x,s)}$.
Thus, there exists the function $\Upsilon(t,x,z)$ \big($\widehat H(t,x,z)$ respectively\big), analytic in $z$ on $\mC\setminus[-\Lambda,\Lambda]$ and continuous up to $[-\Lambda,\Lambda]$, which has the boundary values $\Upsilon_\pm(t,x,\la)$ \big($\widehat H_\pm(t,x,\la)$ respectively\big) from the left/right of $[-\Lambda,\Lambda]$ and the asymptotics $\Upsilon(t,x,z)=O\big(z^{-1}\big)$ as $z\to\infty$, and evidently this function coincides with the analytic continuation of $\Upsilon(t,x,\la)$ mentioned above for the case $\Lambda<\infty$, $\la\in\R\setminus[-\Lambda,\Lambda]$.
Note that
 ($\Upsilon_\pm(t,x,\la)=O(|\la|^{-\mu})$, $|\la|\to+\infty$)
\[
\|\Upsilon_\pm(t,x,\la)\|\le \upsilon,\qquad \upsilon={\rm const}>0,\qquad (t,x,\la)\in\R_+\times\R_+\times[-\Lambda,\Lambda],\qquad \Lambda\le\infty.
\]

In the same way as above, using the properties of ${\mathcal N}(t,x,\la)$ \big(and, accordingly, $\widehat{\mathcal N}(t,x,\la)$\big), the estimate~\eqref{Bound-widehat_N} and the properties of $n(\la)$ indicated in conditions~\ref{Lambda_infty} and~\ref{Lambda_finite}, we obtain that the function
\begin{gather*}
 \alpha(t,x,z)= -\frac{1}{4}\int_0^x\int_{-\Lambda}^\Lambda \frac{[{\mathcal N}(t,y,s)+1] n(s)}{s-z}\, \dd s\dd y = -\frac{1}{4}\int_{-\Lambda}^\Lambda \frac{\widehat{{\mathcal N}}(t,x,s)n(s)}{s-z}\, \dd s,\qquad \Lambda\le \infty,
\end{gather*}
is analytic for $z\in\mC\setminus[-\Lambda,\Lambda]$ and continuous up to $[-\Lambda,\Lambda]$ (evidently, this function coincides with the analytic continuation of $\alpha(t,x,\la)$ mentioned above for the case $\Lambda<\infty$, $\la\in\R\setminus[-\Lambda,\Lambda]$), and $\alpha(t,x,z)=O\big(z^{-1}\big)$ as $z\to\infty$. Its boundary values $\alpha_\pm(t,x,\la)=\alpha(t,x,\la\pm\ii0)$ are equal to
\[
\alpha_\pm(t,x,\la)= -\frac{1}{4} {\rm p.v.}\int_{-\Lambda}^\Lambda \frac{\widehat{{\mathcal N}}(t,x,s)n(s)}{s-\la}\, \dd s \mp\frac{\pi\ii}{4}\widehat{\mathcal N}(t,x,\la)n(\la), \qquad \la\in [-\Lambda,\Lambda].
\]
In addition,
\begin{equation*}
 \Im\alpha(t,x,z)= -\frac{\Im z}{4}\int_{-\Lambda}^\Lambda \frac{\widehat{\mathcal N}(t,x,s)n(s)}{(s-\la)^2+\Im^2 z}\, \dd s,\qquad z\to\infty.
\end{equation*}
Since $\Im\alpha_\pm(t,x,\la)= \mp\frac{\pi}{4}\widehat{\mathcal N}(t,x,\la)n(\la)$, then
\begin{gather*}
 \big|{\rm e}^{\ii\alpha_+(t,x,\la)}\big|=\big|{\rm e}^{-\ii\alpha_-(t,x,\la)}\big|= {\rm e}^{\pi\widehat{\mathcal N}(t,x,\la)n(\la)/4}\le {\rm e}^{\pi\widehat{K}n(\la)/4},
\\
 \big|{\rm e}^{-\ii\alpha_+(t,x,\la)}\big|=\big|{\rm e}^{\ii\alpha_-(t,x,\la)}\big|={\rm e}^{-\pi\widehat{\mathcal N}(t,x,\la)n(\la)/4}\le 1.
\end{gather*}
Thus, $\big\|{\rm e}^{\pm\ii\alpha_\pm(t,x,\la)\sigma_3}\big\|\le e^K$, where $K={\rm const}>0$, and $\big\|{\rm e}^{\pm\ii\eta_\pm(\la)(x-y)\sigma_3}\big\|\le {\rm e}^{Nx}$ for each $t\in \R_+$, $y\in [0,x]$, $x\in \R_+$, $\la\in[-\Lambda,\Lambda]$ ($\Lambda\le\infty$). Note that ${\rm e}^{\pm \ii\alpha(t,x,z)}=1+O\big(z^{-1}\big)$, $z\to\infty$ \big(${\rm e}^{\pm\ii\alpha_\pm(t,x,\la)}\to 1$, $|\la|\to\infty$\big).

It follows from the above that $\big\|\widetilde W_{\pm, 0}(t,x,\la)\big\|\equiv \|I\|=1$ and
 \begin{equation*}
\big\|\widetilde W_{\pm, n}(t,x,\la)\big\|\le {\rm e}^{2(N L_0+K)}\int_0^x (\|H(t,y)\|+\upsilon) \big\|\widetilde W_{\pm, n-1}(t,y,\la)\big\| \dd y,\qquad n=1,2,\dots,
 \end{equation*}
for all $t\in (0,T_0)$, $x\in (0,L_0)$ and $\la\in(-\Lambda_0,\Lambda_0)$, where $T_0, L_0, \Lambda_0>0$ ($\Lambda_0\le\Lambda$ for $\Lambda<\infty$ and $\Lambda_0<\infty$ for $\Lambda=\infty$) are arbitrary numbers.
Then
\[
\big\|\widetilde W_{\pm, 1}(t,x,\la)\big\|\le {\rm e}^{2(N L_0+K)} \int_0^x(\|H(t,y)\|+\upsilon)\, \dd y = {\rm e}^{2(N L_0+K)} S(t,x),
\]
 where $S(t,x)=\int_0^x(\|H(t,y)\|+\upsilon)\, \dd y$, and using the MMI, we obtain \[
 \big\|\widetilde W_{\pm, n}(t,x,\la)\big\|\le \frac{{\rm e}^{2n(NL_0+K)}S^n(t,x)}{n!}.
 \]
Consequently,
\[
\big\|\widetilde W_\pm(t,x,\la)\big\|=\bigg\|\sum_{n=0}^\infty \widetilde W_{\pm, n}(t,x,\la)\bigg\|\le \exp\big\{{\rm e}^{2(NL_0+K)}S(t,x)\big\}.
\]
Due to the properties of $\mathcal E(t,x)$, the function $S(t,x)$ is bounded on any compact set in $\R_+\times\R_+$. Thus, the series $\sum_{n=0}^\infty \widetilde W_{\pm, n}(t,x,\la)$ converge absolutely and locally uniformly for $(t,x,\la)\in\R_+\times\R_+\times[-\Lambda,\Lambda]$, $\Lambda\le\infty$.
Since the functions $\widetilde W_{\pm, n}(t,x,\la)$ are continuous and the series converge locally uniformly on $\R_+\times\R_+\times[-\Lambda,\Lambda]$, then the functions (the sums of the series) $\widetilde W_\pm(t,x,\la)$ are continuous on $\R_+\times\R_+\times[-\Lambda,\Lambda]$, $\Lambda\le\infty$. The continuous differentiability of $\widetilde W_\pm(t,x,\la)$ with respect to $t$ and $x$ is proved in the same way as for $\widetilde W(t,x,\la)$, $\la\in\R\setminus[-\Lambda,\Lambda]$ (see above).

Thus, the existence of the solutions $Y_\pm(t,x,\la)$~\eqref{Y} and property~\ref{Differ} for them have been proved. Further, we will prove~\ref{AsymptY}.

It also follows from the above that there exists the function (analytic continuation) $\widetilde W(t,x,z)$, analytic in $z$ on $\mC\setminus[-\Lambda,\Lambda]$ and continuous up to the boundary $[-\Lambda,\Lambda]$, which satisfies the integral equation~\eqref{widetilde_W} for $z=\la\in\R\setminus[-\Lambda,\Lambda]$, $\Lambda<\infty$, and has the boundary values $\widetilde W_\pm(t,x,\la)$ satisfying the integral equations~\eqref{widetilde_W-pm} for $\la\in[-\Lambda,\Lambda]$, $\Lambda\le\infty$ \big(also, $\widetilde W(t,x,z)$ is continuously differentiable in $t$, $x$ for each $z$\big). For $\Lambda<\infty$, we extend the functions $\widetilde W_\pm(t,x,\la)$ by defining them equal to $\widetilde W_\pm(t,x,\la):=\widetilde W(t,x,\la)$ for $\la\in\R\setminus[-\Lambda,\Lambda]$ \big(accordingly, in a similar way we define all the boundary values included in $\widetilde W_\pm(t,x,\la)$, i.e., $\alpha_\pm(t,x,\la):=\alpha(t,x,\la)$, $\eta_\pm(\la):=\eta(\la)$, etc., for $\la\in\R\setminus[-\Lambda,\Lambda]$, $\Lambda<\infty$\big). Thus, the function $W(t,x,z)= {\rm e}^{\ii\alpha(t,x,z)\sigma_3}\widetilde W(t,x,z){\rm e}^{\ii\eta(z)x\sigma_3}$ is analytic in $z$ on $\mC\setminus[-\Lambda,\Lambda]$ and continuous up to $[-\Lambda,\Lambda]$, it is continuously differentiable in $t$, $x$ (for each $z\in\mC$) and its boundary values $W_\pm(t,x,\la)={\rm e}^{\ii\alpha_\pm(t,x,\la)\sigma_3}\widetilde W_\pm(t,x,\la){\rm e}^{\ii\eta_\pm(\la)x\sigma_3}$ ($W_+(t,x,\la)=W_-(t,x,\la)=W(t,x,\la)$ for $\la\in\R\setminus[-\Lambda,\Lambda]$, $\Lambda<\infty$) satisfy the $x^\pm$-equations~\eqref{pmxeq} for each $t\in\R_+$ and $W_\pm(t,0,\la)=I$. Denote by
\[
\widetilde W^\pm(t,x,z):=\begin{cases} \widetilde W(t,x,z), & z\in\mC_\pm=\{z\in \mC\mid \pm\Im z>0\}, \\ \widetilde W_\pm(t,x,z), & z=\la\in\R \end{cases}
\]
 ($\widetilde W^+(t,x,z)=\widetilde W^-(t,x,z)=\widetilde W(t,x,z)$, $z=\la\in\R\setminus[-\Lambda,\Lambda]$, $\Lambda<\infty$) and introduce the notation $W^\pm(t,x,z)$, $z\in\mC_\pm\cup\R$ in a similar way.
Since ${\rm e}^{\pm\ii\eta(z)x}\to 0$ for $z\in\mC_\pm$ (respectively), $x\in\R_+$ as $z\to\infty$ (or $x\to+\infty$), then the columns $\widetilde W^-[1](t,x,z)$ and $\widetilde W^+[2](t,x,z)$ are bounded in $z\in\mC_-^{\rm cl}=\mC_-\cup\R$ and $z\in\mC_+^{\rm cl}=\mC_+\cup\R$, respectively. Further, taking into account
\begin{gather*}
\begin{split}
 & W^-[1](t,x,z){\rm e}^{-\ii\eta(z)x}={\rm e}^{\ii\alpha(t,x,z)\sigma_3}\widetilde W^-[1](t,x,z), \\
 & W^-[2](t,x,z){\rm e}^{-\ii\eta(z)x}={\rm e}^{\ii\alpha(t,x,z)\sigma_3}\widetilde W^-[2](t,x,z){\rm e}^{-2\ii\eta(z)x},
 \end{split}
\end{gather*}
where
\[
\widetilde W^-[2]{\rm e}^{-2\ii\eta x}= \begin{pmatrix} 0\\{\rm e}^{-2\ii\eta x} \end{pmatrix} + \int_0^x \begin{pmatrix} 0 & \widehat H_{12} {\rm e}^{-2\ii\alpha}{\rm e}^{-2\ii\eta y} \\
\widehat H_{21} {\rm e}^{2\ii\alpha}{\rm e}^{-2\ii\eta x} & 0 \end{pmatrix}
\widetilde W^-[2]\, \dd y
\]
(obviously, $\widehat H(t,x,z)=\widehat H_\pm(t,x,z)$, $\alpha(t,x,z)=\alpha_\pm(t,x,z)$ and $\eta(z)=\eta_\pm(z)$ for $z=\la\in\R$), we obtain that $W^-(t,x,z){\rm e}^{-\ii\eta(z)x}$ is bounded in $z\in\mC_-^{\rm cl}$, and as $z\to\infty$, $z\in\mC_-$, it has the asymptotics of the form
\begin{gather*}
W^-[1](t,x,z){\rm e}^{-\ii\eta(z)x}=\begin{pmatrix}1\\0\end{pmatrix} + O\big(z^{-1}\big),\\
W^-[2](t,x,z){\rm e}^{-\ii\eta(z)x}=\begin{pmatrix}0\\{\rm e}^{-2\ii\eta(z)x}\end{pmatrix} + O\big(z^{-1}\big),
\end{gather*}
and for fixed $x$ as $z\to\infty$, $z\in\mC_-$, it has the asymptotics
\[
W^-[1](t,x,z){\rm e}^{-\ii\eta(z)x}=\begin{pmatrix}1\\0\end{pmatrix} + O\big(z^{-1}\big),\qquad W^-[2](t,x,z){\rm e}^{-\ii\eta(z)x}=O\big(z^{-1}\big).
\]
It can be proved similarly that $W^+(t,x,z){\rm e}^{\ii\eta(z)x}$ is bounded in $z\in\mC_+^{\rm cl}$, and for fixed $x$ as $z\to\infty$, $z\in\mC_+$, it has the asymptotic behavior
\begin{gather*}
W^+[1](t,x,z){\rm e}^{\ii\eta(z)x}= \begin{pmatrix}{\rm e}^{2\ii\eta(z)x}\\0\end{pmatrix}+O\big(z^{-1}\big)= O\big(z^{-1}\big),\\
W^+[2](t,x,z){\rm e}^{\ii\eta(z)x}= \begin{pmatrix}0\\1\end{pmatrix}+O\big(z^{-1}\big).
\end{gather*}
Note that if the symbol $O(\cdot)$ is present in a matrix expression, then it denotes a matrix of the appropriate size whose entries have the indicated order.

The functions $\varkappa(z)$~\eqref{varkappa} and $w(z)$~\eqref{w-funct} are analytic in $\mC\setminus [E,\overline{E}]$, and, hence, the function $M_0(z)=\left(\begin{smallmatrix} a(z) & b(z) \\ b(z) & a(z) \end{smallmatrix}\right)$ defined by $a(z)$ and $b(z)$ of the form~\eqref{ab} is also analytic in $\mC\setminus \big[E,\overline{E}\big]$. Note that since $\varkappa (z)\to 1$ and $\varkappa^{-1}(z)\to 1$ as $z\to \infty$, then $M_0(z)\to I$ as $z\to \infty$. Thus, the function $\Phi(t,\la)$~\eqref{Phi_AKNS} has the analytic continuation $\Phi(t,z)$ for $z\in\mC\setminus\big[E,\overline{E}\big]$ (the orientation on $[E,\overline{E}]$ is chosen from up to down) which is continuous up to the boundary and has the singularities of the type $(z-E)^{-1/4}$, \smash{$\big(z-\overline{E}\big)^{-1/4}$} at the points $E$ and $\overline{E}$.
Since ${\det M_0(z)\equiv 1}$, then ${\det \Phi(t,z)\equiv 1}$. Taking into account that $w(z)=z-\Re E+O\big(z^{-1}\big)$, $z\to\infty$, for each fixed~$t$ we obtain the asymptotics:
\[
\Phi[1](t,z) {\rm e}^{\ii zt}=\begin{pmatrix} 1 \\ 0\end{pmatrix}+O\big(z^{-1}\big),\quad \Phi[2](t,z) {\rm e}^{-\ii zt}=\begin{pmatrix} 0 \\ 1\end{pmatrix}+O\big(z^{-1}\big),\qquad z\to\infty,\ z\in\mC.
\]

It follows from the above that the functions $Y_\pm(t,x,\la)$ have the analytic continuations $Y^\pm(t,x,z)$ for $z\in\mC_\pm^{\rm cl}\setminus\big([-\Lambda,\Lambda]\cup\big[E,\overline{E}\big]\big)$, where $\mC_\pm^{\rm cl}=\mC_\pm\cup\R$ ($Y^+(t,x,z)=Y^-(t,x,z)$ for $\la\in\R\setminus[-\Lambda,\Lambda]$, $\Lambda<\infty$) which are continuous up to $([-\Lambda,\Lambda]\setminus\{\Re E\})\cup[E,\Re E)$ and $([-\Lambda,\Lambda]\setminus\{\Re E\})\cup\big(\Re E,\overline{E}\big]$ respectively, bounded in the neighborhood of the point $\Re E$, have the singularities of the type $(z-E)^{-1/4}$ and $(z-\overline{E})^{-1/4}$ respectively, and in addition, for fixed $t$ and $x$,
\begin{gather*}
Y^-[1](t,x,z) {\rm e}^{\ii (zt-\eta(z)x)}=W^-(t,x,z) {\rm e}^{-\ii\eta(z)x} \Phi[1](t,z) {\rm e}^{\ii zt} = \begin{pmatrix} 1\\ 0\end{pmatrix}+ O\big(z^{-1}\big)
\end{gather*}
as $z\to\infty$, $z\in\mC_-^{\rm cl}$, and
\begin{gather*}
Y^+[2](t,x,z) {\rm e}^{-\ii (zt-\eta(z)x)}=W^+(t,x,z) {\rm e}^{\ii\eta(z)x} \Phi[2](t,z) {\rm e}^{-\ii zt} = \begin{pmatrix} 0\\ 1\end{pmatrix}+ O\big(z^{-1}\big)
\end{gather*}
as $z\to\infty$, $z\in\mC_+^{\rm cl}$.
Thus, the proof of~\ref{AsymptY} is completed.

To prove that there exists a solution $\Psi(t,x,\la)$ of the $t$-equation~\eqref{teq} (for each $x$) satisfying the initial condition ${\Psi(0,x,\la)\equiv I}$, we prove that there exists a solution $\widehat{\Psi}(t,x,\la)=\Psi(t,x,\la){\rm e}^{\ii\la t\sigma_3}$ of the equivalent integral equation
\begin{equation}\label{hatPsi}
 \widehat{\Psi}(t,x,\la)=I-\int_0^t {\rm e}^{-\ii \la (t-\tau)\sigma_3} H(\tau,x)\widehat{\Psi}(\tau,x,\la) {\rm e}^{\ii \la (t-\tau)\sigma_3}\dd\tau.
\end{equation}
We represent the solution of~\eqref{hatPsi} as the Neumann series $\widehat{\Psi}=\sum_{n=0}^\infty \widehat{\Psi}_n$, where ${\widehat{\Psi}_0(t,x,\la)\equiv I}$, $\widehat{\Psi}_n(t,x,\la)=-\int_0^t {\rm e}^{-\ii \la(t-\tau)\sigma_3} H(\tau,x)\widehat{\Psi}_{n-1}(\tau,x,\la) {\rm e}^{\ii \la (t-\tau)\sigma_3}\, \dd\tau$, and prove that the series converges. Since $\big|{\rm e}^{\pm\ii\la(t-\tau)}\big|\le 1$, then $\big\|\widehat{\Psi}_n(t,x,\la)\big\| \le \int_0^t \|H(\tau,x)\| \big\|\widehat{\Psi}_{n-1}(\tau,x,\la)\big\|\dd\tau$.
Hence, $\big\|\widehat{\Psi}_1(t,x,\la)\big\|\le \int_0^t \|H(\tau,x)\|\dd\tau =:M(t,x)$, and by the MMI we obtain $\big\|\widehat{\Psi}_n(t,x,\la)\big\|\le \frac{M^n(t,x)}{n!}$. Consequently, $\big\|\widehat{\Psi}(t,x,\la)\big\|\le \sum_{n=0}^\infty \big\|\widehat{\Psi}_n(t,x,\la)\big\|\le {\rm e}^{M(t,x)}$.
Since (due to the properties of $\mathcal E(t,x)$) $M(t,x)$ is bounded on any compact set in $\R_+\times\R_+$, then the series $\sum_{n=0}^\infty \widehat{\Psi}_n$ converges absolutely and uniformly in $\la$ on $\R$ and in $(t,x)$ on any compact set in $\R_+\times\R_+$. Since the functions $\widehat{\Psi}_n(t,x,\la)$ are continuous and the series converges locally uniformly on $\R_+\times\R_+\times\R$, then the sum of the series $\widehat{\Psi}(t,x,\la)$ is continuous on $\R_+\times\R_+\times\R$, and since $\widehat{\Psi}(t,x,\la)$ is a solution of the integral equation~\eqref{hatPsi}, then it is continuously differentiable in $t$. Obviously, $\widehat{\Psi}_n(t,x,\la)$ are continuously differentiable in $x$. It is easily verified that the series $\sum_{n=0}^\infty \big(\widehat{\Psi}_n\big)_x(t,x,\la)$ converges absolutely and locally uniformly on $\R_+\times\R_+\times\R$, and hence $\widehat{\Psi}(t,x,\la)$ is continuously differentiable in $x$.
The function $\Psi(t,x,\la)=\widehat{\Psi}(t,x,\la) {\rm e}^{-\ii\la t\sigma_3}$ has the analytic continuation $\Psi(t,x,z)$ for $z\in\mC$. Taking into account that $\widehat{\Psi}(t,x,\la)$ has the integral representation~\eqref{hatPsi} and that ${\rm e}^{\pm\ii z t}\to 0$ for $z\in\mC_\pm$ (respectively) as $z\to\infty$ ($t\in\R_+$), we obtain that the functions $\Psi(t,x,z) {\rm e}^{\ii zt}$ and $\Psi(t,x,z){\rm e}^{-\ii zt}$ are bounded in $z\in\mC_+^{\rm cl}$ and $z\in\mC_-^{\rm cl}$, respectively, and have the asymptotics
\begin{alignat*}{3}
& \Psi(t,x,z){\rm e}^{\ii z t}=\begin{pmatrix}1&0 \\ 0&{\rm e}^{2\ii zt} \end{pmatrix}+O\big(z^{-1}\big),\qquad && z\to\infty,\ z\in\mC_+^{\rm cl},&
\\
& \Psi(t,x,z){\rm e}^{-\ii zt}=\begin{pmatrix}{\rm e}^{-2\ii z t}&0\\0&1\end{pmatrix}
+O\big(z^{-1}\big),\qquad&&  z\to\infty, \ z\in\mC_-^{\rm cl}.&
\end{alignat*}
Similar results can be obtained by using the integral representation for $\Psi(t,x,\la)$ by a transformation operator (cf.~\cite{Filipkovska/Kotlyarov/Melamedova}).

Since $H(0,x)\equiv 0$ and $F(0,x,\la)\equiv -\sigma_3$, then for $t=0$ the $x^\pm$-equations~\eqref{pmxeq} take the form $w_x=\ii\eta_\pm(\la)\sigma_3$ for $\la\in\R$, where $\eta_\pm(\la)=\eta(\la)$ for $\la\in\R\setminus[-\Lambda,\Lambda]$ with ${0<\Lambda<\infty}$. These equations have the solutions $w_\pm(x,\la)={\rm e}^{\ii\eta_\pm(\la)x\sigma_3}$ satisfying the initial conditions
\[
\lim_{x\to\infty}w_\pm(x,\la){\rm e}^{-\ii\eta_\pm(\la)x\sigma_3}=I
\] (recall that we prove the theorem, assuming that $L=+\infty$ and $x\in (0,L)$). Obviously, $\det w_\pm(x,\la)\equiv 1$. Due to the properties of $\eta(z)$, there exists the function (analytic continuation) $w(x,z)={\rm e}^{\ii\eta(z)x\sigma_3}$, analytic in $z$ on $\mC\setminus[-\Lambda,\Lambda]$ and continuous up to $[-\Lambda,\Lambda]$, which satisfies the $x^\pm$-equations~\eqref{pmxeq} for $z=\la\in\R\setminus[-\Lambda,\Lambda]$ and have the boundary values $w_\pm(x,\la)$ satisfying~\eqref{pmxeq} for $\la\in[-\Lambda,\Lambda]$, and in addition,
\[
w[1](x,z){\rm e}^{-\ii\eta(z)x}=\begin{pmatrix}1\\0\end{pmatrix},\qquad w[2](x,z){\rm e}^{\ii\eta(z)x}=\begin{pmatrix}0\\1\end{pmatrix}.
\]

It follows from the above that the functions $Z_\pm(t,x,\la)$ have the analytic continuations $Z^\pm(t,x,z)$ for $z\in\mC_\pm^{\rm cl}\setminus[-\Lambda,\Lambda]$ respectively ($Z^+(t,x,z)=Z^-(t,x,z)$ for $\la\in\R\setminus[-\Lambda,\Lambda]$, $\Lambda<\infty$), which are continuous up to $[-\Lambda,\Lambda]$,
\begin{gather*}
Z^+[1](t,x,z){\rm e}^{\ii(zt-\eta(z)x)}=\Psi(t,x,z) {\rm e}^{\ii zt} w[1](x,z) {\rm e}^{-\ii\eta(z)x} = \begin{pmatrix}1\\0\end{pmatrix}+ O\big(z^{-1}\big)
\end{gather*}
as $z\to\infty$, $z\in\mC_+^{\rm cl}$, and
\begin{gather*}
Z^-[2](t,x,z){\rm e}^{-\ii(zt-\eta(z)x)}=\Psi(t,x,z) {\rm e}^{-\ii zt} w[2](x,z) {\rm e}^{\ii\eta(z)x} = \begin{pmatrix}0\\1\end{pmatrix}+ O\big(z^{-1}\big)
\end{gather*}
as $z\to\infty$, $z\in \mC_-^{\rm cl}$.

Thus, the existence of the solutions $Z_\pm(t,x,\la)$~\eqref{Z}, property~\ref{Differ} for them, and property~\ref{AsymptZ} have been proved.

Finally, we will prove~\ref{UnitDeterminant}. Since $Y_\pm(t,x,\la)$ and $Z_\pm(t,x,\la)$ are solutions of the compatible systems~\eqref{teq},~\eqref{pmxeq}, then $\det Y_\pm(t,x,\la)\equiv 1$ and $\det Z_\pm(t,x,\la)\equiv 1$. Indeed, let $P_\pm(t,x,\la)$ be arbitrary solutions of the compatible systems~\eqref{teq},~\eqref{pmxeq}, respectively. Then their determinants $\det P_\pm(t,x,\lambda)$ satisfy the systems
\begin{equation}\label{UVdet}
(\det P_\pm)_t=\tr U(t,x,\lambda)\det P_\pm,\qquad
(\det P_\pm)_x=\tr V_\pm(t,x,\lambda)\det P_\pm.
\end{equation}
Since $\tr U(t,x,\lambda)= \tr V_\pm(t,x,\lambda)=0$ for all $(t,x,\lambda)$, then it follows from~\eqref{UVdet} that the determinants $\det P_\pm(t,x,\lambda)$ do not depend on $t$, $x$. Hence, $\det P_\pm(t,x,\lambda)=\det P_\pm(t_0,x_0,\lambda)$ for~any~$(t,x,\lambda)$ and $(t_0,x_0,\lambda)$. Thus, we obtain that
 \begin{gather*}
\det Y_\pm(t,x,\la)=\det Y_\pm(0,0,\la)=
\det W_\pm(0,0,\la) \det\Phi(0,\la)= 1,\\
\det Z_\pm(t,x,\la)=\det Z_\pm(0,0,\la)=\det \Psi(0,0,\la) \det w_\pm(0,\la)= 1
\end{gather*}
 for all $(t,x,\lambda)$.
 \end{proof}

 \section{Formulation of the Riemann--Hilbert problem}\label{RHP0}

Since $Y_\pm$~\eqref{Y} and $Z_\pm$~\eqref{Z} are the solutions of the compatible AKNS systems \big(the $t$- and $x^\pm$-equations\big)~\eqref{teq},~\eqref{pmxeq}, they are linearly dependent, namely, they satisfy the ``scattering relations''
\begin{equation}\label{scat}
Y_\pm(t,x,\la)=Z_\pm(t,x,\la) T_\pm(\la), \qquad \la\in\R.
\end{equation}
Since the transition matrices (scattering matrices) $T_\pm(\la)$ are independent of $t$ and $x$, they can be presented in the form \big(as shown above, $w_\pm(x,\la)={\rm e}^{\ii\eta_\pm(\la)x\sigma_3}$\big)
\[
T_\pm(\la)=(Z_\pm(0,0,\la))^{-1}Y_\pm(0,0,\la)=
(w_\pm(0,\la))^{-1}\Phi(0,\la)=\Phi(0,\la)=M_0(\la).
\]
Hence,
\[
T_\pm(\la)=T(\la)= \begin{pmatrix}
a(\la) & b(\la) \\ b(\la) & a(\la)
\end{pmatrix}=\frac{1}{2} \begin{pmatrix}
\varkappa (\la)+\varkappa^{-1}(\la) & \varkappa (\la)-\varkappa^{-1}(\la) \\
\varkappa (\la)-\varkappa^{-1}(\la) & \varkappa (\la)+\varkappa^{-1}(\la) \end{pmatrix},
\]
where $a(\la)$, $b(\la)$ and $\varkappa(\la)$ are defined by~\eqref{ab} and~\eqref{varkappa}.

The spectral functions $a(\la)$, $b(\la)$ have the analytic continuations $a(z)$, $b(z)$ to $\mC\setminus[E,\overline{E}]$ which have the following asymptotics at infinity:
\begin{align*}
& a(z)= 1+O\big(z^{-2}\big),\qquad z\to\infty,\ z\in \mC,	\qquad
 b(z)= O\big(z^{-1}\big),\qquad z\to\infty,\ z\in \mC.	
\end{align*}
It was assumed that the interval $\big[E,\overline{E}\big]$ is oriented downward. For $z\in\big(E,\overline{E}\big)$ the functions $a(z)$, $b(z)$ and $\varkappa (z)$ satisfy the relations
\[
\varkappa_-(z)=\ii\varkappa_+(z),\qquad
a_-(z)=\ii b_+(z),\qquad
b_-(z)=\ii a_+(z),
\]
where the signs ``$-$'' and ``$+$'' in subscripts denote nontangential boundary values from the left and right of an oriented contour, and in what follows this notation is also used. It is easy to see that $a(z)$ has no zeros and, hence, a discrete spectrum is empty.

The ``reflection coefficient''
\begin{equation}\label{r}
r(z):=\frac{b(z)}{a(z)}=\frac{\varkappa (z)-\varkappa^{-1}(z)}{\varkappa (z)+\varkappa^{-1}(z)}=\frac{z-\Re E -w(z)}{\ii \Im E}= \frac{\ii \big(w(z)-z+\Re E\big)}{\Im E},
\end{equation}
where $w(z)=\sqrt{(z-E)\big(z-\overline{E}\big)}$~\eqref{w-funct}, is analytic in $z\in\mC\setminus\big[E,\overline{E}\big]$, satisfies the symmetry relation $\overline{r(\bar{z})}=-r(z)$, and
\[
r(z)=O\big(z^{-1}\big),\qquad z\to\infty.
\]
Since $w(\la)\in C^\infty(\R)$, then $r(\la)$ is infinitely differentiable as a function of the real variable $\la$. Note that if we, as above, take the cut along an arc $\gamma$ connecting $E$ and $\overline{E}$ via infinity (on the extended complex plane $\mC\cup\{\infty\}$) without intersection of the real axis and fix the branch of $w(z)$ by the condition $w(0)=|E|$ or $w(\Re E)=\Im E$, then $r(z)$ is analytic on $\mC\setminus\gamma$, including $\R$. The same holds for $a(z)$ and $b(z)$, where we fix the branch of $\varkappa(z)$ by the condition $\varkappa(\Re E)= {\rm e}^{\ii\pi/4}$.

Denote $\phi=\arg\big(\la-\overline{E}\big)$, $\la\in\R$, then $\varkappa^4(\la)=\frac{\la-\overline{E}}{\la-E}= {\rm e}^{\ii 2\phi}$, where $-\pi<2\phi\le\pi$, and $\varkappa(\la)= {\rm e}^{\ii\phi/2}$. Since $r(\la)=\frac{ {\rm e}^{\ii\phi/2}- {\rm e}^{-\ii\phi/2}}{ {\rm e}^{\ii\phi/2}+ {\rm e}^{-\ii\phi/2}}= \ii\tan(\phi/2)$, $\la\in\R$, and $-\pi/2 <\phi\le \pi/2$, then
\[
 1-r^2(\la)=1+|r(\la)|^2=1+\tan^2(\phi/2)=\frac{1}{\cos^2(\phi/2)}\le 2,
\]
and, therefore, $|r(\la)|\le 1$ for $\la\in\R$.

Due to the relations $a_-(z)=\ii b_+(z)$, $b_-(z)=\ii a_+(z)$ for $z\in\big(E,\overline{E}\big)$, the jump of $r(z)$ is
\begin{equation*}
h(z):=r_-(z)-r_+(z)=\frac{a_+^2(z)-b_+(z)^2}{a_+(z)b_+(z)}= \frac{1}{a_+(z)b_+(z)}= -\frac{2\ii w_+(z)}{\Im E},\qquad z\in\big(E,\overline{E}\big).
\end{equation*}
Since for $z=\Re E+\ii\nu\in(E,\overline{E})$,\; $w_\pm(\Re E+\ii\nu)= \lim_{\varepsilon\to 0}\sqrt{(\pm\varepsilon+\ii\nu)^2+\Im^2 E}=\sqrt{\Im^2 E-\nu^2}$, where $\nu=\Im z\in(\Im E,-\Im E)$, then $h(z)=-2\ii \sqrt{\Im^2 E-\nu^2}/\Im E$. Thus, for $z\in\big(E,\overline{E}\big)$ we introduce the function
\begin{equation}\label{h=h+}
h(z)=r_-(z)-r_+(z)=-\frac{2\ii\sqrt{\Im^2 E-\Im^2 z}}{\Im E},
\end{equation}
where $\Im z\in(\Im E,-\Im E)$, $\Re z =\Re E$.
It satisfies the symmetry relation $\overline{h(\bar{z})}=\overline{h(z)}=-h(z)$.

Note that the scattering relations~\eqref{scat} can be rewritten in the form
\begin{align*}
& \frac{Y_\pm[1](t,x,\la)}{a(\la)}=Z_\pm[1](t,x,\la)+ r(\la)Z_\pm[2](t,x,\la),\qquad \la\in\R, \\
& \frac{Y_\pm[2](t,x,\la)}{a(\la)}=r(\la)Z_\pm[1](t,x,\la)+ Z_\pm[2](t,x,\la), \qquad \la\in\R,
\end{align*}
and also, it follows from~\eqref{scat} that
\begin{align*}
&a(\la)=W(Y_\pm[1](t,x,\la),Z_\pm[2](t,x,\la))=W(Z_\pm[1](t,x,\la),Y_\pm[2](t,x,\la)), \\
&b(\la)=W(Z_\pm[1](t,x,\la),Y_\pm[1](t,x,\la))=W(Y_\pm[2](t,x,\la),Z_\pm[2](t,x,\la)),
\end{align*}
where $W(f,g)=\det(f,g)$ is the Wronskian (Wronski determinant) of the solutions $f$, $g$. In addition, $h(z)=\frac{\ii}{a_-(z)a_+(z)}$, $z\in\big(E,\overline{E}\big)$.

Denote
\begin{equation}\label{theta1}
\theta(t,x,z)=zt-\eta(z)x,
\end{equation}
where $\eta(z)$ is defined by~\eqref{eta}, and introduce the contour $\Sigma=\R\cup(E,\overline{E})$ oriented from the left to the right for the real line $\R$ and downward for the interval $(E,\overline{E})$ (see Figure~\ref{Sigma}). The analysis of the phase function $\theta(t,x,z)$~\eqref{theta1} is performed in Appendix~\ref{appendixA}.

\begin{figure}[t]\centering
 \begin{tikzpicture}[smooth,scale=1,thick,font=\small]
\draw[-Stealth, black] (-3.75,0) -- (-2.9,0);
\draw[-Stealth, black] (-3,0) -- (3.75,0);
\draw[-Stealth, black] (-2.25,1.2) -- (-2.25,0.65);
\draw[-Stealth, black] (-2.25,0.75) -- (-2.25,-0.85);
\draw[black] (-2.25,-0.75) -- (-2.25,-1.2);
\draw[color=black, fill=black](-2.25,-1.2) circle(0.03);
\draw[color=black, fill=black](-2.25,1.2) circle(0.03);
\node[black] at (-1.95,1.1) {$E$};
\node[black] at (-1.95,-1.1) {$\overline{E}$};
\node[black] at (3.1,0.3) {$\mathbb{R}$};
 \end{tikzpicture}
 \caption{The oriented contour $\Sigma$.}\label{Sigma}
\end{figure}

Using the scattering relations~\eqref{scat}, the relations obtained for $a(z)$, $b(z)$, $r(z)$ and $h(z)$, and Theorem~\ref{Theorem_SolAKNS}, we obtain the following. \emph{The matrix}
\begin{gather}\label{M}
M(t,x,z)\!=\!\begin{cases}
 \begin{pmatrix}
 Z^+[1](t,x,z){\rm e}^{\ii\theta(t,x,z)} & \displaystyle\frac{Y^+[2](t,x,z)}{a(z)} {\rm e}^{-\ii\theta(t,x,z)}
 \end{pmatrix}\!, & z\in \mC_+\!\!\setminus \! [E,\Re E), \vspace{1mm}\\
 \begin{pmatrix}
\displaystyle \frac{Y^-[1](t,x,z)}{a(z)} {\rm e}^{\ii\theta(t,x,z)} & Z^-[2](t,x,z){\rm e}^{-\ii\theta(t,x,z)}
 \end{pmatrix}\!, & z\in \mC_-\!\!\setminus \! \big(\Re E,\overline{E}\big],
\end{cases}\!\!\!
\end{gather}
\emph{is a solution of the following basic matrix Riemann--Hilbert problem $($the problem of conjugation of boundary values on a contour $\Sigma)$, where the reflection coefficient $r(z)$ and the function $h(z)$ are uniquely determined by the initial and boundary conditions~\eqref{MBini} and~\eqref{MBboundAe}:}

\begin{Theorem*}
Given a contour $\Sigma=\R\cup(E,\overline{E})$ $($see Figure~$\ref{Sigma})$ and functions $r(z)$~\eqref{r} $(z\in\R)$ and $h(z)$~\eqref{h=h+} $(z\in(E,\overline{E}))$, find a $2\times2$ matrix function $M(t,x,z)$ satisfying the following conditions:
\begin{itemize}\itemsep=0pt
	\item $M(t,x,z)$ is analytic in $z\in\mC\setminus \big(\Sigma\cup\big\{E,\overline{E}\big\}\big)$ and has the continuous boundary values $M_\pm(t,x,z)$ in $z\in\Sigma\setminus\{\Re E\}$;

	\item $M(t,x,z)$ is bounded near the point of the contour self-intersection $\Re E$ and has the singularities of the type $(z-E)^{-1/4}$, $\big(z-\overline{E}\big)^{-1/4}$ at the endpoints $E$, $\overline{E}$;

	\item $M(t,x,z)$ satisfies the boundary condition $($which is also called the jump condition$)$
\[
M_-(t,x,z)=M_+(t,x,z)J(t,x,z),\qquad z\in\Sigma\setminus\{\Re E\},
\]
where $J(t,x,z)$ is defined as
\begin{align}
J(t,x,z)=&\begin{pmatrix}
 1 & h(z) {\rm e}^{-2\ii \theta(t,x,z)} \\
 0 & 1	\end{pmatrix},\;&& z\in (E,\Re E), \nonumber\\
 =&\begin{pmatrix}
 1-r^2(z) {\rm e}^{2\ii (\theta_-(t,x,z)-\theta_+(t,x,z))}& -r(z) {\rm e}^{-2\ii\theta_+(t,x,z)} \\
 r(z) {\rm e}^{2\ii \theta_-(t,x,z)} & 1
 \end{pmatrix},\;&& z=\la\in \R, \label{J}\\
 =&\begin{pmatrix}
 1 & 0 \\
 h(z) {\rm e}^{2\ii \theta(t,x,z)} & 1
 \end{pmatrix},\;&& z\in (\Re E,\overline{E});\nonumber
\end{align}

	\item $M(t,x,z)$ satisfies the normalization condition $M(t,x,z)=I+O\big(z^{-1}\big)$, $z\to\infty$.
\end{itemize}
\end{Theorem*}

\begin{Remark}\label{JumpMatrix}
The jump matrix $J(t,x,z)$ for $z=\la\in\R$ takes two different form:
 \begin{alignat*}{3}
&J(t,x,\la) =\begin{pmatrix}
1-r^2(\la) {\rm e}^{-\pi n(\la) x}& -r(\la) {\rm e}^{-2\ii\theta_+(t,x,\la)} \\
r(\la) {\rm e}^{2\ii \theta_-(t,x,\la)} & 1
\end{pmatrix}, \qquad && \la\in[-\Lambda,\Lambda],\quad \Lambda\le\infty, & \\
&\hphantom{J(t,x,\la)}{}
=\begin{pmatrix}
1-r^2(\la)& -r(\la) {\rm e}^{-2\ii\theta(t,x,\la)} \\
r(\la) {\rm e}^{2\ii \theta(t,x,\la)} & 1
\end{pmatrix}, \qquad && \la\in \R\setminus [-\Lambda,\Lambda],\quad \Lambda<\infty,&
 \end{alignat*}
because $\eta_+(\la)=\eta_-(\la)$ and $\theta_+(t,x,\la)=\theta_-(t,x,\la)$
for $\la\notin [-\Lambda,\Lambda]$, and $\theta_-(t,x,\la)-\theta_+(t,x,\la)=(\eta_+(\la)-\eta_-(\la))x$, $\eta_+(\la)-\eta_-(\la)=\pi\ii n(\la)/2$ for $\la\in[-\Lambda,\Lambda]$ (recall that for the case when $\Lambda=+\infty$ we set $[-\Lambda,\Lambda]=\R$).
\end{Remark}

 \section[Solvability of the IBVP for the MB equations and the RH problem]{Solvability of the IBVP for the MB equations\\ and the RH problem}\label{Solv_IBVP_RHP}

 \begin{Theorem}\label{ThRH0}
Let the functions ${\mathcal E}(t,x)$, $\rho(t,x,\la)$ and $\mathcal N(t,x,\la)$ be a solution of the IBVP for the MB equations~\eqref{MB1a} with the initial and boundary conditions~\eqref{MBini} and~\eqref{MBboundAe} and have the properties indicated in Theorem~$\ref{Theorem_SolAKNS}$, and let $a(z)$ and $b(z)$ be the corresponding spectral functions. Then there exists a unique matrix function $M(t,x,z)$ $(M=(M_{ij})_{1\le i,j\le 2})$ which solves the basic RH problem {\rm RH}$_0$, and this matrix function defines the complex electric field envelope ${\mathcal E}(t,x)$ and the density matrix $F(t,x,\la)=\left(\begin{smallmatrix}\mathcal N(t,x,\la) & \rho(t,x,\la)\\ \overline{\rho(t,x,\la)} & -\mathcal N(t,x,\la) \end{smallmatrix}\right)$ by the formulas
\begin{gather}
{\mathcal E}(t,x)=-4\ii \lim_{z\to\infty} \big(z M(t,x,z)\big)_{12},\label{E1}
\\
F(t,x,\la)=- M(t,x,\la+\ii0) {\rm e}^{-\ii\la t\sigma_3}M^{-1}(0,x,\la+\ii0)\sigma_3 M(0,x,\la+\ii0) {\rm e}^{\ii\la t\sigma_3}\nonumber \\
\hphantom{F(t,x,\la)=}{} \times M^{-1}(t,x,\la+\ii0)\label{F1}
\end{gather}
$($more precisely, the function $F(t,x,\la)$ is defined by the formula~\eqref{F1} for $\la\in\R\setminus\{\Re E\}$ and defined appropriately at the point $\la=\Re E$ so that $F(t,x,\la)$ is continuous in $\la\in\R)$.
 \end{Theorem}
 \begin{proof}
The existence of the solution $M(t,x,z)$~\eqref{M} of the basic RH problem {\rm RH}$_0$ was proved in Section~\ref{RHP0}. Let us prove the uniqueness of the solution. First we show that if $M(t,x,z)$ is a~solution of this RH problem, then $\det M(t,x,z)\equiv 1$. It follows from the properties of $M(t,x,z)$ that the function $\det M(t,x,z)$ is analytic for $z\in\mC\setminus\Sigma^{\rm cl}$, where $\Sigma^{\rm cl}=\Sigma\cup \big\{E,\overline{E}\big\}$, and has continuous boundary values on $\Sigma\setminus\{\Re E\}$, and $\det M_-(t,x,z)=\det M_+(t,x,z)\det J(t,x,z)=\det M_+(t,x,z)$, $z\in\Sigma\setminus\{\Re E\}$, since $\det J(t,x,z)\equiv 1$. Hence, $\det M(t,x,z)$ is analytic for all $z\in\mC$ except for the points $E$, $\overline{E}$, $\Re E$ which are removable singularities by virtue of the properties of $M(t,x,z)$. Thus, after appropriate modifications (definitions) at these points, $\det M(t,x,z)$ is analytic for $z\in\mC$, and since $\det M(t,x,z)=1+O\big(z^{-1}\big)$ as ${z\to \infty}$ (by virtue of the normalization condition for $M(t,x,z)$), then by the Liouville theorem $\det M(t,x,z)\equiv 1$. Hence, there exists $M^{-1}(t,x,z)$ and it is analytic on $\mC\setminus\Sigma^{\rm cl}$ and has continuous boundary values on $\Sigma\setminus\{\Re E\}$. Now suppose that there is another solution $\hat{M}(t,x,z)$ of the basic RH problem {\rm RH}$_0$. Consider the function $N(t,x,z):=\hat{M}(t,x,z)M^{-1}(t,x,z)$ (cf. \cite[p.~189]{Deift}). Note that $N_-(t,x,z)=\hat{M}_+(t,x,z)J(t,x,z)J^{-1}(t,x,z)M_+^{-1}(t,x,z)=N_+(t,x,z)$, $z\in\Sigma\setminus\{\Re E\}$, and $N(t,x,z)\to I$ as $z\to \infty$. Using the same arguments as for the proof that $\det M(t,x,z)\equiv 1$, we obtain that $N(t,x,z)\equiv I$. Consequently, $\hat{M}(t,x,z)\equiv M(t,x,z)$, which proves the uniqueness of the solution.

Since $\mathcal E$, $\rho$ and $\mathcal N$ is a solution of the IBVP for the MB equations, then the matrix function
\begin{equation}\label{Phi}
\Phi(t,x,z)=M(t,x,z) {\rm e}^{-\ii\theta(t,x,z)\sigma_3},
\end{equation}
where $\theta(t,x,z)$ is defined in~\eqref{theta1}, satisfies the $t$-equation~\eqref{teq}, namely, $\Phi_t=-(\ii z\sigma_3+H(t,x))\Phi$, $z\in\mC\setminus\Sigma^{\rm cl}$, and hence the matrix $M(t,x,z)$ satisfies the equation $M_t=-\ii z[\sigma_3,M]-HM$. Taking into account that $M(t,x,z)=I+O\big(z^{-1}\big)$ as $z\to\infty$, we can represent $M(t,x,z)$ as $M(t,x,z)=I+\frac{m(t,x)}{z}+o\big(z^{-1}\big)$, $z\to\infty$, where $m(t,x)=\lim_{z\to\infty}z(M(t,x,z)-I)$. Substituting this expansion into the equation, we obtain that the equality $H(t,x)+\ii[\sigma_3,m(t,x)]=0$ holds. Hence, ${\mathcal E}(t,x)=-4\ii m_{12}(t,x)=-4\ii \lim_{z\to\infty} \big(z M(t,x,z)\big)_{12}$ (as above, the subindex denotes $(1,2)$-entry of the corresponding matrix).

Further, since ${\mathcal E}$, $\rho$, $\mathcal N$ is a solution of the MB equations~\eqref{MB1a}, then the equation~\eqref{HFeq_2} is satisfied, i.e., $F_t+[\ii\la\sigma_3+H,F]=0$. This equation can be represented in the form
 \begin{equation}\label{HFeq_2.1}
F_t=[U,F],
 \end{equation}
where $U(t,x,\la)$ is the matrix function~\eqref{U-MB} (note that since $U^*=-U$, the equation~\eqref{HFeq_2.1} can be represented as $F_t=U F + F U^*$). The initial condition~\eqref{MBini} for the MB equations implies $F(0,x,\la)\equiv -\sigma_3$. Furthermore, for $z=\la+\ii 0$ ($\la\ne \Re E$) the matrix function $\Phi(t,x,z)$~\eqref{Phi} satisfies the $t$-equation~\eqref{teq}, i.e., ${W_t=U W}$, with the same matrix function $U(t,x,\la)$, and $\det\Phi(t,x,z)=\det M(t,x,z) \equiv 1$. Consequently, the matrix function (the Cauchy matrix) $W(t,x,\la)=\Phi(t,x,\la+\ii0)\Phi^{-1}(0,x,\la+\ii0)$ is a solution of the $t$-equation~\eqref{teq} and satisfies the initial condition $W(0,x,\la)\equiv I$.
Then the function $F(t,x,\la)=-W(t,x,\la)\sigma_3W^{-1}(t,x,\la)$ is a unique solution of the initial value problem~\eqref{HFeq_2.1}, $F(0,x,\la)\equiv -\sigma_3$. Since $W(t,x,\la)= M(t,x,\la+\ii0) {\rm e}^{-\ii\la t\sigma_3} M^{-1}(0,x,\la+\ii0)$, then $F(t,x,\la)$ has the form~\eqref{F1}. Due to the properties of $\Phi(t,x,z)$, the function $F(t,x,\la)$ is continuous in $\la\in\R\setminus\{\Re E\}$ and bounded in the neighborhood of $\la=\Re E$. Thus, we can define it appropriately at this point so that $F(t,x,\la)$ becomes continuous in $\la$ on $\R$. Namely, define the matrix $M(t,x,z)$~\eqref{M} for $z=\Re E+\ii0$ as
\begin{gather*}
 M(t,x,\Re E+\ii0)=M_+(t,x,\Re E) \\
 \qquad := \begin{pmatrix} Z^+[1](t,x,\Re E+\ii0){\rm e}^{\ii\theta_+(t,x,\Re E)} & \displaystyle\frac{Y^+[2](t,x,\Re E+\ii0)}{a(\Re E)} {\rm e}^{-\ii\theta_+(t,x,\Re E)} \end{pmatrix},
\end{gather*}
where $Z^+[1](t,x,\Re E+\ii0)=Z_+[1](t,x,\Re E)$ and $Y^+[2](t,x,\Re E+\ii0)=Y_+[2](t,x,\Re E)$ ($Z_+(t,x,\la)$~\eqref{Z} and $Y_+(t,x,\la)$~\eqref{Y}, $\la\in\R$, are the solution of the $t$- and $x^\pm$-equations), and $a(\Re E)=\big(\varkappa (\Re E)+\varkappa^{-1}(\Re E)\big)/2$, $\varkappa(\Re E)={\rm e}^{\ii\pi/4}$ (in this case $r(\Re E)=\ii$, and $r(\la)$ is infinitely differentiable as a function of the real variable $\la$). Then $F(t,x,\la)$ defined by~\eqref{F1} is continuous in $\la\in\R$.
 \end{proof}

 \begin{Theorem}\label{causality-region}
For $t\le x$ and $t, x\ge 0$, a solution of the IBVP~\eqref{MB1a}--\eqref{MBboundAe} is trivial, i.e., ${\mathcal E(t,x)\equiv 0}$, $\rho(t,x,\la)\equiv 0$, and $\mathcal N(t,x,\la)\equiv -1$.
 \end{Theorem}
 \begin{proof}
First, note that
\[
\sgn[\Re(\ii \theta(t,x,z))]=\sgn\Im z
\]
for $0\le t<x$ and $t=x$, $t,x>0$, and $\theta(0,0,z)\equiv 0$ (see Appendix~\ref{appendixA}). The analyticity of $r(z)$ outside the interval $[E,\overline{E}]$ allows one to apply the following transformation:
$M^{(1)}(t,x,z)=M(t,x,z) G^{(1)}(t,x,z)$, where
\begin{equation}\label{G(1)_t-le-x}
G^{(1)}(t,x,z)=\begin{cases}
\begin{pmatrix}
1 & -r(z) {\rm e}^{-2\ii\theta(t,x,z)} \\
0 & 1 \end{pmatrix}, & z\in\mC_+\setminus[E,\Re{E}),\vspace{1mm}\\
\begin{pmatrix}
1 & 0 \\
-r(z) {\rm e}^{2\ii\theta(t,x,z)} & 1 \end{pmatrix}, & z\in\mC_-\setminus(\Re{E},\overline{E}].
\end{cases}
\end{equation}
Then we obtain the equivalent RH problem
\[
M^{(1)}_-(t,x,z)=M^{(1)}_+(t,x,z)J^{(1)}(t,x,z),\qquad z\in\Sigma\setminus\{\Re E\},
\]
with the jump matrix
\[
J^{(1)}(t,x,z)= \big(G^{(1)}_+(t,x,z)\big)^{-1}J(t,x,z)G^{(1)}_-(t,x,z),\qquad z\in\Sigma,
\]
where $G^{(1)}_\pm(t,x,z)$ are nontangential boundary values of $G^{(1)}(t,x,z)$ on the left and right of the contour $\Sigma$ respectively.
Due to the factorization \smash{$J(t,x,z)=G^{(1)}_+(t,x,z)\big(G^{(1)}_-(t,x,z)\big)^{-1}$} of the jump matrix~\eqref{J} and the jump relations $r_-(z)- r_+(z)= h(z)$ for $z\in(E,\overline{E})$, the jump matrix $J^{(1)}(t,x,z)\equiv I$ for $z\in\Sigma$. Since $G^{(1)}(t,x,z)=I+O\big(z^{-1}\big)$ as $z\to \infty$, $t\le x$, and the points~$E$,~$\Re E$,~$\overline{E}$ are removable singularities of $M^{(1)}(t,x,z)$, then $M^{(1)}(t,x,z) \equiv I$ in view of the Liouville theorem.
Hence, $m^{(1)}(t,x)= \lim_{z\to\infty} z\big(M^{(1)}(t,x,z)-I\big)\equiv 0$. Since $\big(G^{(1)}(t,x,z)\big)^{-1}\to I$ as $z\to\infty$, then
\[
{\mathcal E}(t,x)=-4\ii \lim_{z\to\infty}(z M(t,x,z))_{12}= -4\ii\lim_{z\to\infty} \big(zM^{(1)}(t,x,z)\big(G^{(1)}(t,x,z)\big)^{-1}\big)_{12}\equiv 0.
\]
Since $M(t,x,\la+\ii 0)\equiv I$, then $F(t,x,\la)\equiv-\sigma_3$ (see~\eqref{F1}) and hence ${\rho(t,x,\la)\equiv 0}$ and ${\mathcal N(t,x,\la)\equiv -1}$.
 \end{proof}

Note that, since by Theorem~\ref{causality-region} in the ``causality region'' $t\le x$ ($t,x\ge 0$) the solution of the IBVP~\eqref{MB1a}--\eqref{MBboundAe} is trivial, the basic RH problem RH$_0$ associated with the IBVP~\eqref{MB1a}--\eqref{MBboundAe} provides the causality principle.

Below, we will prove that the basic problem RH$_0$ 
has a unique solution $M(t,x,z)$, infinitely differentiable in $t$, $x$, and that the matrix function $M(t,x,z)$ generates solutions of the AKNS systems~\eqref{teq}, ~\eqref{pmxeq} and, as a result, a solution of the MB equations, and this solution satisfied the initial and boundary conditions. Also, an integral representation of the solution of the MB equations through the solution of a singular integral equation will be given.

 \begin{Theorem}\label{Ex-Uniq}
For each $t\in\R_+$ and each $x\in (0,L)$, $L\le\infty$, the basic problem RH$_0$ has a~unique solution $M(t,x,z)$, and this solution is continuous in $(t,x)\in\R_+\times(0,L)$.
 \end{Theorem}
 \begin{proof}
Let $t$ and $x$ be arbitrary fixed elements from $\R_+$ and from $(0,L)$ respectively. The boundary condition $M_-(t,x,z)=M_+(t,x,z)J(t,x,z)$ on $\Sigma\setminus\{\Re E\}$ can be rewritten in the form of the jump condition $M_+(t,x,z)-M_-(t,x,z)=M_+(t,x,z)[I-J(t,x,z)]$, $z\in\Sigma\setminus\{\Re E\}$. Since $M(\infty)=I$, then a solution of the basic problem RH$_0$, if it exists, has the form
\begin{equation}\label{sol_M}
 M(t,x,z) =\frac{1}{2\pi\ii}\int_\Sigma \frac{M_+(t,x,s)[I-J(t,x,s)]}{s-z}\, \dd s +I,\qquad z\in \mC\setminus\Sigma.
\end{equation}
Passing to the limit in the equality~\eqref{sol_M} as $z\to\zeta_+$, where ${\zeta_+=\lim_{z\to \zeta, z\in + \text{side}}z}$, ${\zeta\in\Sigma}$ and ``$+$~side'' denotes a positive side (a boundary value on the left) of the oriented contour $\Sigma$ (we exclude the points of self-intersection of the contour), and carrying out elementary transformations, we obtain the following singular integral equation for $M_+(t,x,\zeta)$:
\begin{gather*}
 M_+(t,x,\zeta)-I= \frac{1}{2\pi\ii} \int_\Sigma \frac{[M_+(t,x,s)-I][I-J(t,x,s)]}{s-\zeta_+}\, \dd s \\
 \hphantom{M_+(t,x,\zeta)-I=}{} + \frac{1}{2\pi\ii} \int_\Sigma \frac{I-J(t,x,s)}{s-\zeta_+}\, \dd s, \qquad \zeta\in\Sigma.
\end{gather*}
Denote $W(t,x,\zeta):=M_+(t,x,\zeta)-I$ and introduce
the singular integral operator $\mathbf{K}$ and operator function $\Pi$ defined by
\begin{gather*}
 \mathbf{K}[W](t,x,\zeta):= \frac{1}{2\pi\ii} \int_\Sigma \frac{W(t,x,s)[I-J(t,x,s)]}{s-\zeta_+}\, \dd s,
\\
 \Pi(t,x,\zeta):= \frac{1}{2\pi\ii} \int_\Sigma \frac{I-J(t,x,s)}{s-\zeta_+}\, \dd s.
\end{gather*}
We will prove that for any fixed $t$, $x$ there exists a unique solution $W(t,x,\zeta)$ of the singular integral equation (cf. \cite{Kotlyarov13})
\begin{equation}\label{IntEq-RH0}
 W(t,x,\zeta)-\mathbf{K}[W](t,x,\zeta) =\Pi(t,x,\zeta),\qquad \zeta\in\Sigma,
\end{equation}
which is considered in the space of operator ($2\times2$ matrix) functions $W(t,x,\zeta)=:W(\zeta)\in L^2(\Sigma)$.
The operator $ C_\pm[f](\zeta):=\frac{1}{2\pi\ii} \int_\Sigma\frac{f(s)}{s-\zeta_\pm}\, \dd s\in B\big(L^2(\Sigma)\big)$, where $B\big(L^2(\Sigma)\big)$ is the (Banach) space of bounded linear operators from $L^2(\Sigma)$ into $L^2(\Sigma)$, since $\Sigma$ is a Carleson curve (see, e.g., \cite{Spitkovskii} and references therein, or \cite{Lenells17}).
The function $I-J(t,x,s)$ belongs to $L^2(\Sigma)$ with respect to $s$. Thus, $\Pi(t,x,\zeta)\in L^2(\Sigma)$ as a function of the variable $\zeta$. Note that $I-J(t,x,s)\in L^{\infty}(\Sigma)$ in $s$. Since $\Re J(t,x,z)$ is positive definite for $z\in \R$, $\det J(t,x,z)\equiv 1$, the contour $\Sigma$ is symmetric with respect to the real axis $\R$ and $J^{-1}(t,x,z)=J^*(t,x,\bar{z})$ for $z\in \Sigma\setminus \R$ \big(thus, the contour~$\Sigma$ and the matrix $J^{-1}(t,x,z)$ satisfy the Schwarz reflection principle\big), then it follows from \cite[Theorem~9.3, p.~984]{Zhou} that there exists the operator $(I-\mathbf{K})^{-1}\in B\big(L^2(\Sigma)\big)$.
Consequently, the singular integral equation~\eqref{IntEq-RH0} has a unique solution $W(t,x,\zeta)$ for any fixed $t$ and $x$, which belongs to $L^2(\Sigma)$ with respect to $\zeta$. Therefore, there exists the solution $M(t,x,z)$ of the basic problem RH$_0$, which can be obtained by the formula
\begin{equation}\label{sol_RH0}
 M(t,x,z) =\frac{1}{2\pi\ii}\int_\Sigma \frac{[W(t,x,s)+I][I-J(t,x,s)]}{s-z}\, \dd s +I,\qquad z\in \mC\setminus\Sigma.
\end{equation}
Since the operator $I-\mathbf{K}\in B\big(L^2(\Sigma)\big)$ and the operator function $\Pi\in L^2(\Sigma)$ depends continuously on the parameters $(t,x)\in\R_+\times(0,L)$, then the inverse operator $(I-\mathbf{K})^{-1}$ and the solution $W(t,x,\zeta)$ of the equation~\eqref{IntEq-RH0} also depend continuously on $(t,x)$. Taking into account the form~\eqref{sol_RH0} of the solution $M(t,x,z)$, we obtain that $M(t,x,z)$ is continuous in $(t,x)\in\R_+\times(0,L)$.

The proof of the uniqueness of the solution $M(t,x,z)$ is carried out in the same way as for Theorem~\ref{ThRH0}.
\end{proof}

 \begin{Theorem}\label{smoothnessRH}
For any $(t,x)\in\R_+\times(0,L)$, $L\le\infty$, a solution of the basic problem RH$_0$ is infinitely differentiable in $t$ and $x$.
 \end{Theorem}
 \begin{proof}
By Theorem~\ref{Ex-Uniq}, the basic problem RH$_0$ has a unique solution $M(t,x,z)$ which can be obtained by the formula~\eqref{sol_RH0}.

In order for the solution $M(t,x,z)$, as well as the solution $W(t,x,\zeta)$ of the singular integral equation~\eqref{IntEq-RH0}, to be differentiable in $t$ and $x$, it is necessary that the integrals that arise when differentiating the solution $M(t,x,z)$ and the equation~\eqref{IntEq-RH0} be convergent. The matrix $J(t,x,z)$ and its derivatives in $t$ and $x$ are responsible for the decrease of the integrands.

On unbounded parts of the contour $\Sigma$, i.e., for $z\in\R$ with $|z|\gg 1$ ($|z|$ is sufficiently large), the jump matrix has the form (see Remark~\ref{JumpMatrix})
\[
J(t,x,z)=\begin{pmatrix} 1-r^2(z)& -r(z) {\rm e}^{-2\ii\theta(t,x,z)} \\
r(z) {\rm e}^{2\ii \theta(t,x,z)} & 1 \end{pmatrix}
\]
if $\Lambda<\infty$ ($z\in\R\setminus [-\Lambda,\Lambda]$), and
\[
J(t,x,z)=\begin{pmatrix} 1-r^2(z) {\rm e}^{-\pi n(z) x}& -r(z) {\rm e}^{-2\ii\theta_+(t,x,z)} \\
r(z) {\rm e}^{2\ii \theta_-(t,x,z)} & 1 \end{pmatrix}
\]
if $\Lambda=\infty$ ($z\in \R$).

Since $r(z)=O\big(z^{-1}\big)$ and $n(z)\to 0$ as $z\to\pm\infty$ ($z\in\R$) and $\Re(2\ii \theta_\pm(t,x,z))=\pm \pi n(z)x/2$ ($z\in\R$), then $J(t,x,z)=I+O\big(z^{-1}\big)$ as $z\to \pm\infty$ for $\Lambda\le\infty$.

\emph{Consider the case when $t\le x$}. In this case $M(t,x,z)=\big(G^{(1)}(t,x,z)\big)^{-1}$, where $G^{(1)}$ is defined in~\eqref{G(1)_t-le-x} (see the proof of Theorem~\ref{causality-region}), and since $\theta(t,x,z)$ is infinitely differentiable in the parameters $t$ and $x$, then $M(t,x,z)$ is also infinitely differentiable in the parameters.

\emph{Now, consider the case when $t>x$}.

\emph{Consider $\Lambda<\infty$}.
Choose any numbers $\la_1, \la_2\in\R$ with $|\la_i|$ sufficiently large and such that $\la_1<-\Lambda$, $\la_1<\Re E$, $\la_2>\Lambda$, and perform the following ``$\delta$-transformation'':
\[
M^{(1)}(t,x,z)=M(t,x,z)\delta^{\sigma_3}(z), \qquad \delta^{\sigma_3}(z)=\begin{pmatrix} \delta(z) & 0 \\ 0 & \delta^{-1}(z) \end{pmatrix},
\]
where $\delta(z)$ is a solution of the following problem of conjugation of boundary values (RH problem):
\begin{itemize}\itemsep=0pt
\item $\delta(z)$ is analytic in $\mC\setminus ((-\infty,\la_1]\cup[\la_2,+\infty))$,
\item $\delta_+(z)=\delta_-(z)\big(1-r^2(z)\big)$
 on $(-\infty,\la_1]\cup[\la_2,+\infty)$,
\item $\delta(z)\to 1$ as $z\to \infty$.
\end{itemize}
Note that $r^2(\la)<0$ for $\la\in\R$, since $\Im w(\la)=0$ on $\R$ and $\Im E>0$.
A unique solution of the problem given above is the function
\begin{gather*}
\delta(z)=\exp\left\{\frac{1}{2\pi\ii}\int_{s\in (-\infty,\la_1]\cup[\la_2,+\infty)} \!\! \frac{\ln (1-r^2(s))\, \dd s}{s-z}\right\},\qquad z\in \mC\setminus \big((-\infty,\la_1]\cup[\la_2,+\infty)\big).
\end{gather*}

For $M^{(1)}(t,x,z)$ the boundary condition has the form
\[
M^{(1)}_-(t,x,z) = M^{(1)}_+(t,x,z)J^{(1)}(t,x,z)
\]
with $J^{(1)}(t,x,z)= \delta_+^{-\sigma_3}(z)J(t,x,z)\delta_-^{\sigma_3}(z)$ which after applying factorizations takes the form
\begin{align*}
J^{(1)}(t,x,z)=&
 \begin{pmatrix} 1 & 0 \\
 \displaystyle\frac{r(z)}{1-r^2(z)} \delta_+^2(z) {\rm e}^{2\ii\theta(t,x,z)} & 1 \end{pmatrix}\!\!
 \begin{pmatrix} 1 & \displaystyle-\frac{r(z)}{1-r^2(z)} \delta_-^{-2}(z) {\rm e}^{-2\ii\theta(t,x,z)} \\
 0 & 1 \end{pmatrix}, \nonumber\\
 & \hspace{82mm}
 z\in (-\infty,\la_1]\cup[\la_2,+\infty), \nonumber\\
 =& \begin{pmatrix} 1-r^2(z){\rm e}^{2\ii(\theta_-(t,x,z)-\theta_+(t,x,z))} & -r(z)\delta^{-2}(z) {\rm e}^{-2\ii\theta_+(t,x,z)} \\ r(z)\delta^2(z) {\rm e}^{2\ii\theta_-(t,x,z)} & 1 \end{pmatrix},\hspace{5mm} z\in(\la_1,\la_2), \nonumber\\
 =& \begin{pmatrix} 1 & h(z)\delta^{-2}(z) {\rm e}^{-2\ii \theta(t,x,z)} \\
 0 & 1 \end{pmatrix}, \hspace{52mm} z\in (E,\Re E), \nonumber\\
 =& \begin{pmatrix} 1 & 0 \\
 h(z)\delta^2(z) {\rm e}^{2\ii\theta(t,x,z)} & 1 \end{pmatrix}, \hspace{57mm} z\in (\Re E,\overline{E}). 
\end{align*}

We introduce the smooth curves $L_1$, $L_3$ and $\overline{L_1}$, $\overline{L_3}$, symmetric to $L_1$, $L_3$ with respect to the real axis, the additional jump contour $\widetilde{\Sigma}=L_1\cup L_3\cup \overline{L_1}\cup \overline{L_3}$ (the points $\la_1$, $\la_2$ do not belong to $\widetilde{\Sigma}$), the domains $D_i$ and the domains $\overline{D_i}$, symmetric to $D_i$ with respect to the real axis, $i=1,2,3$, which are displayed in Figure~\ref{conturG2-LaFin}. Thus, the complex $z$-plane is decomposed into the parts $\mC=\widetilde{\Sigma}\cup\Sigma \cup\bigl(\cup_{i=1}^3 D_i\bigr)\cup\bigl(\cup_{i=1}^3 \overline{D_i}\bigr)$.
\begin{figure}[t] \centering
 \begin{tikzpicture}[smooth,scale=1,thick,font=\small]
\draw[-Stealth, black] (-6,0) -- (0.5,0);
\draw[black] (0.4,0) -- (6,0);
\draw[-Stealth, black] (-1.25,0.75) -- (-1.25,0.25);
\draw[-Stealth, black] (-1.25,0.35) -- (-1.25,-0.45);
\draw[black] (-1.25,-0.35) -- (-1.25,-0.75);
\draw[color=black, fill=black](-1.25,0.75) circle(0.03);
\draw[color=black, fill=black](-1.25,-0.75) circle(0.03);
\draw[black,very thick] (-5,1.7) .. controls (-3.2,0.7) and (-3.2,-0.7) .. (-5,-1.7);
\draw[-Stealth, black,very thick] (-6,1.94) .. controls (-5.75,1.93) and (-5.25,1.8).. (-4.95,1.68);
\draw[-Stealth, black,very thick] (-6,-1.94) .. controls (-5.75,-1.93) and (-5.25,-1.8).. (-4.95,-1.68);
\draw[color=black, fill=black](-3.65,0) circle(0.05);
\draw[Stealth-Stealth,black,very thick] (5,1.7) .. controls (3.2,0.7) and (3.2,-0.7) .. (5,-1.7);
\draw[black,very thick] (6,1.94) .. controls (5.75,1.93) and (5.25,1.8).. (4.95,1.68);
\draw[black,very thick] (6,-1.94) .. controls (5.75,-1.93) and (5.25,-1.8).. (4.95,-1.68);
\draw[color=black, fill=black](3.62,0) circle(0.05);
\node[black] at (-4,-0.3) {$\lambda_{1}$};
\node[black] at (-5.3,0.9) {$D_{3}$};
\node[black] at (-5.3,-0.9) {$\overline{D_{3}}$};
\node[black] at (-4.7,1.98) {$L_{3}$};
\node[black] at (-4.7,-1.98) {$\overline{L_{3}}$};
\node[black] at (-1,0.97) {$E$};
\node[black] at (-1,-1) {$\overline{E}$};
\node[black] at (0.9,1.5) {$D_{2}$};
\node[black] at (0.9,-1.5) {$\overline{D_{2}}$};
\node[black] at (1.4,0.25) {$\mathbb{R}$};
\node[black] at (4,-0.3) {$\lambda_{2}$};
\node[black] at (5.3,0.9) {$D_{1}$};
\node[black] at (5.3,-0.9) {$\overline{D_{1}}$};
\node[black] at (4.7,1.98) {$L_{1}$};
\node[black] at (4.7,-1.98) {$\overline{L_{1}}$};
 \end{tikzpicture}
 \caption{The contours $\Sigma=\R\cup\big(E,\overline{E}\big)$ and $\widetilde{\Sigma}=L_1\cup L_3\cup \overline{L_1}\cup \overline{L_3}$ and the domains $D_i$, $\overline{D_i}$, $i=1,2,3$.}\label{conturG2-LaFin}
\end{figure}
Further, we transform the matrix $M^{(1)}(t,x,z)$ into
\[
M^{(2)}(t,x,z)= M^{(1)}(t,x,z)G^{(2)}(t,x,z),
\]
where
\begin{alignat*}{3}
&G^{(2)}(t,x,z)= \begin{pmatrix} 1 & 0 \\ \displaystyle \frac{r(z)}{1-r^2(z)} \delta^2(z) {\rm e}^{2\ii \theta(t,x,z)} & 1 \end{pmatrix},\qquad && z\in D_1\cup D_3, &\nonumber\\
 &\hphantom{G^{(2)}(t,x,z)}{} = \begin{pmatrix} 1 & \displaystyle\frac{r(z)}{1-r^2(z)} \delta^{-2}(z) {\rm e}^{-2\ii\theta(t,x,z)} \\ 0 & 1 \end{pmatrix}, \qquad && z\in\overline{D_1}\cup\overline{D_3}, &\nonumber\\
 &\hphantom{G^{(2)}(t,x,z)}{}= I, \qquad && z\in D_2\cup\overline{D_2}.& 
\end{alignat*}
Since ${\rm e}^{2\ii\theta(t,x,z)}\to 0$ and ${\rm e}^{-2\ii\theta(t,x,z)}\to 0$ as $z\to \infty$ for $\Im z>0$ and $\Im z<0$ respectively (see the case $t>x$, $\Lambda<\infty$ in Appendix~\ref{appendixA} and Figure~\ref{ReiTheta1}), then $G^{(2)}(t,x,z)\to I$ and hence $M^{(2)}(t,x,z)\to I$ as $z\to\infty$.
For $M^{(2)}(t,x,z)$ the jump (boundary) condition has the form $M^{(2)}_-(t,x,z) = M^{(2)}_+(t,x,z)J^{(2)}(t,x,z)$, ${z\in \Sigma\cup\widetilde{\Sigma}\setminus\{\Re E,\la_1,\la_2\}}$, with the jump matrix ${J^{(2)}(t,x,z)= \big(G_+^{(2)}(t,x,z)\big)^{-1}J^{(1)}(t,x,z)G_-^{(2)}(t,x,z)}$ defined on the new contour $\Sigma\cup \widetilde{\Sigma}$ whose orientation remains the same on $\Sigma= \R\cup(E,\overline{E})$, and all branches of $\widetilde{\Sigma}$ (i.e., $L_i$, $\overline{L_i}$, $i=1,3$) are oriented from the left to the right. It is easy to verify that
\begin{align}
J^{(2)}(t,x,z)&=I, && z\in (-\infty,\la_1)\cup(\la_2,+\infty), \nonumber\\
 &=\begin{pmatrix}
 1 & 0 \\
 \frac{r(z)}{1-r^2(z)} \delta^2(z) {\rm e}^{2\ii \theta(t,x,z)} & 1 \end{pmatrix}, &&z\in L_1\cup L_3, \smallskip \nonumber \\
 &=\begin{pmatrix}
 1 & -\frac{r(z)}{1-r^2(z)}\delta^{-2}(z) {\rm e}^{-2\ii\theta(t,x,z)} \\
 0 & 1 \end{pmatrix}, &&z\in \overline{L_1}\cup\overline{L_3}, \nonumber\\
 &= J^{(1)}(t,x,z), &&z\in(\la_1,\la_2)\cup\big(E,\overline{E}\big)\setminus\{\Re E\}. \label{case2J2-dif}
\end{align}

Introduce the contour $\Upsilon= \widetilde{\Sigma}\cup[\la_1,\la_2]\cup[E,\overline{E}]=L_1\cup L_3\cup \overline{L_1}\cup\overline{L_3}\cup[\la_1,\la_2]\cup[E,\overline{E}]$. As in the proof of Theorem~\ref{Ex-Uniq}, we obtain that for any fixed $t$, $x$ there exists a unique solution $W^{(2)}(t,x,\zeta)$ of the singular integral equation (the equation is considered in the space of operator functions $W^{(2)}(t,x,\zeta)\in L^2(\Upsilon)$ with respect to $\zeta$; $W^{(2)}(t,x,\zeta):=M^{(2)}_+(t,x,\zeta)-I$)
\begin{equation}\label{IntEq2-RH0}
 W^{(2)}(t,x,\zeta)-\mathbf{K}^{(2)}\big[W^{(2)}\big](t,x,\zeta) =\Pi^{(2)}(t,x,\zeta),\qquad \zeta\in\Upsilon,
\end{equation}
where
 \begin{align*}
&\mathbf{K}^{(2)}\big[W^{(2)}\big](t,x,\zeta):= \frac{1}{2\pi\ii} \int_\Upsilon \frac{W^{(2)}(t,x,s)\big[I-J^{(2)}(t,x,s)\big]}{s-\zeta_+}\, \dd s, \\
&\Pi^{(2)}(t,x,\zeta):= \frac{1}{2\pi\ii} \int_\Upsilon \frac{I-J^{(2)}(t,x,s)}{s-\zeta_+}\, \dd s,
 \end{align*}
and then we find $M^{(2)}(t,x,z)$ by the formula
\begin{equation}\label{sol2_RH0}
 M^{(2)}(t,x,z)=I+ \frac{1}{2\pi\ii}\int_\Upsilon \frac{\big[W^{(2)}(t,x,s)+I\big]\big[I-J^{(2)}(t,x,s)\big]}{s-z}\, \dd s,\qquad z\in \mC\setminus\Upsilon.
\end{equation}
Since ${\rm e}^{\pm 2\ii\theta(t,x,z)}\to 0$ as $z\to \infty$, $z\in\mC_\pm$, respectively (see the case $t>x$, $\Lambda<\infty$ in Appendix~\ref{appendixA}), then for sufficiently large $z\in\widetilde{\Sigma}$ (i.e., on the parts of the jump contour $\Upsilon$ contained in a~neighborhood of infinity), the matrix $I-J^{(2)}(t,x,z)$ decreases exponentially and its derivatives of any order in $t$ and $x$ decrease fast enough so that the integrals that arise when differentiating are convergent.
Consequently, the solution $W^{(2)}(t,x,\zeta)$ of the equation~\eqref{IntEq2-RH0} and the function $M^{(2)}(t,x,z)$~\eqref{sol2_RH0} are infinitely differentiable in $t$ and $x$. Since $M^{(2)}(t,x,z)$ and $G^{(2)}(t,x,z)$ are infinitely differentiable in $t$, $x$, then $M^{(1)}(t,x,z)$ and $M(t,x,z)$ also have this property.

\emph{Consider $\Lambda=\infty$ $($as before, $t>x)$}.

We introduce the smooth curves $L_2=L_2(\xi)$, $\overline{L_2}=\overline{L_2}(\xi)$ and the domains $D_2=D_2(\xi)\subset \mC_+\setminus[E,\Re{E})$, $D_4=D_4(\xi)\subset\mC_+$, $\overline{D_2}=\overline{D_2}(\xi)\subset\mC_-\setminus\big(\Re{E},\overline{E}\big]$ and $\overline{D_4}=\overline{D_4}(\xi)\subset\mC_-$, symmetric with respect to $\R$, where $\xi=\frac{x}{4(t-x)}>0$, which are such that (see Figures~\ref{conturG2-LaInf1} and~\ref{conturG2-LaInf2}):%
\begin{itemize}[label=\textendash]\itemsep=0pt
\item all sufficiently large $z$ belonging to $L_2(\xi)$ and $\overline{L_2}(\xi)$ (i.e., the parts of $L_2$ and $\overline{L_2}$ contained in a neighborhood of infinity) are contained in the domain bounded from above by the curve~$\gamma^+_\xi$ and from below by the curve $\gamma^-_\xi$, where $\gamma^\pm_\xi$ are defined by the equality $\int_{-\infty}^\infty \frac{n(s)\, \dd s}{|s-z|^2}=\frac{1}{\xi}$ \big( $\Re(\ii\theta)=0$ on $\gamma^\pm_\xi$; see the case $t>x$, $\Lambda=\infty$ in Appendix~\ref{appendixA} and Figure~\ref{ReiTheta2}\big);
\item the interval $[E,\overline{E}]$ lies in the domain bounded by $L_2$ and $\overline{L_2}$;
\item the domain $D_2$ is bounded by $L_2$ and $\R\cup[E,\Re{E})$;
\item the domain $\overline{D_2}$ is bounded by $\overline{L_2}$ and $\R\cup\big(\Re{E},\overline{E}\big]$.
\end{itemize}
Then complex $z$-plane is decomposed into the parts $\mC=L_2\cup\overline{L_2}(\xi)\cup\Sigma\cup D_2\cup D_4\cup \overline{D_2}\cup \overline{D_4}$. Define the contour $\widetilde{\Sigma}=L_2\cup\overline{L_2}$, where $L_2$ and $\overline{L_2}$ are oriented from the left to the right. The orientation on $\Sigma$ remains the same as above.
\begin{figure}[t] \centering
 \begin{tikzpicture}[smooth,scale=1,thick,font=\small]
\draw[black] (-6,0) -- (6,0);
\draw[black] (-1.2,0.85) -- (-1.2,-0.85);
\draw[color=black, fill=black](-1.2,0.85) circle(0.03);
\draw[color=black, fill=black](-1.2,-0.85) circle(0.03);
\draw[black,very thick] (-6,0.15) .. controls (-1.3,1.65) and (1.3,1.65).. (6,0.15);
\draw[dashed, black,very thick] (-6,0.5) .. controls (-1.2,2.1) and (1.2,2.1).. (6,0.5);
\draw[black,very thick] (-6,-0.15) .. controls (-1.3,-1.65) and (1.3,-1.65).. (6,-0.15);
\draw[dashed, black,very thick] (-6,-0.5) .. controls (-1.2,-2.1) and (1.2,-2.1).. (6,-0.5);
\node[black] at (-2.1,0.4) {$D_{2}$};
\node[black] at (-2.1,-0.4) {$\overline{D_{2}}$};
\node[black] at (-2.7,1.8) {$D_{4}$};
\node[black] at (-2.7,-1.8) {$\overline{D_{4}}$};
\node[black] at (2,0.8) {$L_{2}$};
\node[black] at (2,-0.8) {$\overline{L_{2}}$};
\node[black] at (-0.91,0.8) {$E$};
\node[black] at (-0.91,-0.8) {$\overline{E}$};
\node[black] at (0.6,0.2) {$\mathbb{R}$};
\node[black] at (4.4,1.34) {$\gamma_{\xi}^{+}$};
\node[black] at (4.4,-1.34) {$\gamma_{\xi}^{-}$};
 \end{tikzpicture}
 \caption{The contours $\Sigma=\R\cup\big(E,\overline{E}\big)$ and $\widetilde{\Sigma}=L_2\cup\overline{L_2}$, the curves $\gamma^\pm_\xi$ and the domains $D_i$, $\overline{D_i}$, $i=2,4$, in the case when the interval $\big[E,\overline{E}\big]$ lies inside the domain bounded by the curves $\gamma^+_\xi$, $\gamma^-_\xi$.}\label{conturG2-LaInf1}
\end{figure}
\begin{figure}[t] \centering
 \begin{tikzpicture}[smooth,scale=1,thick,font=\small]
\draw[black] (-6,0) -- (6,0);
\draw[black] (0.4,0) -- (6,0);
\draw[black] (-2,1.6) -- (-2,-1.6);
\draw[color=black, fill=black](-2,1.6) circle(0.03);
\draw[color=black, fill=black](-2,-1.6) circle(0.03);
\draw[black,very thick] (-6,0.15) .. controls (-1,2.85) and (1,2.85).. (6,0.15);
\draw[dashed, black,very thick] (-6,0.4) .. controls (-1.2,1.9) and (1.2,1.9).. (6,0.4);
\draw[black,very thick] (-6,-0.15) .. controls (-1,-2.85) and (1,-2.85).. (6,-0.15);
\draw[dashed, black,very thick] (-6,-0.4) .. controls (-1.2,-1.9) and (1.2,-1.9).. (6,-0.4);
\node[black] at (-0.7,0.8) {$D_{2}$};
\node[black] at (-0.7,-0.8) {$\overline{D_{2}}$};
\node[black] at (-4.5,1.6) {$D_{4}$};
\node[black] at (-4.5,-1.6) {$\overline{D_{4}}$};
\node[black] at (2.2,0.9) {$\gamma_{\xi}^{+}$};
\node[black] at (2.2,-0.9) {$\gamma_{\xi}^{-}$};
\node[black] at (-1.75,1.7) {$E$};
\node[black] at (-1.75,-1.7) {$\overline{E}$};
\node[black] at (1,0.2) {$\mathbb{R}$};
\node[black] at (4.4,1.34) {$L_{2}$};
\node[black] at (4.4,-1.345) {$\overline{L_{2}}$};
 \end{tikzpicture}
 \caption{The contours $\Sigma=\R\cup\big(E,\overline{E}\big)$ and $\widetilde{\Sigma}=L_2\cup\overline{L_2}$, the curves $\gamma^\pm_\xi$ and the domains $D_i$, $\overline{D_i}$, $i=2,4$, in the case when the interval $\big[E,\overline{E}\big]$ intersects $\gamma^+_\xi$, $\gamma^-_\xi$.}\label{conturG2-LaInf2}
\end{figure}
Further, we transform the matrix $M(t,x,z)$ into
\[
M^{(2)}(t,x,z)= M(t,x,z)G^{(2)}(t,x,z),
\]
where
\begin{alignat*}{3}
&G^{(2)}(t,x,z)= \begin{pmatrix} 1 & -r(z) {\rm e}^{-2\ii\theta(t,x,z)} \\ 0 & 1 \end{pmatrix}, \qquad && z\in D_2, & \nonumber\\
 &\hphantom{G^{(2)}(t,x,z)}{} = \begin{pmatrix} 1 & 0 \\ -r(z) {\rm e}^{2\ii\theta(t,x,z)} & 1 \end{pmatrix}, && z\in\overline{D_2},& \nonumber\\
 &\hphantom{G^{(2)}(t,x,z)}{}= I, \qquad && z\in D_4\cup\overline{D_4}.& 
\end{alignat*}
Then ${M^{(2)}_-(t,x,z) = M^{(2)}_+(t,x,z)J^{(2)}(t,x,z)}$, $z\in \widetilde{\Sigma}\cup\Sigma\setminus\{\Re E\}$, where
\[
J^{(2)}(t,x,z)= \big(G_+^{(2)}(t,x,z)\big)^{-1}J(t,x,z)G_-^{(2)}(t,x,z),\qquad z\in \widetilde{\Sigma}\cup\Sigma,
\]
has the following form:
\begin{alignat}{3}
&J^{(2)}(t,x,z)=I,\qquad && z\in\Sigma, &\nonumber\\
 &\hphantom{J^{(2)}(t,x,z)}{} =\begin{pmatrix} 1 & -r(z) {\rm e}^{-2\ii\theta(t,x,z)} \\ 0 & 1 \end{pmatrix}, \qquad && z\in L_2,& \nonumber \\
 &\hphantom{J^{(2)}(t,x,z)}{}=\begin{pmatrix} 1 & 0 \\ r(z) {\rm e}^{2\ii\theta(t,x,z)} & 1 \end{pmatrix}, \qquad && z\in \overline{L_2}.& \label{case2J2-dif2}
\end{alignat}

As above, we consider the singular integral equation~\eqref{IntEq2-RH0}, where $\Upsilon=\widetilde{\Sigma}=L_2\cup\overline{L_2}$, and find $M^{(2)}(t,x,z)$ \big($G^{(2)}(t,x,z)\to I$ and $M^{(2)}(t,x,z)\to I$ as $z\to\infty$\big) by the formula~\eqref{sol2_RH0}. Since ${\rm e}^{-2\ii\theta(t,x,z)}\to 0$ and ${\rm e}^{2\ii\theta(t,x,z)}\to 0$ as $z\to \infty$ for $z$ from the domains bounded by $\R$, \smash{$\gamma^+_\xi$} and $\R$, $\gamma^-_\xi$ respectively (see Figure~\ref{ReiTheta2} in Appendix~\ref{appendixA}), then for sufficiently large \smash{$z\in\widetilde{\Sigma}$} (i.e., on the parts of the jump contour $\Upsilon$ contained in a neighborhood of infinity), the matrix $I-J^{(2)}(t,x,z)$ decreases exponentially and its derivatives of any order in $t$ and $x$ decrease fast enough. Consequently, the solution $W^{(2)}(t,x,\zeta)$ of the singular integral equation, the function $M^{(2)}(t,x,z)$ and, hence, the function $M(t,x,z)$ are infinitely differentiable in $t$ and $x$.
\end{proof}

\begin{Theorem}\label{solvMBbyM}
Let
 \begin{equation}\label{Phi_copy}
\Phi(t,x,z):= M(t,x,z) {\rm e}^{-\ii\theta(t,x,z)\sigma_3},
 \end{equation}
where $M(t,x,z)$ is a solution of the basic problem RH$_0$. Then $\Phi(t,x,z)$ is a unique solution of the AKNS system
 \begin{align}
&\Phi_t=-\big(\ii z\sigma_3+H(t,x)\big)\Phi, \label{teq1} \\
&\Phi_x=\big(\ii z\sigma_3+H(t,x)-\ii G(t,x,z)\big)\Phi,\qquad G(t,x,z)=\frac{1}{4}\int_{-\infty}^\infty \frac{F(t,x,s)n(s)\, \dd s}{s-z}, \label{xeq1}
 \end{align}
where $t\in\R_+$, $x\in (0,L)$, $L\le\infty$, and $z\in\mC\setminus \Sigma^{\rm cl}$, $\Sigma^{\rm cl}=\Sigma\cup \big\{E,\overline{E}\big\}$.
In the system~\eqref{teq1} and~\eqref{xeq1}, the function $H(t,x)$ is defined by
 \begin{equation}\label{H}
H(t,x)=\frac{1}{2} \begin{pmatrix}
 0&{\mathcal E}(t,x)\\
 -\overline{\mathcal E(t,x)}&0\end{pmatrix}=-\ii [\sigma_3, m(t,x)],
 \end{equation}
where
\begin{equation}\label{m_lim}
m(t,x)=\lim_{z\to\infty}z(M(t,x,z)-I),\qquad {\mathcal E}(t,x)=-4\ii \lim_{z\to\infty} (zM(t,x,z))_{12},
\end{equation}
the function $F(t,x,\la)$ is defined by
\begin{align}
F(t,x,\la)=&{} \begin{pmatrix}
 {\mathcal N}(t,x,\la) &\rho(t,x,\la) \\ \overline{\rho(t,x,\la)} &-{\mathcal N}(t,x,\la) \end{pmatrix} \nonumber\\
 =&{}- M(t,x,\la+\ii0) {\rm e}^{-\ii\la t\sigma_3}M^{-1}(0,x,\la+\ii0)\sigma_3 M(0,x,\la+\ii0) {\rm e}^{\ii\la t\sigma_3}\nonumber \\
 &{}\times M^{-1}(t,x,\la+\ii0),\qquad \la\in\R, \label{F}
\end{align}
where $M(t,x,\la+\ii0)=M_+(t,x,\la)$ is defined appropriately at the point $\la=\Re E$ so that $F(t,x,\la)$ is continuous in $\la\in\R$, and these functions satisfy the Maxwell--Bloch equations in the matrix form
 \begin{align}
&H_t(t,x)+H_x(t,x)-\frac{1}{4}\int_{-\infty}^\infty [\sigma_3,F(t,x,s)]n(s){\rm d}s =0, \label{HFeq1} \\
&F_t(t,x,\la)+[\ii\la\sigma_3+H(t,x),F(t,x,\la)]=0, \label{HFeq2}
 \end{align}
where $t\in\R_+$, $x\in (0,L)$, $\la\in\R$, and the initial and boundary conditions
\[
F(0,x,\la)\equiv -\sigma_3,\qquad H(0,x)\equiv 0,\qquad H(t,0)=\frac{1}{2} \begin{pmatrix}0&\mathcal E_{\rm in}(t)\\-\overline{\mathcal E_{\rm in}(t)}&0\end{pmatrix},
\]
that is, the functions ${\mathcal E}(t,x)$, $\rho(t,x,\la)$ and $\mathcal N(t,x,\la)$ satisfy the MB equations~\eqref{MB1a} and the initial and boundary conditions~\eqref{MBini} and~\eqref{MBboundAe}.
 \end{Theorem}

 \begin{proof}
First, note that $\Phi(t,x,z)$~\eqref{Phi_copy} is infinitely differentiable in $t$ and $x$ since $M(t,x,z)$ (see Theorem~\ref{smoothnessRH}) and $\theta(t,x,z)$~\eqref{theta1} have this property and that the solution $M(t,x,z)$ of the basic problem RH$_0$ is unique (see Theorem~\ref{Ex-Uniq}). Since $M(t,x,z)$ satisfies the normalization condition $M(t,x,z)=I+O\big(z^{-1}\big)$, $z\to\infty$, and can be represented in the form~\eqref{sol_RH0}, then it has the asymptotics
\[
M(t,x,z)=I+m(t,x)z^{-1}+O\big(z^{-2}\big),\qquad z\to\infty, \quad z\in\mC\setminus\Sigma,
\]
where $m(t,x)$ has the form~\eqref{m_lim}, i.e., $m(t,x)=\lim_{z\to\infty}z(M(t,x,z)-I)$ (the limit at $\infty$ is nontangential). In addition,
\[
M_t(t,x,z)=m_t(t,x)z^{-1}+O\big(z^{-2}\big),\qquad z\to\infty, \quad z\in\mC\setminus\Sigma.
\]
Since $M(t,x,z)$ is a solution of the basic problem RH$_0$ and $\Phi(t,x,z)= M(t,x,z) {\rm e}^{-\ii\theta(t,x,z)\sigma_3}$, then $\Phi(t,x,z)$ is analytic in $z\in \mC\setminus\Sigma^{\rm cl}$, its boundary values $\Phi_\pm(t,x,z)$ are continuous in $z\in\Sigma\setminus\{\Re E\}$, and $\Phi(t,x,z)$ satisfies the boundary condition
\begin{equation}\label{BC_Phi}
\Phi_-(t,x,z)=\Phi_+(t,x,z) {\rm e}^{\ii\theta_+(t,x,z)\sigma_3}J(t,x,z) {\rm e}^{-\ii\theta_-(t,x,z)\sigma_3},\qquad z\in\Sigma\setminus\{\Re E\},
\end{equation}
where the matrix $ {\rm e}^{\ii\theta_+(t,x,z)\sigma_3}J(t,x,z) {\rm e}^{-\ii\theta_-(t,x,z)\sigma_3}$ is independent on $t$ for all $z\in\mC$ and is independent on $x$ for $z\in\mC\setminus [-\Lambda,\Lambda]$ (since $\theta_+(t,x,z)=\theta_-(t,x,z)$ for $z\in (E,\overline{E})\setminus\{\Re E\}$ and $z=\la\in \R\setminus [-\Lambda,\Lambda]$ and $\theta_-(t,x,\la)-\theta_+(t,x,\la)=\pi\ii n(\la)x/2$ for $\la\in[-\Lambda,\Lambda]$). This implies
\[
\frac{\partial\Phi_-(t,x,z)}{\partial t}\Phi^{-1}_-(t,x,z)= \frac{\partial\Phi_+(t,x,z)}{\partial t}\Phi^{-1}_+(t,x,z)
\]
for $z\in\Sigma\setminus\{\Re E\}$. Consequently, the logarithmic derivative $\Phi_t(t,x,z)\Phi^{-1}(t,x,z)$ is analytic in $z\in\mC\setminus\bigl\{\Re E,E,\overline{E}\bigr\}$ and the points $\Re E$, $E$, $\overline{E}$ are removable singularities due to the properties of $M(t,x,z)$. Thus, we can define $\Phi_t(t,x,z)\Phi^{-1}(t,x,z)$ at these points so that it becomes analytic in $z\in\mC$. Using the above asymptotic expansions for $M(t,x,z)$ and $M_t(t,x,z)$ and~\eqref{Phi_copy}, we obtain
\[
\Phi_t(t,x,z)\Phi^{-1}(t,x,z)= -\ii z\sigma_3+\ii[\sigma_3,m(t,x)]+ O\big(z^{-1}\big),\qquad z\to\infty.
\]
Hence, by the Liouville theorem,
\[
\Phi_t(t,x,z)\Phi^{-1}(t,x,z)=-\ii z\sigma_3-H(t,x),
\]
where $H(t,x)=-\ii[\sigma_3,m(t,x)]$ and, therefore, it has the structure $H(t,x)=\left(\begin{smallmatrix}0&q(t,x)\\p(t,x)&0\end{smallmatrix}\right)$ (since $H \sigma_3+\sigma_3 H=0$).
Using the symmetries of the contour $\Sigma$ and the jump matrix $J(t,x,z)$, namely, $\sigma_2\overline{J(t,x,\bar{z})}\sigma_2=(J^*(t,x,\bar{z}))^{-1}=J(t,x,z)$ for $z\in\Sigma\setminus\R$ and $J^*(t,x,z)=J(t,x,z)$ for $z=\la\in\R$, we find that the matrix $H(t,x)$ is anti-Hermitian, i.e., $H(t,x)= -H^*(t,x)$, and therefore $q(t,x)=-\overline{p(t,x)}=-2\ii m_{12}(t,x)$, where $m(t,x)=(m_{ij}(t,x))_{1\le i,j\le 2}$ is defined by~\eqref{m_lim}. We set $q(t,x):=\mathcal E(t,x)/2$, then $\mathcal E(t,x)$ has the form~\eqref{E1}, i.e., ${\mathcal E}(t,x)=-4\ii m_{12}(t,x)=-4\ii \lim_{z\to\infty} \big(zM(t,x,z)\big)_{12}$. Also, we denote $U(t,x,z):=\Phi_t(t,x,z)\Phi^{-1}(t,x,z)$. Thus, it is proved that the matrix function $\Phi(t,x,z)$ satisfies the $t$-equation~\eqref{teq1} which becomes the $t$-equation~\eqref{teq}, where $U(t,x,\la)=-(\ii\la\sigma_3+H(t,x))$~\eqref{U-MB}, as $z=\la+\ii 0$.

Using~\eqref{BC_Phi}, we obtain that
\[
\frac{\partial\Phi_-(t,x,z)}{\partial x}\Phi^{-1}_-(t,x,z)= \frac{\partial\Phi_+(t,x,z)}{\partial x}\Phi^{-1}_+(t,x,z)
\]
for $z\in\Sigma\setminus [-\Lambda,\Lambda]$, $z\ne \Re E$. Consequently, the logarithmic derivative $\Phi_x(t,x,z)\Phi^{-1}(t,x,z)$ is analytic for $z\in\mC\setminus([-\Lambda,\Lambda]\cup\{\Re E,E,\overline{E}\})$ and the points $E$, $\overline{E}$, as well as $\Re E$ if $\Re E\not\in [-\Lambda,\Lambda]$, are removable singularities. Thus, we can modify $\Phi_x(t,x,z)\Phi^{-1}(t,x,z)$ so that it becomes analytic in $z\in\mC\setminus[-\Lambda,\Lambda]$. Taking into account asymptotic expansions for $M(t,x,z)$ and $M_x(t,x,z)$ and the formulas~\eqref{Phi_copy} and~\eqref{eta}, we obtain
\[
\Phi_x(t,x,z)\Phi^{-1}(t,x,z)=\ii z\sigma_3 +H(t,x) +O\big(z^{-1}\big),\qquad z\to\infty,\quad z\in\mC\setminus[-\Lambda,\Lambda].
\]
Note that in~\eqref{BC_Phi}, for $z=\la\in[-\Lambda,\Lambda]$ the matrix $\widetilde{J}(x,\la):= {\rm e}^{\ii\theta_+(t,x,\la)\sigma_3}J(t,x,\la){\rm e}^{-\ii\theta_-(t,x,\la)\sigma_3}$ can be represented as
\[
\widetilde{J}(x,\la)= \left(\begin{matrix} 1 & -r(\la) \\ 0 & 1	\end{matrix}\right) {\rm e}^{\frac{\pi n(\la)x\sigma_3}{2}} \left(\begin{matrix} 1 & 0 \\ r(\la) & 1	\end{matrix}\right).
\]
 Thus, using the boundary condition~\eqref{BC_Phi} for $z=\la\in\R$, i.e., $\Phi_-(t,x,\la)=\Phi_+(t,x,\la)\widetilde{J}(x,\la)$, where $\Phi_\pm(t,x,\la)=\Phi(t,x,\la\pm\ii0)$, we obtain that $\Phi_x(t,x,z)\Phi^{-1}(t,x,z)$ has the following jump across $[-\Lambda,\Lambda]$:
\begin{align*}
&{}\Phi_x(t,x,\la+\ii0)\Phi^{-1}(t,x,\la+\ii0)- \Phi_x(t,x,\la-\ii0)\Phi^{-1}(t,x,\la-\ii0) \\
&{}\qquad = -\Phi(t,x,\la+\ii0)\widetilde{J}_x(x,\la)\widetilde{J}^{-1}(x,\la) \Phi^{-1}(t,x,\la+\ii0) \\
&{}\qquad =-\frac{\pi n(\la)}{2}\Phi(t,x,\la+\ii0)\begin{pmatrix} 1 & -r(\la) \\ 0 & 1	\end{pmatrix}\sigma_3\begin{pmatrix} 1 & r(\la) \\ 0 & 1	\end{pmatrix} \Phi^{-1}(t,x,\la+\ii0),
\end{align*}
and we denote
\[
F(t,x,\la):=-\Phi(t,x,\la+\ii0)\begin{pmatrix} 1 & -r(\la) \\ 0 & 1	\end{pmatrix}\sigma_3\begin{pmatrix} 1 & r(\la) \\ 0 & 1	\end{pmatrix} \Phi^{-1}(t,x,\la+\ii0).
\]
 Thus, the function $\Phi_x(t,x,z)\Phi^{-1}(t,x,z)$ can be found by the formula
\[
 \Phi_x(t,x,z)\Phi^{-1}(t,x,z)=\ii z\sigma_3 +H(t,x)+ \frac{1}{4\ii}\int_{[-\Lambda,\Lambda]}\frac{n(s)F(t,x,s)}{s-z}\, \dd s, \qquad z\in\mC\setminus[-\Lambda,\Lambda].
\]
Since $J^*(t,x,\la)=J(t,x,\la)$, $\la\in\R$, then $\widetilde{J}^*(x,\la)=\widetilde{J}(x,\la)$, $\la\in\R$. This yields
\begin{align*}
\frac{\pi n(\la)}{2}F(t,x,\la)& =(\Phi^*(t,x,\la+\ii0))^{-1}\Phi^*_x(t,x,\la+\ii0)- (\Phi^*(t,x,\la-\ii0))^{-1}\Phi^*_x(t,x,\la-\ii0)\\
& =\frac{\pi n(\la)}{2}F^*(t,x,\la),
\end{align*}
 and hence the matrix $F(t,x,\la)$ is Hermitian, i.e., $F(t,x,\la)=F^*(t,x,\la)$.
In addition, since $\det\Phi(t,x,z)=\det M(t,x,z) \equiv 1$, then
$\tr(\Phi_x(t,x,\la\pm\ii0)\Phi^{-1}(t,x,\la\pm\ii0))= (\det\Phi(t,x,\la\pm\ii0))_x\equiv0$, and therefore $\tr F(t,x,\la)\equiv 0$. Consequently, $F(t,x,\la)$ has the structure
\[
F(t,x,\la):= \begin{pmatrix}{\mathcal N}(t,x,\la)&{\mathcal\rho}(t,x,\la)\\
\overline{{\mathcal\rho}(t,x,\la)}&-{\mathcal N}(t,x,\la)\end{pmatrix}.
\]
Thus, $\Phi(t,x,z)$ satisfies the $x$-equation~\eqref{xeq1}.
It is proved that the matrix function $\Phi(t,x,z)$ satisfies the $t$- and $x$-equations~\eqref{teq1} and~\eqref{xeq1}, that is,
\begin{alignat}{3}
 & \Phi_t=U(t,x,z)\Phi, \qquad && U(t,x,z)=-\ii z\sigma_3-H(t,x), & \label{teq1z} \\
 & \Phi_x =V(t,x,z)\Phi, \qquad && V(t,x,z)=-U(t,x,z)-\ii G(t,x,z), & \nonumber 
\end{alignat}
where
$G(t,x,z)=\frac{1}{4}\int_{-\infty}^\infty \frac{F(t,x,s)n(s)}{s-z}\, \dd s$, $z\in\mC\setminus[-\Lambda,\Lambda]$, and $\Phi_\pm(t,x,z)=\Phi(t,x,\la\pm\ii0)$ for $z=\la\in\R$ satisfy the $t$-equation~\eqref{teq1z} and the $x^\pm$-equations
\begin{equation}\label{pmxeq1}
 \Phi_x =V_\pm(t,x,\la)\Phi,
\end{equation}
where $V_\pm(t,x,\la)=-U(t,x,\la)-\ii G_\pm(t,x,\la)=\ii\la\sigma_3+H(t,x)-\ii G_\pm(t,x,\la)$~\eqref{V_pm-MB}, $G_\pm(t,x,\la)=G(t,x,\la\pm\ii0)$ has the form~\eqref{G_pm}, $G_\pm(t,x,\la)=G(t,x,z)$ for $z=\la\in\R\setminus[-\Lambda,\Lambda]$ and $G(t,x,\la):=\frac{1}{4} {\rm p.v.}\int_{-\infty}^\infty \frac{F(t,x,s)n(s)}{s-\la}\, \dd s$ for $\la\in[-\Lambda,\Lambda]$.
The compatibility conditions of the systems~\eqref{teq1z},~\eqref{pmxeq1} imply~\eqref{ZCC_pm}. Using~\eqref{ZCC_pm}, we obtain the equation
\begin{align*}
&{} H_x(t,x)+H_t(t,x)-\frac{1}{4}\int_{-\infty}^\infty [\sigma_3,F(t,x,s)]n(s)\, \dd s \\
&{}\qquad -\frac{\ii}{4} \int_{-\infty}^\infty \frac{(F_t(t,x,s)+[\ii s\sigma_3+H(t,x),F(t,x,s)])n(s)}{s-\la\mp\ii0}\, \dd s=0
\end{align*}
that yields the MB equations in the matrix form~\eqref{HFeq1} and~\eqref{HFeq2} which are equivalent to the MB equations~\eqref{MB1a}. Thus, the matrices $\Phi(t,x,\la\pm\ii0)$ satisfy the systems of the equations~\eqref{teq1z} and~\eqref{pmxeq1} which coincide with the AKNS systems~\eqref{teq},~\eqref{pmxeq}, and the compatibility condition~\eqref{ZCC}, where $U(t,x,\la)$, $V(t,x,\la)$ have the form~\eqref{U-MB} and~\eqref{V-MB}, is fulfilled. Hence the functions ${\mathcal E}(t,x)$, ${\mathcal N}(t,x,\la)$ and ${\mathcal\rho}(t,x,\la)$ satisfy the MB equations~\eqref{MB1a}.

Let us proof that the initial and boundary conditions~\eqref{MBini} and~\eqref{MBboundAe} are fulfilled. Since $t=0\le x$ for all $x\in [0,L)$ ($L\le\infty$), then for $t=0$ as in the proof of Theorem~\ref{causality-region} we obtain that ${\mathcal E(0,x)\equiv 0}$, $\rho(0,x,\la)\equiv 0$ and $\mathcal N(0,x,\la)\equiv -1$ or $H(0,x)\equiv 0$ and $F(0,x,\la)\equiv -\sigma_3$. Taking into account the form of the solution $M(t,x,z)$~\eqref{M} and Theorem~\ref{Theorem_SolAKNS}, the solution of the basic problem RH$_0$ for $x=0$ has the form
\begin{equation*}
M(t,0,z)\!=\!\begin{cases}
 \begin{pmatrix}
 \Psi[1](t,0,z){\rm e}^{\ii zt} & \displaystyle\frac{\Phi[2](t,z)}{a(z)} {\rm e}^{-\ii zt}
 \end{pmatrix},& z\in \mC_+\setminus [E,\Re E), \vspace{1mm}\\
 \begin{pmatrix}
 \displaystyle\frac{\Phi[1](t,z)}{a(z)} {\rm e}^{\ii zt} & \Psi[2](t,0,z){\rm e}^{-\ii zt}
 \end{pmatrix},& z\in \mC_-\setminus (\Re E,\overline{E}],
\end{cases}
\end{equation*}
where $\Psi(t,0,z)$ satisfies the $t$-equation $\Psi_t=(-\ii z\sigma_3-H(t,0))\Psi$ and $\Psi(0,0,z)=I$, and $\Phi(t,z)={\rm e}^{-\ii\Re E t\sigma_3} M_0(z) {\rm e}^{-\ii w(z) t\sigma_3}$ (recall that $\Phi(t,\la)=\Phi_0(t,0,\la)$~\eqref{Phi_AKNS}, where $\Phi_0(t,x,\la)$ is the background solution of the AKNS system~\eqref{teq},~\eqref{xeq} for which ${\mathcal E}_{\rm bg}(t,0)={\mathcal E}_{\rm in}(t)$ is the input signal~\eqref{MBboundAe}, as shown in Section~\ref{BasicSolAKNS}) satisfies the $t$-equation $\Phi_t=(-\ii z\sigma_3-H(t,0))\Phi$ with $H(t,0)=H_0(t)=\frac{1}{2}
\left(\begin{smallmatrix}0&{\mathcal E}_{\rm in}(t)\\-\overline{{\mathcal E}_{\rm in}(t)}\end{smallmatrix}\right)$~\eqref{H_0}. Thus, using~\eqref{H} and~\eqref{m_lim}, we obtain that $\mathcal E(t,0)=-4\ii\lim_{z\to\infty}(z(M(t,0,z)-I))_{12}= 2(H_0(t))_{12}={\mathcal E}_{\rm in}(t)$ or $H(t,0)=H_0(t)$.

Since $F(t,x,\la)$ satisfies the equation~\eqref{HFeq2}, then just as in the proof of Theorem~\ref{ThRH0} we get~\eqref{F}, i.e.,
\begin{gather*}
F(t,x,\la)= - M(t,x,\la+\ii0) {\rm e}^{-\ii\la t\sigma_3}M^{-1}(0,x,\la+\ii0)\sigma_3 M(0,x,\la+\ii0)\\
\hphantom{F(t,x,\la)=}{}\times {\rm e}^{\ii\la t\sigma_3}M^{-1}(t,x,\la+\ii0), \qquad \la\in\R\setminus\{\Re E\}.
\end{gather*}
The function $F(t,x,\la)$ as well as $M(t,x,\la+\ii0)=M_+(t,x,\la)$ is continuous in $\la\in\R\setminus\{\Re E\}$ and bounded near $\la=\Re E$. Thus, we can appropriately define $M_+(t,x,\Re E)$ (see the proof of Theorem~\ref{ThRH0}) so that $F(t,x,\la)$ is continuous in $\la\in\R$.

The proof of the uniqueness of $\Phi(t,x,z)$ is similar to the proof of the uniqueness of $M(t,x,z)$ in Theorem~\ref{ThRH0}.
 \end{proof}

Theorems~\ref{Ex-Uniq},~\ref{smoothnessRH} and~\ref{solvMBbyM} give the following corollary.
\begin{Corollary}
The function $H(t,x)$~\eqref{H} and the electric field envelope ${\mathcal E}(t,x)$ can also be defined by
\begin{gather}
H(t,x)=-\ii \big[\sigma_3,m^{(2)}(t,x)\big], \nonumber\\
m^{(2)}(t,x)=\frac{1}{2\pi\ii} \int_\Upsilon\big[W^{(2)}(t,x,s)+I\big]\big[J^{(2)}(t,x,s)-I\big]\, \dd s,\label{m2}
\end{gather}
and
\begin{equation}\label{mathcalE_int}
{\mathcal E}(t,x)= \frac{2}{\pi} \int_\Upsilon\big(\big[W^{(2)}(t,x,s)+I\big]\big[I-J^{(2)}(t,x,s)\big]\big)_{12}\, \dd s,
\end{equation}
where $W^{(2)}(t,x,z)$ is a solution of a singular integral equation of the form~\eqref{IntEq2-RH0}, $J^{(2)}(t,x,z)$ has the form~\eqref{case2J2-dif} and~\eqref{case2J2-dif2} for $\Lambda<\infty$ and $\Lambda=\infty$ respectively, and $\Upsilon=L_1\cup L_3 \cup\overline{L_1}\cup\overline{L_3} \cup[\la_1,\la_2]\cup\big[E,\overline{E}\big]$ $($see Figure~$\ref{conturG2-LaFin})$ and $\Upsilon=L_2\cup\overline{L_2}$ $($see Figures~$\ref{conturG2-LaInf1}$ and~$\ref{conturG2-LaInf2})$ for $\Lambda<\infty$ and $\Lambda=\infty$ respectively.
 \end{Corollary}

 \begin{proof}
Using the proof of Theorem~\ref{smoothnessRH}, in particular, the formula~\eqref{sol2_RH0}, we obtain that
\begin{gather*}
M^{(2)}(t,x,z)=I+\frac{m^{(2)}(t,x)}{z} +O\big(z^{-2}\big),\\
M^{(2)}_t(t,x,z)=\frac{m^{(2)}_t(t,x)}{z}+O\big(z^{-2}\big),\qquad z\to\infty,\quad z\in\mC\setminus\Upsilon,
\end{gather*}
where $m^{(2)}(t,x)$ has the form~\eqref{m2}. Since
\[
m(t,x)=\lim_{z\to\infty}z(M(t,x,z)-I)= \lim_{z\to\infty}z\big(M^{(2)}(t,x,z)-I\big)=m^{(2)}(t,x),
\]
 where
\[
M(t,x,z)=M^{(2)}(t,x,z)\big(G^{(2)}(t,x,z)\big)^{-1}\delta^{-\sigma_3}(z)
\]
and
\[
M(t,x,z)=M^{(2)}(t,x,z)\big(G^{(2)}(t,x,z)\big)^{-1}
\]
for $\Lambda<\infty$ and $\Lambda=\infty$ respectively (see the proof of Theorem~\ref{smoothnessRH}), then, using the formula~\eqref{H}, we obtain~\eqref{m2} and~\eqref{mathcalE_int}.
 \end{proof}

\appendix

\section{Appendix. Analysis of the phase function}\label{appendixA}

Consider the phase function $\theta$~\eqref{theta1}:
\begin{gather*}
\theta(t,x,z)=zt-\eta(z)x=
\tau\left(z-\xi\int_{-\Lambda}^\Lambda \frac{n(s)}{s-z}\, \dd s\right),\\
 \xi=\frac{x}{4\tau},\qquad \tau =t-x\ \left(x=\frac{4\xi}{1+4\xi} t\right) ,
\end{gather*}
where $t$ and $x$ as well as $\tau$ and $\xi$ are parameters. In what follows, we use the same notation $\theta$ for the function of the parameters $x$, $t$ or $\tau$, $\xi$.
We will study the function $\Re(\ii\theta(t,x,z))$ to obtain its signature table characterizing the behavior of the exponents ${\rm e}^{\pm 2\ii\theta(t,x,z)}$. It has the form
\begin{gather}
\Re(\ii\theta(t,x,z)) = \nu \left( x-t+\frac{x}{4}\int_{-\Lambda}^\Lambda \frac{n(s)\, \dd s}{|s-z|^2}\right) = \tau\nu (\xi \Pi(\la,\nu)-1) \nonumber \\
\hphantom{\Re(\ii\theta(t,x,z))}{} = \frac{x\nu}{4} \left( \Pi(\la,\nu)-\frac{1}{\xi}\right) , \label{Reitheta}
\end{gather}
where $z=\la+\ii\nu$ and
 \begin{equation}\label{Pi}
\Pi(\la,\nu):=\int_{-\Lambda}^\Lambda \frac{n(s)\, \dd s}{(s-\la)^2+\nu^2}.
 \end{equation}
The function $\theta(t,x,z)$ is analytic for $z\in \mC\setminus[-\Lambda,\Lambda]$ and therefore the functions $\Re(\ii\theta(t,x,z))=-\Im \theta(t,x,z)$ and $\Im(\ii\theta(t,x,z))=\Re \theta(t,x,z)$ are harmonic in $z\in\mC\setminus[-\Lambda,\Lambda]$. It has the continuous boundary values $\theta_\pm(t,x,\la)=\theta(t,x,\la\pm\ii 0)=\la t-\eta_\pm(\la)x$ on $[-\Lambda,\Lambda]$ (recall that the contour $[-\Lambda,\Lambda]$ is oriented from $-\Lambda$ to $\Lambda$), where $\eta_\pm(\la)$ are defined in~\eqref{eta_pm}, $z=\la+\ii\nu$, which can be represented in the form
\[
\theta_\pm(t,x,\la)= \tau\left(\la-\xi {\rm p.v.}\int_{-\Lambda}^\Lambda \frac{n(s)}{s-\la}\, \dd s\right)\mp \frac{\pi\ii}{4}n(\la)x,\qquad \la\in[-\Lambda,\Lambda].
\]

Since the equality $\Re(\ii\theta(t,x,z))=0$ is satisfied if\; $\nu=\Im z=0$\; or
 \begin{equation}\label{LevelLine}
\Pi(\la,\nu) =\frac{1}{\xi},
 \end{equation}
where $\Pi(\la,\nu)$ is defined by~\eqref{Pi}, then the curve defined by the equality $\Re(\ii\theta(t,x,z))=0$ divides the complex plane into two domains $\mC_\pm$ (i.e., this curve is $\nu=0$) when $\xi<0$ ($t<x$) or $t=x$, and into four domains whose common boundary is defined by the equalities $\nu=0$ and~\eqref{LevelLine} when $\xi>0$ ($t>x$).

Since
\begin{equation*}
\frac{\dd \theta}{\dd z}= \tau\left(1-\xi\int_{-\Lambda}^\Lambda \frac{n(s)\, \dd s}{(s-z)^2}\right)= \tau\left(1-\xi\int_{-\Lambda}^\Lambda \frac{n(s)\, \dd s}{(s-\la)^2-2\ii(s-\la)\nu-\nu^2}\right)=0
\end{equation*}
($z\in \mC\setminus[-\Lambda,\Lambda]$) if {\samepage
\begin{align*}
\frac{1}{\xi}& =\int_{-\Lambda}^\Lambda \frac{n(s)\, \dd s}{(s-z)^2}=
\int_{-\Lambda}^\Lambda \frac{n(s)(s-\overline{z})^2 \dd s}{|s-z|^4}\\
& =
\int_{-\Lambda}^\Lambda \frac{n(s)\big[(s-\la)^2-\nu^2\big]\, \dd s}{\big[(s-\la)^2+\nu^2\big]^2}+
2 \ii\nu \int_{-\Lambda}^\Lambda \frac{n(s)(s-\la)\, \dd s}{\big[(s-\la)^2+\nu^2\big]^2}=
I_1(\la,\nu) + 2 \ii \nu I_2(\la,\nu),
\end{align*}
then $\frac{\dd\theta}{\dd z}=0$ if $I_1(\la,\nu)=\frac{1}{\xi}$, and $\nu=0$ or $I_2(\la,\nu)=0$.}

First, consider the case when $\nu=0$ and $I_1(\la,\nu)=\frac{1}{\xi}$. In this case, the derivative $\dd \theta/ \dd z$ takes the form
\begin{equation}\label{dtheta_dz}
\frac{\dd \theta}{\dd z}=\tau \left(1-\xi\int_{-\Lambda}^\Lambda
\frac{n(s)\, \dd s}{(s-\la)^2}\right)=\tau \big(1-\xi I_1(\la,0)\big),\qquad z=\la\notin[-\Lambda,\Lambda],
\end{equation}
where $I_1(\la,0)=\Pi(\la,0)$. Hence (since the derivative~\eqref{dtheta_dz} does not exist for $\la\in[-\Lambda, \Lambda]$), for $\xi>0$ ($\tau>0$) the stationary points of the function $\ii\theta(t,x,z)=\ii\theta(\tau,\xi,z)$ are located on the real line outside the interval $[-\Lambda,\Lambda]$ and they are the points of the intersection of the real line $\nu=0$ and the level line~\eqref{LevelLine}. Since $I_1(\la,0)=\int_{-\Lambda}^\Lambda \frac{n(s)\, \dd s}{(s-\la)^2}>0$, then for $\xi<0$ ($\tau<0$) the derivative $\frac{\dd\theta}{\dd z}<0$ and hence there are no stationary points.

Now, consider the case when ${\nu\ne 0}$ and
\[
I_2(\la,\nu)=\int_{-\Lambda}^\Lambda \frac{n(s)(s-\la)\, \dd s}{\big[(s-\la)^2+\nu^2\big]^2}=0,\qquad
I_1(\la,\nu)=\int_{-\Lambda}^\Lambda \frac{n(s)\big[(s-\la)^2-\nu^2\big]\, \dd s}{\big[(s-\la)^2+\nu^2\big]^2}=\frac{1}{\xi}.
\]
In this case ${\xi>0}$. If $\la\ne 0$, ${\nu\ne 0}$,
then ${I_2(\la,\nu)\ne 0}$. If $\la=0$ and ${\nu\ne 0}$, then
\[
I_2(0,\nu)= \int_{-\Lambda}^\Lambda\frac{n(s)s}{\big[s^2+\nu^2\big]^2}\, \dd s = \left(\int_{-\Lambda}^0+\int_{0}^\Lambda\right) \frac{n(s)s}{\big[s^2+\nu^2\big]^2}\, \dd s=0
\]
 when $n(s)$ is an even function ($n(-s)=n(s)$), for example, when $n(\la)$ has the form~\eqref{Ex.n-const}.
Then the points $z=\ii\nu\ne 0$ will be stationary if
\begin{equation*}
{I_1(0,\nu)=\int_{-\Lambda}^\Lambda \frac{n(s)\big[s^2-\nu^2\big]\, \dd s}{\big[s^2+\nu^2\big]^2}=\frac{1}{\xi}}.
\end{equation*}
Obviously, $\int_{-\Lambda}^\Lambda \frac{n(s)s^2\dd s}{\big[s^2+\nu^2\big]^2}<\int_{-\Lambda}^\Lambda \frac{n(s)\nu^2\dd s}{\big[s^2+\nu^2\big]^2}$, i.e., $I_1(0,\nu)<0$, for $|\nu|>\Lambda$.
Since $n(s)$ is even,
\[
I_1(0,\nu)=2\int_{0}^\Lambda \frac{n(s)\big[s^2-\nu^2\big]\, \dd s}{\big[s^2+\nu^2\big]^2}.
\]
Consider the function $\chi(s):=\frac{s^2}{[s^2+\nu^2]^2}$. Since $\chi'(s)=\frac{2s(\nu^2-s^2)}{[s^2+\nu^2]^3}=0$ for $s=0$ and $s^2=\nu^2$, and $\chi''(s)=\frac{2\nu^4-16\nu^2s^2+6s^4}{[s^2+\nu^2]^4}$ is positive for $s=0$ and negative for $s^2=\nu^2$ ($\nu\ne 0$), then $\chi(s)$ attains a strict maximum at $s=\pm\nu$. Consequently, $\frac{n(s)s^2}{[s^2+\nu^2]^2}< \frac{n(s)\nu^2}{[s^2+\nu^2]^2}$, $s^2\ne\nu^2$ ($\nu\ne 0$, $|\nu|\le \Lambda$), and hence $\int_{0}^\Lambda \frac{n(s)s^2\dd s}{[s^2+\nu^2]^2}< \int_{0}^\Lambda \frac{n(s)\nu^2\dd s}{[s^2+\nu^2]^2}$, i.e., $I_1(0,\nu)<0$, for $|\nu|\le \Lambda$, $\nu\ne 0$. In particular, for $n(\la)$ of the form~\eqref{Ex.n-const} $I_1(0,\nu)=-\frac{1}{\Lambda^2+\nu^2}<0$ for all $\nu$.
Thus, in the present case the function $\ii\theta(t,x,z)$ does not have stationary points.

Since $n(\la)\ge 0$ and $\int_{-\infty}^\infty n(\lambda)\, \dd\lambda=1$, then
\[
 0<\Pi(\la,\nu)\le\sup_{s\in [-\Lambda,\Lambda]} \frac{1}{(s-\la)^2+\nu^2}\int_{-\Lambda}^\Lambda n(s)\, \dd s = \frac{1}{\inf\limits_{s\in [-\Lambda,\Lambda]}(s-\la)^2+\nu^2}.
\]
Thus,
\[
\frac{1}{\xi}= \Pi(\la,\nu)\le \begin{cases}
 \displaystyle \frac{1}{\nu^2}, & |\la|\le \Lambda, \\ \displaystyle\frac{1}{(\Lambda-|\la|)^2+\nu^2}, & |\la|> \Lambda. \end{cases}
\]
Hence, $|\nu|\le \sqrt{\xi}$ for $\la\in \R$ and $(|\la|-\Lambda)^2+\nu^2\le \xi$ when $\la\in (-\infty,-\Lambda)\cup(\Lambda,\infty)$. Since $\Pi(\la,\nu)=\Pi(\la,-\nu)$ for any $\nu$, the function $\Pi(\la,\nu)$ and consequently the level line~\eqref{LevelLine} are symmetric with respect to the real axis $\nu=0$. If $n(-\la)=n(\la)$ for all $\la$ ($n(\la)$ is an even function), then $\Pi(\la,\nu)$ and the level line~\eqref{LevelLine} are also symmetric with respect to the imaginary axis $\la=0$.
Moreover,
\[
 \frac{1}{\xi}=\Pi(\la,\nu)\ge\inf_{s\in [-\Lambda,\Lambda]} \frac{1}{(s-\la)^2+\nu^2} = \frac{1}{(\Lambda+|\la|)^2+\nu^2}\ge \frac{1}{2\big(\Lambda^2+\la^2\big)+\nu^2},
\]
and, therefore, $\xi\le 2\big(\Lambda^2+\la^2\big)+\nu^2$.

Note that
$\lim_{\nu\to0}\Pi(\la,\nu)= \lim_{\nu\to0}\int_{-\Lambda}^\Lambda \frac{n(s)\, \dd s}{(s-\la)^2+\nu^2}= +\infty$ for $\la \in [-\Lambda,\Lambda]$. Taking into account the signatures of the derivatives
\[
\frac{\partial\Pi}{\partial\la}= 2\int_{-\Lambda}^\Lambda
\frac{n(s)(s-\la)\, \dd s}{\big[(s-\la)^2+\nu^2\big]^2},\qquad \frac{\partial\Pi}{\partial\nu}=-2\nu \int_{-\Lambda}^\Lambda
\frac{n(s)\, \dd s}{\big[(s-\la)^2+\nu^2\big]^2},
\]
we conclude that when $\Lambda <\infty$, for each $\xi\in\R_+$ there exists the smooth and closed level line~\eqref{LevelLine}, which is symmetric with respect to the real axis $\nu=0$, and we denote this curve (level line) by $\gamma_\xi$ (see Figure~\ref{ReiTheta1}). Since the derivative $\frac{\partial\widetilde{\Pi}}{\partial\la}=\frac{\partial\Pi}{\partial\la}$, where $\widetilde{\Pi}(\la,\nu,\xi)=\Pi(\la,\nu)-\frac{1}{\xi}$,\; $\xi\in\R_+$, exists and is not equal to zero for $\nu\in \R$ and $|\la|>\Lambda$, then in some neighborhoods of the stationary points (recall that they are the points of the intersection of the real line $\nu=0$ and the level line~\eqref{LevelLine}) the level line can be determined by the smooth function $\la=\la_-(\nu,\xi)$ for $\la<-\Lambda$ and $\la=\la_+(\nu,\xi)$ for $\la>\Lambda$. We denote the stationary points by $\la_\pm=\la_\pm(\xi)$ ($\la_-<-\Lambda$ and $\la_+>\Lambda$). At a large $\xi$, the level line~\eqref{LevelLine} goes to the vicinity of the infinity in the $z$-plane, and at a small $\xi$, it envelops the segment $[-\Lambda, \Lambda]$ ($|\nu|\to +0$ and $|\la|\to \Lambda+0$ for $\xi\to +0$). As $\Lambda\to+\infty$ the stationary points tend to $\pm\infty$.
When $\Lambda=+\infty$ the real line $\nu=0$ becomes an asymptote for the level line~\eqref{LevelLine} ($\la_\pm(\nu,\xi)\to \pm\infty$ as $\nu\to 0$), which consists of the two curves~$\gamma^\pm_\xi$ symmetric with respect to the real axis (see Figure~\ref{ReiTheta2}).
As $\nu\to\pm 0$ the curves $\gamma^\pm_\xi$ tend to the real line from $\mC_\pm$ respectively.

Taking into account~\eqref{Reitheta}, we obtain that
\[
\sgn[\Re(\ii\theta(t,x,z))]=\sgn\Im z
\]
for $0\le t<x$, i.e., $\xi<0$, and for $t=x$, $t,x>0$, and $\Re(\ii\theta(0,0,z))\equiv 0$. This case ($t\le x$, $t,x\ge 0$) is studied in Theorem~\ref{causality-region}.

As shown above, in the case $\tau=t-x>0$ and, accordingly, $\xi>0$, the curve defined by the equality $\Re(\ii\theta(t,x,z))=0$, where $z=\la+\ii\nu$ and $\Lambda <\infty$, consists of the real line $\R$ and the closed level line $\gamma_\xi$ defined by the equality $\Pi(\la,\nu)=1/\xi$ (the level line~\eqref{LevelLine}). For any fixed $\xi>0$, the interval $[-\Lambda,\Lambda]$ is located inside the oval $\gamma_\xi$, and $\gamma_\xi$ intersects the real line at the stationary points $\la_\pm=\la_\pm(\xi)$ ($\la_-<-\Lambda$ and $\la_+>\Lambda$) which are simple ones because $\frac{\partial^2\theta}{\partial\la^2}= -2\xi \int_{-\Lambda}^\Lambda \frac{n(s)\, \dd s}{(s-\la)^3}$ is negative (i.e., strictly negative) for $\la<-\Lambda$ and positive for $\la>\Lambda$. When $\Lambda\to+\infty$ ($\xi>0$ is fixed) the stationary points $\la_\pm$ tend to $\pm\infty$. For $\tau>0$ ($\xi>0$) and $\Lambda=+\infty$, the real line $\nu=0$ is an asymptote for the level line~\eqref{LevelLine} consisting of the two curves~$\gamma^\pm_\xi$, symmetric with respect to the real line.

It is easy to verify that for $t>x$ ($\xi>0$) the signature table of $\Re(\ii\theta(t,x,z))$~\eqref{Reitheta} has the form presented in Figures~\ref{ReiTheta1} and~\ref{ReiTheta2}. The signature table shows that the exponents ${\rm e}^{-2\ii\theta(t,x,z)}$ and ${\rm e}^{2\ii\theta(t,x,z)}$ vanish exponentially on the subintervals $(E,\Re E)$ and $\big(\Re E,\overline{E}\big)$ respectively when $z$ and the interval $\big(E,\overline{E}\big)$ lie inside of $\gamma_\xi$ (when $\Lambda<\infty$) or between the curves $\gamma^\pm_\xi$ (when $\Lambda=\infty$).
They increase in $\tau$ unboundedly on the subintervals $(E,\Re E)$ and $\big(\Re E,\overline{E}\big)$ respectively if $z$ and the interval $\big(E,\overline{E}\big)$ lie outside of $\gamma_\xi$ (when $\Lambda<\infty$) or on those parts of the subintervals $(E,\Re E)$ and $\big(\Re E,\overline{E}\big)$ respectively that lie above the curve $\gamma^+_\xi$ and below the curve $\gamma^-_\xi$ if $z$ also lies in the corresponding regions (when $\Lambda=\infty$).
\begin{figure}[t] \centering
 \begin{tikzpicture}[smooth,scale=1,thick,font=\small]
\draw[black] (-4.5,0) -- (4.5,0);
\draw[color=black, fill=black](-2.25,0) circle(0.05);
\draw[color=black, fill=black](2.25,0) circle(0.05);
\draw[color=black, fill=black](-3.27,0) circle(0.05);
\draw[color=black, fill=black](3.27,0) circle(0.05);
\draw[very thick, black] (0,0) ellipse [x radius=93pt, y radius=48.5pt];
\node[black] at (4.3,0.2) {$\mathbb{R}$};
\node[black] at (0,0.8) {$+$};
\node[black] at (0,2) {$-$};
\node[black] at (0,-0.8) {$-$};
\node[black] at (0,-2) {$+$};
\node[black] at (-2.27,-0.3) {$-\Lambda$};
\node[black] at (2.25,-0.3) {$\Lambda$};
\node[black] at (-3.85,-0.3) {$\lambda_{-}(\xi)$};
\node[black] at (3.85,-0.3) {$\lambda_{+}(\xi)$};
\node[black] at (1.25,1.78) {$\gamma_{\xi}$};
 \end{tikzpicture}
 \caption{The curve $\R\cup\gamma_\xi$ defined by $\Re(\ii\theta(t,x,z))=0$ and the signature table of $\Re (\ii \theta)$ when $t>x$ ($\xi>0$) and $\Lambda<\infty$.}\label{ReiTheta1}
\end{figure}
\begin{figure}[t]\centering
 \begin{tikzpicture}[smooth,scale=1,thick,font=\small]
\draw[black] (-4.5,0) -- (4.5,0);
\draw[black,very thick] (-4.5,0.15) .. controls (-1.3,1.45) and (1.3,1.45).. (4.5,0.15);
\draw[black,very thick] (-4.5,-0.15) .. controls (-1.3,-1.45) and (1.3,-1.45).. (4.5,-0.15);
\node[black] at (2.3,0.25) {$\mathbb{R}$};
\node[black] at (0,0.6) {$+$};
\node[black] at (0,1.45) {$-$};
\node[black] at (0,-0.6) {$-$};
\node[black] at (0,-1.45) {$+$};
\node[black] at (1.5,1.3) {$\gamma_{\xi}^{+}$};
\node[black] at (1.5,-1.3) {$\gamma_{\xi}^{-}$};
 \end{tikzpicture}
 \caption{The curve $\R\cup \gamma^+_\xi\cup \gamma^-_\xi$ defined by $\Re(\ii\theta(t,x,z))=0$ and the signature table of $\Re (\ii \theta)$ when $t>x$ ($\xi>0$) and $\Lambda=\infty$.}\label{ReiTheta2}
\end{figure}

\subsection*{Acknowledgements}

The author would like to thank Vladimir Kotlyarov (B.~Verkin ILTPE of NAS of Ukraine) for useful discussions and suggestions. Also, the author would like to thank the anonymous referees for careful reading of the manuscript and their comments.
This work was partially supported by the National Academy of Sciences of Ukraine (project 0121U111968 ``Nonstandard nonlocal and peakon integrable equations: asymptotics and inverse scattering transform'') and the Alexander von Humboldt Foundation (the host institution: Friedrich-Alexander University of Erlangen-Nuremberg, Chair for Dynamics, Control, Machine Learning and Numerics).

\pdfbookmark[1]{References}{ref}
\LastPageEnding

\end{document}